\documentclass[12pt]{amsart}

\usepackage[top=1in, bottom=1in, left=1in, right=1in]{geometry}
\usepackage{times}

\usepackage{amssymb,amsmath,amsthm}
\usepackage{mathrsfs} 
\usepackage{enumitem}
\usepackage[hyphens]{url}
\usepackage{dsfont}
\usepackage{graphicx,tikz}

\usepackage[colorlinks=true, linkcolor=blue, citecolor=green, urlcolor=purple]{hyperref}   

\usepackage{graphicx}
\usepackage{color}

\numberwithin{equation}{section}

\newtheorem{thm}{Theorem}
\newtheorem{lem}{Lemma}[section]
\newtheorem{prop}[lem]{Proposition}
\newtheorem{cor}[thm]{Corollary}

\newtheorem*{ModelProblem}{Model Problem}
\newtheorem*{ModelProblem*}{Model Problem*}

\newtheorem*{kthm}{Khinchin's theorem}

\theoremstyle{definition}
\newtheorem{dfn}[lem]{Definition}

\theoremstyle{remark}
\newtheorem*{rem}{Remark}

\newcommand{\N}{\mathbb{N}}

\newcommand{\Z}{\mathbb{Z}}
\newcommand{\Q}{\mathbb{Q}}
\newcommand{\R}{\mathbb{R}}

\newcommand{\CA}{\mathcal{A}}

\newcommand{\CB}{\mathcal{B}}
\newcommand{\CC}{\mathcal{C}}

\newcommand{\CE}{\mathcal{E}}

\newcommand{\CH}{\mathcal{H}}

\newcommand{\CK}{\mathcal{K}}
\newcommand{\CL}{\mathcal{L}}

\newcommand{\CN}{\mathcal{N}}

\newcommand{\CP}{\mathcal{P}}

\newcommand{\CR}{\mathcal{R}}
\newcommand{\CS}{\mathcal{S}}

\newcommand{\CV}{\mathcal{V}}
\newcommand{\CW}{\mathcal{W}}
\newcommand{\CX}{\mathcal{X}}

\newcommand{\ds}{\displaystyle}
\newcommand{\bs}\boldsymbol{}

\newcommand{\eq}[1]{ \begin{equation} \begin{split} #1 \end{split} \end{equation} }
\newcommand{\al}[1]{\begin{align} #1 \end{align} }
\newcommand{\als}[1]{\begin{align*} #1 \end{align*} }

\newcommand{\un}{\mathds{1}}

\newcommand{\dee}{\mathrm{d}}

\DeclareMathOperator{\supp}{supp}
\DeclareMathOperator{\lcm}{lcm}

\renewcommand{\tilde}{\widetilde}
\newcommand{\eps}{\varepsilon}
\renewcommand{\phi}{\varphi}

\definecolor{red}{rgb}{1,0,0}
\definecolor{orange}{rgb}{0.7,0.3,0}
\definecolor{blue}{rgb}{.2,.6,.75}
\definecolor{green}{rgb}{.4,.7,.4}

\renewcommand{\le}{\leqslant}
\renewcommand{\ge}{\geqslant}

\begin{document}

\title{On the Duffin-Schaeffer conjecture}
\author{Dimitris Koukoulopoulos}
\address{D\'epartement de math\'ematiques et de statistique\\
Universit\'e de Montr\'eal\\
CP 6128 succ. Centre-Ville\\
Montr\'eal, QC H3C 3J7\\
Canada}
\email{koukoulo@dms.umontreal.ca}

\author{James Maynard}
\address{Mathematical Institute, Radcliffe Observatory quarter, Woodstock Road, Oxford OX2 6GG, England}
\email{james.alexander.maynard@gmail.com}


\date{\today}

\begin{abstract}
Let $\psi:\mathbb{N}\to\mathbb{R}_{\ge0}$ be an arbitrary function from the positive integers to the non-negative reals. Consider the set $\mathcal{A}$ of real numbers $\alpha$ for which there are infinitely many reduced fractions $a/q$ such that $|\alpha-a/q|\le \psi(q)/q$. If $\sum_{q=1}^\infty \psi(q)\phi(q)/q=\infty$, we show that $\mathcal{A}$ has full Lebesgue measure. This answers a question of Duffin and Schaeffer. As a corollary, we also establish a conjecture due to Catlin regarding non-reduced solutions to the inequality $|\alpha- a/q|\le \psi(q)/q$, giving a refinement of Khinchin's Theorem.
\end{abstract}

\maketitle

\section{Introduction}
Let $\psi:\N\to\R_{\ge0}$ be an arbitrary function from the positive integers to the non-negative reals. Given $\alpha\in\R$, we wish to understand when we can find infinitely many integers $a$ and $q$ such that
\begin{equation}\label{eq:rational approximations}
\bigg|\alpha-\frac{a}{q}\bigg| \le \frac{\psi(q)}{q}.
\end{equation}
Clearly, it suffices to restrict our attention to numbers $\alpha\in[0,1]$. 

When $\psi(q)=1/q$ for all $q$, Dirichlet's approximation theorem implies that, given {\it any} irrational $\alpha\in[0,1]$, there are  infinitely many coprime integers $a$ and $q$ satisfying  \eqref{eq:rational approximations}. On the other hand, the situation can become significantly more complicated if $\psi$ behaves more irregularly. Even small variations in the size of $\psi$ can cause \eqref{eq:rational approximations} to have no solutions for certain numbers $\alpha$. However, there are several results in the literature that show that, under rather general conditions on $\psi$, \eqref{eq:rational approximations} has infinitely many solutions for {\it almost all} $\alpha\in[0,1]$, in the sense that the residual set has null Lebesgue measure. 

The prototypical such `metric' result was proven by Khinchin in 1924 \cite{Khinchin-paper} (see also \cite[Theorem 32]{Khinchin-book}). To state his result, we let $\lambda$ denote the Lebesgue measure on $\R$.

\begin{kthm}
Consider a function $\psi:\N\to[0,+\infty)$ such that the sequence $(q\psi(q))_{q=1}^\infty$ is decreasing, and let $\CK$ denote the set of real numbers $\alpha\in[0,1]$ for which \eqref{eq:rational approximations} has infinitely many solutions $(a,q)\in\Z^2$ with $0\le a\le q$. Then the following hold:
\begin{enumerate}
	\item If $\sum_{q\ge1} \psi(q)<\infty$, then $\lambda(\CK)=0$.
	\item  If $\sum_{q\ge1} \psi(q)=\infty$, then $\lambda(\CK)=1$.
\end{enumerate}
\end{kthm}
There is an intuitive way to explain why Khinchin's result ought to be true. Consider the sets
\eq{\label{def:K_q}
\CK_q= [0,1]\cap \bigcup_{a=0}^q   \Big[\frac{a-\psi(q)}{q}, \frac{a+\psi(q)}{q}\Big] ,
}
so that\footnote{Recall that if $X_1,X_2,\dots$ is a sequence of sets of real numbers, then $\limsup_{n\to\infty}X_n$ denotes the set of real numbers lying in infinitely many $X_n$'s.}
\[
\CK
	=\limsup_{q\to\infty}\CK_q.
\]
In addition,
\[
\min\{\psi(q),1/2\}\le \lambda(\CK_q)\le 2\min\{\psi(q),1/2\}.
\]
Thus, part (a) of Khinchin's theorem is an immediate corollary of the `easy' direction of the Borel-Cantelli lemma from Probability Theory \cite[Lemma 1.2]{harman} applied to the probability space $[0,1]$ equipped with the measure $\lambda$. If we knew, in addition, that the sets $\CK_q$ were mutually independent, then we could apply the `hard' direction of the Borel-Cantelli lemma \cite[Lemma 1.3]{harman} to deduce part (b) of Khinchin's theorem. Of course, the sets $\CK_q$ are not mutually independent, so the difficulty in Khinchin's proof is showing that there is enough `approximate independence', so that $\CK$ still has full measure.

\medskip

In 1941, Duffin and Schaeffer \cite{DS} undertook a study of the limitations to the validity of Khinchin's theorem, since the condition that $q\psi(q)$ is decreasing is not a necessary condition. They discovered that it is more natural to focus on reduced solutions $a/q$ to \eqref{eq:rational approximations} that avoid overcounting issues arising when working with arbitrary fractions $a/q$. To this end, let 
\begin{equation}\label{eq:A_q-dfn}
\CA_q:= [0,1]\cap \bigcup_{\substack{1\le a\le q \\ \gcd(a,q)=1}}  \Big[\frac{a-\psi(q)}{q}, \frac{a+\psi(q)}{q}\Big] .
\end{equation}
and
\begin{equation}\label{eq:A-dfn}
\CA:= \limsup_{q\to\infty} \CA_q.
\end{equation}
Just like before, using the `easy' direction of the Borel-Cantelli lemma, we immediately find that
\begin{equation}\label{DS-BC}
\sum_{q=1}^\infty \frac{\phi(q)\psi(q)}{q}<\infty
\qquad\implies\qquad \lambda(\CA)=0.
\end{equation}
In analogy to Khinchin's result, Duffin and Schaeffer  conjectured that we also have the implication
\begin{equation}\label{DS-conj}
\sum_{q=1}^\infty \frac{\phi(q)\psi(q)}{q}=\infty
\qquad\implies\qquad \lambda(\CA)=1.
\end{equation}
This is listed as Problem 46 in Montgomery's lectures \cite[Page 204]{montgomery}.

The main result of the present paper is a proof of the Duffin-Schaeffer conjecture:

\begin{thm}\label{thm:MainThm}
Let $\psi:\N\rightarrow \R_{\ge 0}$ be a function such that
\[
\sum_{q=1}^\infty \frac{\psi(q)\phi(q)}{q}=\infty.
\]
Let $\CA$ be the set of $\alpha\in[0,1]$ for which the inequality
\eq{
	\label{eq:rational approximations again}
\bigg|\alpha-\frac{a}{q}\bigg| \le \frac{\psi(q)}{q}
}
has infinitely many coprime solutions $a$ and $q$. 
Then $\CA$ has Lebesgue measure 1. 
\end{thm}

As a direct corollary, we obtain Catlin's conjecture \cite{catlin} that deals with solutions to \eqref{eq:rational approximations again} where the approximations are not necessarily reduced fractions, giving an extension of Khinchin's Theorem.

\begin{thm}\label{thm:catlin}
	Let $\psi:\N\rightarrow \R_{\ge 0}$ and let $\CK$ denote the set of $\alpha\in[0,1]$ for which the inequality \eqref{eq:rational approximations again} has infinitely many solutions $(a,q)\in\Z^2$ with $0\le a\le q$. Define $\psi^*:\N\rightarrow\R_{\ge 0}$ by
	\[
	\psi^*(q):=\phi(q)\sup\{\psi(n)/n \,:\, n\in\N,\ q|n\} .
	\]
	Then the following hold:
	\begin{enumerate}[topsep=2pt, itemsep=2pt]
		\item If $\sum_{q=1}^\infty \psi^*(q)<\infty$, then $\lambda(\CK)=0$.
		\item If $\sum_{q=1}^\infty \psi^*(q)=\infty$, then $\lambda(\CK)=1$.
	\end{enumerate}
\end{thm}

\medskip

There has been much partial progress on the Duffin-Schaeffer conjecture in previous work. The assumption that the sequence $(q\psi(q))_{q=1}^\infty$ is decreasing implies that $(\psi(q)/q)_{q=1}^\infty$ is also decreasing. In particular, if $a/q$ is a fraction satisfying \eqref{eq:rational approximations}, then so is its reduction $a_1/q_1$. Thus, as observed by Walfisz \cite{walfisz} (in work predating Duffin and Schaeffer's conjecture), Khinchin's Theorem implies the Duffin-Schaeffer conjecture when $q\psi(q)$ is decreasing. In the same paper, he strengthened part (b) of Khinchin's theorem as follows: if $\sum_{q\ge1}\psi(q)=\infty$ and $\psi(q)\ll\psi(2q)$ for all $q\in\N$, then the set of $\alpha\in[0,1]$ for which \eqref{eq:rational approximations} has infinitely many coprime solutions $a$ and $q$ has Lebesgue measure 1.

Duffin and Schaeffer \cite{DS} had already established their conjecture \eqref{DS-conj} when $\psi$ is sufficiently `regular', in the sense that the function $\phi(q)/q$ behaves like the constant function 1 when weighted with $\psi$. More precisely, they proved \eqref{DS-conj} under the assumption that
\[
\limsup_{Q\to\infty} \frac{\sum_{q\le Q}\psi(q)\phi(q)/q}{\sum_{q\le Q}\psi(q)}>0 .
\]
Since then, a variety of results towards the Duffin-Schaeffer conjecture have been proven. The first significant step was achieved by Erd\H os \cite{erdos} and then improved by Vaaler \cite{vaaler}, who demonstrated \eqref{DS-conj} when $\psi(q)=O(1/q)$. In addition, Pollington and Vaughan \cite{PV} proved that the $d$-dimensional analogue of the Duffin-Schaeffer conjecture holds for any $d\ge2$. 

The proof of all three aforementioned results can be found in Harman's book \cite{harman} (see Theorems 2.5, 2.6 and  3.6, respectively), along with various other cases of the Duffin-Schaeffer conjecture (see Theorems 2.9, 2.10, 3.7 and 3.8).

More recently, the focus shifted towards establishing variations of \eqref{DS-conj}, where the assumption that the series $\sum_{q\ge1}\psi(q)\phi(q)/q$ diverges is replaced by a slightly stronger assumption. The first result of this kind was proven in 2006 by Haynes,  Pollington and Velani \cite{extra-div1}, and was improved in 2013 by Beresnevich, Harman, Haynes and Velani \cite{extra-div2}. The strongest published such result is the recent theorem of Aistleitner, Lachmann, Munsch, Technau and Zafeiropoulos \cite{aistleitner} who showed that 
\eq{\label{extra divergence}
\sum_{q=1}^\infty \frac{\phi(q)\psi(q)}{q(\log{q})^\eps}=\infty
\qquad\implies\qquad \lambda(\CA)=1,
}
for any fixed $\eps>0$. In 2014, Aistleitner \cite{aistleiner} established a companion result to the above one: he showed that if $\sum_{q=1}^\infty\psi(q)\phi(q)/q$ diverges and $\psi$ is not `too concentrated', in the sense that
\eq{\label{slow divergence}
\sum_{ 2^{2^j}<q\le 2^{2^{j+1}} } \frac{\psi(q)\phi(q)}{q}=O(1/j)\qquad \text{for all}\quad j\ge1 ,
}
then $\lambda(\CA)=1$. 

\begin{rem} In the recent progress report \cite{aistleitner2}, Aistleitner explains how to improve on \eqref{extra divergence} and \eqref{slow divergence}. In particular, his refined arguments allow him to replace $(\log q)^\eps$ by $(\log\log q)^\eps$ in \eqref{extra divergence}.
\end{rem} 

Finally, Beresnevich and Velani \cite{hausdorff DS} have proven that the Duffin-Schaeffer conjecture implies a Hausdorff measure version of itself. An immediate corollary of their results when combined with Theorem \ref{thm:MainThm} is the following.
\begin{cor}
Let $\psi:\N\rightarrow [0,1/2]$. Write $\CA$ for the set of $\alpha\in[0,1]$ such that \eqref{eq:rational approximations again} has infinitely many coprime solutions $a$ and $q$, and set
\[
s=\inf\Bigl\{\beta\in\mathbb{R}_{\ge 0}:\,\sum_{q=1}^\infty \phi(q)(\psi(q)/q)^\beta<\infty\Bigl\}.
\]
Then the Hausdorff dimension $\dim_{\CH}(\CA)$ of $\CA$ satisfies
\[
\dim_{\CH}(\CA)=\min(s,1).
\]
\end{cor}

\medskip

The proof of Theorem \ref{thm:catlin}, assuming Theorem \ref{thm:MainThm}, is explained in Section \ref{sec:catlin}. For an outline of the proof of Theorem \ref{thm:MainThm}, we refer the readers to Section \ref{sec:outline}. Finally, the structure of the rest of the paper is presented in Section \ref{sec:structure}.

\subsection*{Notation} The letter $\mu$ will always denote a generic measure on $\N$. We reserve the letter $\lambda$ for the Lebesgue measure on $\R$.

Sets will be typically denoted by capital calligraphic letters such as $\CA,\CV$ and $\CE$. A triple $G=(\CV,\CW,\CE)$ denotes a bipartite graph with vertex sets $\CV$ and $\CW$ and edge set $\CE\subseteq\CV\times\CW$. 

Given a set or an event $\CE$, we let $\un_\CE$ denote its indicator function.

The letter $p$ will always denote a prime number. We also write $p^k\|n$ to mean that $p^k$ is the exact power of $p$ dividing the integer $n$. 

When we write $(a,b)$, we mean the pair of $a$ and $b$. In contrast, we write $\gcd(a,b)$ for the greatest common divisor of the integers $a$ and $b$ and $\lcm(a,b)$ for the least common multiple of $a$ and $b$. 

Finally, we adopt the usual asymptotic notation of Vinogradov: given two functions $f,g:X\to\R$ and a set $Y\subseteq X$, we write ``$f(x)\ll g(x)$ for all $x\in Y$'' if there is a constant $c=c(f,g,Y)>0$ such that $|f(x)|\le cg(x)$ for all $x\in Y$. The constant is absolute unless otherwise noted by the presence of a subscript. If $h:X\to\R$ is a third function, we use Landau's notation ``$f=g+O(h)$ on $Y$'' to mean that $|f-g|\ll h$ on $Y$. Typically the set $Y$ is clear from the context and so not stated explicitly.

We introduce several new quantities and associated notation in Section \ref{sec:Graph} which are tailored to our application. In the interest of concreteness we have decided to use explicit constants in several parts of the argument, but we encourage the reader not to concern themselves with numerics on a first reading.

\subsection*{Acknowledgements}
First and foremost, we would like to thank Sam Chow, Leo Goldmakher and Andrew Pollington for their valuable insights to this project: we have had extended discussions with them on various aspects of the Duffin-Schaeffer conjecture and are indebted to them for their contributions. In addition, we would like to thank Sam Chow for pointing out the connection of our paper to Catlin's conjecture and the construction of the counterexample given in Section \ref{sec:Concluding}, and Sanju Velani for introducing J.M. to this problem. Finally, we are grateful to Christopher Aistleitner, Ben Green, Alan Haynes and Sam Chow for sending us various comments and corrections on an earlier version of our paper, as well as to the anonymous referees of the paper for their very detailed comments.

Our project began in the Spring of 2017 during our visit to the Mathematical Sciences Research Institute in Berkeley, California (supported by the National Science Foundation under Grant No.~DMS-1440140). 
In addition, a significant part of our work took place during two visits of J.M. to the Centre de recherche math\'ematiques in Montr\'eal in November 2017 and May 2018, and during the visit of D.K. to the University of Oxford in the Spring of 2019 (supported by Ben Green's Simons Investigator Grant 376201). We would like to thank our hosts for their support and hospitality.

D.K. was also supported by the Natural Sciences and Engineering Research Council of Canada (Discovery Grant  2018-05699) and by the Fonds de recherche du Qu\'ebec - Nature et technologies (projet de recherche en \'equipe - 256442). J.M. was also supported by a Clay Research Fellowship during the first half of this project, and this project has received funding from the European Research Council (ERC) under the European Union's Horizon 2020 research and innovation programme (grant agreement No 851318) for the later stages.

\section{Deduction of Theorem \ref{thm:catlin} from Theorem \ref{thm:MainThm}}\label{sec:catlin}

Most of the details of this deduction can be found in Catlin's original paper \cite{catlin}. We give them here as well for the sake of completeness. For easy reference, let
\[
S=\sum_{q=1}^\infty \phi(q)\sup_{\substack{n\in\N\\ q|n}} \frac{\psi(n)}{n} .
\]
Firstly, we deal with a rather trivial case.

\bigskip

\noindent
{\it Case 1:} There is a sequence of integers $q_1<q_2<\cdots$ such that $\psi(q_i)\ge1/2$ for all $i$. 

\medskip

By passing to a subsequence if necessary, we may assume that $q_{i+1}\ge 2q_i^2$ for all $i$. Recall the definition of the set $\CK_q$ from  \eqref{def:K_q}. Since $\psi(q_i)\ge1/2$, we infer that $\CK_{q_i}=[0,1]$ for each $i$. As a consequence, $\CK=[0,1]$. We claim that we also have $S=\infty$. Indeed, for each $d|q_i$, we have 
\[
\sup_{\substack{n\in\N\\ d|n}} \frac{\psi(n)}{n}\ge \frac{\psi(q_i)}{q_i}\ge\frac{1}{2q_i}.
\]
Consequently,
\eq{\label{eq:catlin1}
\sum_{q_{i-1}<q\le q_i} \phi(q)\sup_{\substack{n\in\N\\ q|n}} \frac{\psi(n)}{n}
\ge \sum_{\substack{q_{i-1}<q\le q_i\\ q|q_i}} \frac{\phi(q)}{2q_i}\ge  \frac{1}{2q_i}\sum_{q|q_i}\phi(q)-\frac{1}{2q_i}\sum_{q\le q_{i-1}}\phi(q)\ge \frac{1}{4},
}
since $\sum_{q|q_i}\phi(q)=q_i$ and $\sum_{q\le q_{i-1}}\phi(q)\le q_{i-1}^2\le q_i/2$. Summing \eqref{eq:catlin1} over all $i\ge2$ proves our claim that $S=\infty$.

Hence, if we are in Case 1, we see that $S=\infty$ and $\CK=[0,1]$, so that Theorem \ref{thm:catlin} holds.

\bigskip

\noindent
{\it Case 2:} There are finitely many $q\in\N$ with $\psi(q)\ge1/2$.

\medskip

Note that in this case replacing $\psi$ by $\min\{\psi,1/2\}$ does not affect either the convergence of $S$, nor which numbers lie in the set $\CK=\limsup_{q\to\infty}\CK_q$. Hence, we may assume without loss of generality that $\psi\le1/2$. In particular, we have that $\lim_{n\to\infty}\psi(n)/n=0$, so that we may replace $\sup$ by $\max$ in the definition of $S$. We now follow an argument due to Catlin.

Consider the function $\xi$ defined by 
\[
\frac{\xi(q)}{q}=\max_{\substack{n\in\N\\ q|n}} \frac{\psi(n)}{n}
\]
and the sets
\[
\CC_q= [0,1]\cap \bigcup_{\substack{1\le a\le q \\ \gcd(a,q)=1}}  \Big[\frac{a-\xi(q)}{q}, \frac{a+\xi(q)}{q}\Big] 
\qquad\text{and}\qquad 
\CC:= \limsup_{q\to\infty} \CC_q .
\]
These are the analogues of the sets $\CA_q$ and $\CA$ that appear in Theorem \ref{thm:MainThm}, but with $\xi$ in place of $\psi$. We claim that
\eq{\label{eq:catlin2}
\CC\setminus\Q=\CK\setminus\Q.
}
This will immediately complete the proof of Theorem \ref{thm:catlin}(b) by applying Theorem \ref{thm:MainThm}. In addition, Theorem \ref{thm:catlin}(a) will follow from \eqref{DS-BC}.

Indeed, if $\alpha \in \CC\setminus\Q$, then there are infinitely many reduced fractions $a_j/q_j$ such $|\alpha-a_j/q_j|\le\xi(q_j)/q_j$. By the definition of $\xi$, there is some $n_j$ that is a multiple of $q_j$ such that $\xi(q_j)/q_j=\psi(n_j)/n_j$. If we let $m_j=a_j n_j/q_j$, then $|\alpha-m_j/n_j|\le \psi(n_j)/n_j$ for all $j$. Since $\lim_{j\to\infty}q_j=\infty$ and $n_j\ge q_j$ for each $j$, we also have that $\lim_{j\to\infty}n_j=\infty$,  whence $\alpha \in \CK$. 

Conversely, let $\alpha \in \CK\setminus\Q$. Then there are infinitely many pairs $(m_j,n_j)\in\N^2$ such that $|\alpha-m_j/n_j|\le\psi(n_j)/n_j$. If we let $a_j/q_j$ be the fraction $m_j/n_j$ in reduced form, we also have that $|\alpha-a_j/q_j|\le\psi(n_j)/n_j\le \xi(q_j)/q_j$, where the last inequality follows by noticing that $q_j|n_j$. This shows that $\alpha\in \CC$, as long as we can show that infinitely many of the fractions $a_j/q_j$ are distinct. But if this were not the case, there would exist a fraction $a/q$ such that $a_j/q_j=a/q$ for infinitely many $j$, so that $|\alpha-a/q|\le \psi(n_j)/n_j\le1/(2n_j)$ for all such $j$. Letting $j\to\infty$, we find that $\alpha=a/q\in\Q$, a contradiction. 

This completes the proof of \eqref{eq:catlin2}, and hence of Theorem \ref{thm:catlin} in all cases.

\section{Outline of the proof of Theorem \ref{thm:MainThm}}\label{sec:outline}

The purpose of this section is to explain in rough terms the main ideas that go into the proof of our main result. To simplify various technicalities, let us consider the special case where the function $\psi$ satisfies the following conditions:
\begin{enumerate}
\item $\psi(q)=0$ or $\psi(q)=q^{-c}$ for every $q\in\mathbb{N}$\,;
\item $\psi$ is non-zero only on square-free integers $q$\,;
\item There exists an infinite sequence $2<x_1<x_2<\dots$ such that:
\begin{enumerate}
\item $x_j>x_{j-1}^2$\,;
\item $\psi$ is supported on $\cup_{i=1}^{\infty}[x_i,2x_i]$\,;
\item for each $i$ we have
\[
\sum_{q\in [x_i,2x_i]}\frac{\phi(q)}{q}\psi(q)\in[1,2]  \,.
\]
\end{enumerate}
\end{enumerate}
In this set-up, it follows from a well-known second moment argument (which will be explained in detail in Section \ref{sec:Prelim}) that to establish the Duffin-Schaeffer conjecture it is sufficient to show that for any $x\in \{x_1,\,x_2,\dots\}$ we have
\[
\sum_{\substack{q,r\in\CS\\\gcd(q,r)\le M(q,r)}}\frac{\phi(q)}{q}\cdot \frac{\phi(r)}{r}\cdot P(q,r)\ll x^{2c},
\]
where 
\begin{align*}
\CS&:=\{q\in\Z\cap[x,2x] : \psi(q)\neq0\},\\
M(q,r)&:=\max\{q\psi(r),r\psi(q)\}\asymp x^{1-c},\\
P(q,r)&:=\prod_{\substack{p|qr/\gcd(q,r)^2\\ p>M(q,r)/\gcd(q,r)}}\Bigl(1+\frac{1}{p}\Bigr).
\end{align*}
Note that we have the estimate
\[
\sum_{q\in \CS}\frac{\phi(q)}{q}\asymp x^{c}\sum_{q\in\CS}\frac{\phi(q)}{q}\psi(q)\asymp x^{c} ,
\]
so the key to the proof is to show that $P(q,r)\ll 1$ on average over $q,r\in\CS$. This would then show suitable `approximate independence' of the sets $\CA_q$ defined by  \eqref{eq:A_q-dfn}. The size of $P(q,r)$ is controlled by small primes dividing exactly one of $q,r$. With this in mind, let us consider separately the contribution from $q,r$ with
\begin{equation}
\sum_{\substack{p|qr/\gcd(q,r)^2\\ p\ge t}}\frac{1}{p}\approx 1
\label{eq:tThreshold}
\end{equation}
for different thresholds $t$ (which we think of as small compared with $x$). A calculation then shows that it is sufficient to show that for each $t\ge1$
\begin{equation}
\sum_{\substack{q,r\in\CS \\ \gcd(q,r)\ge x^{1-c}/t}}\frac{\phi(q)}{q}\cdot  \frac{\phi(r)}{r}\ll \frac{x^{2c}}{t}.
\label{eq:jSum}
\end{equation}
In particular, we need to understand the structure of a set $\CS$ where many of the pairs $(q,r)\in\CS^2$ have a large common factor. There are $O(x^c)$ choices of $q\in \CS$ weighted by $\phi(q)/q$. Given $q\in \CS$, there are $x^{o(1)}$ divisors of $q$ that are at least $x^{1-c}/t$. In turn, given such a divisor $d$, there are $O(x^c t)$ integers $r\in[x,2x]$ which are a multiple of $d$ (forgetting the constraint $r\in\CS$). This gives a bound $t x^{2c+o(1)}$ for the sum in \eqref{eq:jSum}, and so the key problem is to win back a little bit more than the $x^{o(1)}$ factor from the divisor bound. We wish to do this by gaining a structural understanding of sets $\CS$ where many pairs have a large GCD. One way that many pairs in $\CS$ can have a large GCD is if a positive proportion of elements of $\CS$ are a multiple of some fixed divisor $d$. It is natural to ask if this is the only such construction. If we ignore the $\phi(q)/q$ weights, this leads us to the following prototypical question that we shall refer to as the {\it Model Problem}.

\begin{ModelProblem}
Let $\CS\subseteq[x,2x]$ satisfy $\#\CS\asymp x^{c}$ and be such that there are $\#\CS^2/100$ pairs $(a_1,a_2)\in\CS^2$ with $\gcd(a_1,a_2)>x^{1-c}$. Must it be the case that there is an integer $d\gg x^{1-c}$ which divides $\gg \#\CS$ elements of $\CS$? 
\end{ModelProblem}
It turns out that the answer to this Model Problem as stated is `no', but a technical variant of it that is sufficient for proving Theorem \ref{thm:MainThm} has a positive answer. For the purposes of this section, we will ignore this subtle issue; we will return to it and discuss it in detail in Section \ref{sec:Concluding}.

\medskip

To attack our Model Problem, we use a `compression' argument, roughly inspired by the papers of Erd\H os-Ko-Rado \cite{erdoskorado} and Dyson \cite{dyson}. We will repeatedly pass to subsets of $\CS$ where we have increasing control over whether given primes occur in the GCDs or not, whilst at the same time showing that the size of the original set is controlled in terms of the size of the new set. At the end of the iteration procedure we will then have arrived at a subset which controls the size of $\CS$,  and where we know that all large GCDs are caused by a fixed divisor. Since the final set then has a very simple GCD structure, we will have enough information to establish \eqref{eq:jSum}.

To enable the iterations, we pass to a bipartite setup. We start out with sets $\CV_0=\CW_0=\CS$. Then, we construct two decreasing sequences of sets $\CV_0\supset \CV_1\supset\CV_2\supset \cdots$ and $\CW_0\supset\CW_1\supset\CW_2\supset \cdots$, as well as a sequence of primes $p_1,p_2,\dots$ such that either $p_j$ divides all elements of $\CV_j$, or $p_j$ is coprime to all elements of $\CV_j$ (and similarly with $\CW_j$). Since $\CS$ contains only square-free integers in the simplified set-up of this section, this means that there will be exponents $k_j,\ell_j\in\{0,1\}$ such that  $p_j^{k_j}\|v$ for all $v\in\CV_j$, and $p_j^{\ell_j}\|w$ for all $w\in\CW_j$. Hence, if we let $a_j=p_1^{k_1}\cdots p_j^{k_j}$ and $b_j=p_1^{\ell_1}\cdots p_j^{\ell_j}$, then $a_j$ will divide all elements of $\CW_j$ and $b_j$ will divide all elements of $\CV_j$.

We will construct the sets $\CV_1,\CV_2,\dots$ and $\CW_1,\CW_2,\dots$ in an iterative fashion. Assume that after $j$ iterations we have arrived at the sets $\CV_j,\CW_j\subseteq\CS$. We then pick a prime $p_{j+1}$ that is different from $p_1,\dots,p_j$, and that occurs as the prime factor of $\gcd(v,w)$ for some $v\in\CV_j$, $w\in\CW_j$ with $\gcd(v,w)>x^{1-c}/t$. Our goal is to pass judiciously to subsets $\CV_{j+1}\subseteq\CV_j$ and $\CW_{j+1}\subseteq\CW_j$ where either $\CV_{j+1}$ is all elements of $\CV_j$ that are divisible by $p_{j+1}$, or $\CV_{j+1}$ is all elements of $\CV_j$ coprime to $p_{j+1}$ (and similarly with $\CW_{j+1}$). Since we're assuming that $\CS$ contains only square-free integers, we then will completely know the $p_{j+1}$-divisibility of all elements of $\CV_{j+1}$ and $\CW_{j+1}$, so in particular all GCDs between an element of $\CV_{j+1}$ and $\CW_{j+1}$ will either be multiple of $p_{j+1}$, or all will be coprime to $p_{j+1}$.

Eventually, we will arrive at a pair of sets $(\CV_J,\CW_J)$ such that every pair $(v,w)\in\CV_J\times\CW_J$ with $\gcd(v,w)>x^{1-c}/t$ has the property that all prime factors of $\gcd(v,w)$ will lie in the set $\{p_1,\dots,p_J\}$ (and moreover we will ensure that there is at least one such pair). This terminates the iterative procedure. By construction, all elements of $\CV=\CV_J$ will be divisible by the fixed integer $a=a_J$, and similarly all elements of $\CW=\CW_J$ will be divisible by the fixed integer $b=b_J$. In addition, if $v\in\CV$ and $w\in\CW$ has $\gcd(v,w)>x^{1-c}/t$, then in fact $\gcd(v,w)$ will be exactly equal to $\gcd(a,b)$ since we know the $p_j$-divisibility for all elements of $\CV$ and $\CW$. Thus, $\gcd(a,b)>x^{1-c}/t$ and actually every pair $v\in \CV$ and $w\in\CW$ has $\gcd(v,w)=\gcd(a,b)$. 

Naturally, the success of the above strategy depends on improving the `structure' of the pair of sets $(\CV_j,\CW_j)$ at each stage of the algorithm. This will enable us to control a quantity like the left hand side of \eqref{eq:jSum} in terms of a related quantity for $(\CV,\CW)=(\CV_J,\CW_J)$. An initially appealing choice to measure the `structure' might be
\[
\delta_j=\frac{\#\{(v,w)\in\CV_j\times\CW_j:\, \gcd(v,w)>x^{1-c}/t\}}{\#\CV_j\cdot \#\CW_j},
\]
namely the density of pairs $(v,w)$ with large GCD at stage $j$. Iteratively increasing this quantity would try to mimic a `density increment' strategy such as that used in the proof of Roth's Theorem on arithmetic progressions \cite{roth1,roth2}. Unfortunately, such an argument loses all control over the size of the sets $\CV_j,\CW_j$, and so we lose control over the sum in \eqref{eq:jSum}.

An alternative suggestion might be to consider a different quantity which focuses on the size of the sets. Recall that all elements of $\CV_j$ are a multiple of $a_j$, all elements of $\CW_j$ are a multiple of $b_j$, and that in our final step we have $\gcd(a,b)>x^{1-c}/t$. Thus
\al{
\label{eq:anatomy2}
\#\CV\cdot \#\CW 
	&\le \#\{(v,w)\in(\Z\cap[x,2x])^2 : a|v,\ b|w\}  \ll \frac{x^2}{ab} \le  t^2 x^{2c}\cdot \frac{\gcd(a,b)^2}{ab} ,
}
where we used that $\CV=\CV_J$ and $\CW=\CW_J$ are subsets of $[x,2x]$ in \eqref{eq:anatomy2}. Thus, one might try to iteratively increase the quantity
\[
\#\CV_j\cdot \#\CW_j\cdot  \frac{a_jb_j}{\gcd(a_j,b_j)^2}.
\]
This would adequately control \eqref{eq:jSum}, but unfortunately it is not possible to guarantee that this quantity increases at each stage, and so this proposal also fails.

However, the variant
\begin{equation}
\delta_j^{10}\cdot \#\CV_j\cdot \#\CW_j \cdot\frac{a_jb_j}{\gcd(a_j,b_j)^2} 
\label{eq:QualityAttempt}
\end{equation}
turns out to (more-or-less) work well. Indeed, if the quantity \eqref{eq:QualityAttempt} increases at each iteration, and at the final iteration all elements of $\CV=\CV_J$ are a multiple of $a=a_J$, all elements of $\CW=\CW_J$ are a multiple of $b=b_J$, and all edges come from pairs $(v,w)$ with $\gcd(v,w)=\gcd(a,b)>x^{1-c}/t$, then we find that
\begin{equation}
\delta_0^{10}\cdot \#\CS^2\le \delta_J^{10}\cdot \#\CV\cdot \#\CW\cdot \frac{ab}{\gcd(a,b)^2}
	\le \#\CV\cdot \#\CW \cdot \frac{ab}{\gcd(a,b)^2}\ll t^2x^{2c}.
\label{eq:IterationBound1}
\end{equation}

We note that in our setup $\#\CS\asymp x^{c}$, and that
\begin{equation}
\sum_{\substack{q,r\in\CS \\ \gcd(q,r)>x^{1-c}/t}}1=\delta_0\cdot \#\CS^2.
\label{eq:EdgeRewrite}
\end{equation}
If it so happens that $\delta_0\le 1/t$, then we trivially obtain \eqref{eq:jSum} (ignoring the $\phi(q)/q$ weighting) from \eqref{eq:EdgeRewrite}. On the other hand, if $\delta_0\gg 1/t $, then \eqref{eq:IterationBound1} falls short of \eqref{eq:jSum} only by a factor $t^{12}$.

Finally, to win the additional factor of $t^{12}$ we make use of the fact that any edge $(q,r)$ in our graph satisfies \eqref{eq:tThreshold}. The crucial estimate is that 
\eq{\label{eq:anatomy-gain}
\#\Bigl\{n<x:\sum_{\substack{p|n\\ p\ge t}}\frac{1}{p}\ge 1\Bigr\}\ll e^{-t} x .
}
This was the key idea in the earlier work of Erd\H os \cite{erdos} and Vaaler \cite{vaaler} on the Duffin-Schaeffer conjecture. In our case, our iteration procedure has essentially reduced the proof to a similar situation to their work. 

Indeed, in \eqref{eq:anatomy2}, we may restrict our attention to pairs $(v,w)$ such that $a|v$, $b|w$ and 
\eq{\label{eq:anatomy-condition-rough}
\sum_{\substack{p|vw/\gcd(v,w)^2 \\ p\ge t}} \frac{1}{p} \approx 1.
}
Unless most of the contribution to the above sum of comes from primes in $a$ and $b$, we can apply \eqref{eq:anatomy-gain} to win a factor of size $e^{-t}=o(t^{-12})$ in \eqref{eq:anatomy2}. Finally, if the small primes in $a$ and $b$ do cause a problem, then a more careful analysis of our iteration procedure shows that we actually are able to increase the quantity \eqref{eq:QualityAttempt} by more than $t^{12}$ by the final stage $J$, which also suffices for establishing \eqref{eq:jSum} in this case.

\medskip

The above description has ignored several important technicalities; it turns out that the $\phi(q)/q$ weights are vital for our argument to work (see the discussion in Section \ref{sec:Concluding}). In addition, we do not quite work with \eqref{eq:QualityAttempt} but with a closely related (but more complicated) expression to enable this quantity to increase at each iteration. The iteration  procedure of our argument is broken up into different stages. In between two of the principal iterative stages, we perform a certain `clean-up' step at which we allow a small loss in the quantity \eqref{eq:QualityAttempt}. This step is essential in order to keep track of the condition \eqref{eq:anatomy-condition-rough} (which could otherwise become meaningless after too many iterations).

\section{Structure of the paper}\label{sec:structure}
In the first half of the paper that consists of Sections \ref{sec:Prelim}-\ref{sec:HighDegree}, we reduce the proof of Theorem \ref{thm:MainThm} to three technical iterative statements about particular graphs, which we call `GCD graphs' (see Definition \ref{gcd graph dfn}). Specifically, in Section \ref{sec:Prelim} we use a second moment argument to reduce the proof to Proposition \ref{prop:GCD}, which claims a suitable bound for sums of the form \eqref{eq:jSum}. Here, we make use of Lemmas \ref{lem:Gallagher}-\ref{lem:Pollington} which are standard results from the literature. In Section \ref{sec:Graph} we introduce the key terminology of the paper and translate Proposition \ref{prop:GCD} into Proposition \ref{prop:Graph}, a statement about edges in a particular `GCD graph'. In Section \ref{sec:Anatomy} we use results about the anatomy of integers (Lemmas \ref{lem:Tenenbaum} and \ref{lem:Anatomy}) to reduce the situation to establishing Proposition \ref{prop:BoundOnEdges}, a technical statement claiming the existence of a `good' GCD subgraph (where `good' here means that there are integers $a$ and $b$ such that all vertices in $\CV$ are divisible by $a$, those in $\CW$ are divisible by $b$, and if $(v,w)$ is an edge, then $\gcd(v,w)=\gcd(a,b)$). Then in Section \ref{sec:Iterative}, we reduce the proof of Proposition \ref{prop:BoundOnEdges} to five iterative claims which form the heart of the paper:  Propositions \ref{prop:IterationStep1}-\ref{prop:SmallPrimes} and Lemmas \ref{lem:Cosmetic}-\ref{lem:HighDegreeSubgraph}. In Sections \ref{sec:CosmeticProof} and \ref{sec:HighDegree} we then directly establish Lemmas \ref{lem:Cosmetic} and \ref{lem:HighDegreeSubgraph}, respectively, leaving the second half of the paper to demonstrate the key statements of Propositions \ref{prop:IterationStep1}-\ref{prop:SmallPrimes}.

The dependency diagram for the first half of the paper is as follows:

\smallskip

  \begin{center}
\makebox[\textwidth]{\parbox{1.5\textwidth}{
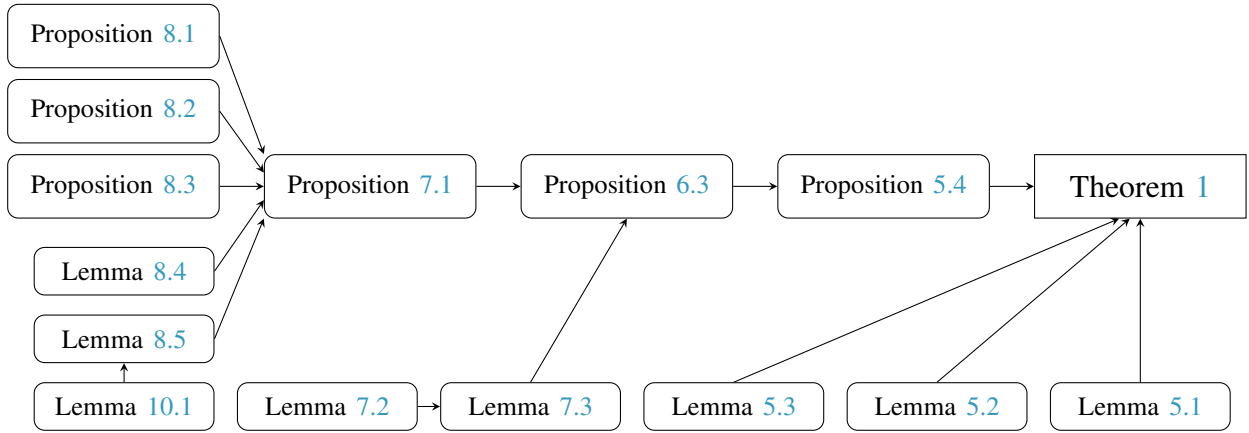
\begin{center}
   \tikzstyle{interface}=[draw, text width=6em,
      text centered, minimum height=2.0em]
   \tikzstyle{daemon}=[draw, text width=6em,
      minimum height=2em, text centered, rounded corners]
   \tikzstyle{lemma}=[draw, text width=5em,
      minimum height=1.5em, text centered, rounded corners]
   \tikzstyle{dots} = [above, text width=6em, text centered]
   \tikzstyle{wa} = [daemon, text width=6em,
      minimum height=2em, rounded corners]
   \tikzstyle{ur}=[draw, text centered, minimum height=0.01em]
   \def\blockdist{1.3}
   \def\edgedist{0.}
   \begin{tikzpicture}
      \node (wa)[interface]  {Theorem \ref{thm:MainThm}};
      \path (wa.west)+(-2,0) node (d1)[daemon] {\footnotesize Proposition  \ref{prop:GCD}};
      \path (d1.west)+(-2,0) node (d2)[daemon] {\footnotesize Proposition \ref{prop:Graph}};
      \path (d2.west)+(-2,0) node (d3)[daemon] {\footnotesize Proposition \ref{prop:BoundOnEdges}};

      \path (d3.west)+(-2,2) node (d6)[daemon] {\footnotesize Proposition \ref{prop:IterationStep1}};
      \path (d3.west)+(-2,1) node (d7)[daemon] {\footnotesize Proposition \ref{prop:IterationStep2}};
      \path (d3.west)+(-2,0) node (d8)[daemon] {\footnotesize Proposition \ref{prop:SmallPrimes}};
      \path (wa.south)+(0,-2.5) node (l1)[lemma] {\footnotesize Lemma \ref{lem:Gallagher}};
      \path (l1.west)+(-1.5,0) node (l2)[lemma] {\footnotesize Lemma \ref{lem:Large}};
      \path (l2.west)+(-1.5,0) node (l3)[lemma] {\footnotesize Lemma \ref{lem:Pollington}};
      \path (l3.west)+(-1.5,0) node (l4)[lemma] {\footnotesize Lemma \ref{lem:Anatomy}};
      \path (l4.west)+(-1.5,0) node (l5)[lemma] {\footnotesize Lemma \ref{lem:Tenenbaum}};
            \path (l5.west)+(-1.5,0) node (l6)[lemma] {\footnotesize Lemma \ref{lem:HighDegree}};
            
                  \path (l6)+(0,1.8) node (d4)[lemma] {\footnotesize Lemma \ref{lem:Cosmetic}};
      \path (l6)+(0,0.9) node (d5)[lemma] {\footnotesize Lemma \ref{lem:HighDegreeSubgraph}};
     \path [draw, ->,>=stealth] (d1.east) -- node [above] {} (wa.west) ;
     \path [draw, ->,>=stealth] (d2.east) -- node [above] {} (d1.west) ;
     \path [draw, ->,>=stealth] (d3.east) -- node [above] {} (d2.west) ;
      \path [draw, ->,>=stealth] (d4.east) -- node [above] {} ([yshift=7]d3.south west) ;
      \path [draw, ->,>=stealth] (d5.east) -- node [above] {} (d3.south west) ;
     \path [draw, ->,>=stealth] (d6.east) -- node [above] {} (d3.north west) ;
 \path [draw, ->,>=stealth] (d7.east) -- node [above] {} ([yshift=-7]d3.north west) ;
      \path [draw, ->,>=stealth] (d8.east) -- node [above] {} (d3.west) ;
      \path [draw, ->,>=stealth] (l1.north) -- node [above] {} (wa.south) ;
      \path [draw, ->,>=stealth] (l2.north) -- node [above] {} ([xshift=-4]wa.south) ;
            \path [draw, ->,>=stealth] (l3.north) -- node [above] {} ([xshift=-8]wa.south) ;
                        \path [draw, ->,>=stealth] (l5.east) -- node [above] {} (l4.west) ;
                                    \path [draw, ->,>=stealth] (l4.north) -- node [above] {} (d2.south) ;
                                    \path [draw, ->,>=stealth] (l6.north) -- node [above] {} (d5.south) ;
   \end{tikzpicture}
\end{center}}}
   \end{center}
%
%
%
%
%
%

\medskip

The second half of the paper consists of Sections \ref{sec:Prep}-\ref{sec:IterationStep2}, and it is devoted to proving each of Proposition \ref{prop:IterationStep1}, \ref{prop:IterationStep2} and \ref{prop:SmallPrimes}. Before we embark on the proofs directly, we first establish several preparatory lemmas in Section \ref{sec:Prep}. In particular we prove Lemmas \ref{lem:Pigeonhole}-\ref{lem:NoSmallSetEdges} which are minor results on GCD graphs we will use later on. Section \ref{sec:IterationStep1} is dedicated to the proof of Proposition \ref{prop:IterationStep1}, which is the easier iteration step, and relies on two auxiliary results: Lemmas  \ref{lem:EdgeSets} and \ref{lem:MainLem}. Section \ref{sec:SmallPrime} is dedicated to the proof of Proposition \ref{prop:SmallPrimes}, the iteration procedure for small primes. This proposition follows from Lemma \ref{lem:SmallIteration}, in turn relying on Lemmas \ref{lem:Pigeonhole}, \ref{lem:UnbalancedSetEdges1} and \ref{lem:SmallPrime}. Finally, in Section \ref{sec:IterationStep2} we prove Proposition \ref{prop:IterationStep2}, which is the most delicate part of the iteration procedure. This follows quickly from Lemma \ref{lem:MainLem2}, which in turn relies on Lemmas \ref{lem:UnbalancedSetEdges1}-\ref{lem:NoSmallSetEdges}. The dependency diagram for the second half of the paper is as follows:

%
%
%
%
%
%

\smallskip

  \begin{center}
\makebox[\textwidth]{\parbox{1.5\textwidth}{
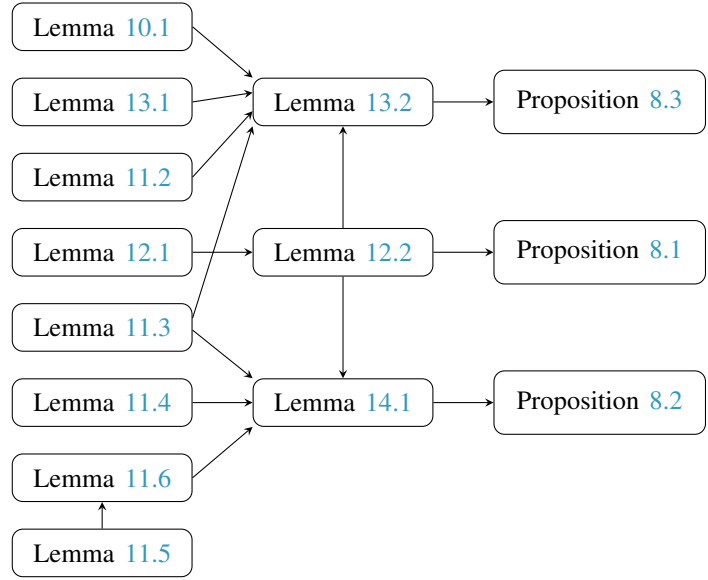
\begin{center}
   \tikzstyle{interface}=[draw, text width=6em,
      text centered, minimum height=2.0em]
   \tikzstyle{daemon}=[draw, text width=6em,
      minimum height=2em, text centered, rounded corners]
   \tikzstyle{lemma}=[draw, text width=5em,
      minimum height=1.5em, text centered, rounded corners]
   \tikzstyle{dots} = [above, text width=6em, text centered]
   \tikzstyle{wa} = [daemon, text width=6em,
      minimum height=2em, rounded corners]
   \tikzstyle{ur}=[draw, text centered, minimum height=0.01em]
   \def\blockdist{1.3}
   \def\edgedist{0.}
      \begin{tikzpicture}
   \node (d6)[daemon] {\footnotesize Proposition \ref{prop:IterationStep1}};
   \path (d6)+(0,2) node (d8)[daemon] {\footnotesize Proposition \ref{prop:SmallPrimes}};
   \path (d6)+(0,-2) node (d7)[daemon] {\footnotesize Proposition \ref{prop:IterationStep2}};
   \path (d6.west)+(-2,0) node (l1)[lemma] {\footnotesize Lemma \ref{lem:MainLem}};
   \path (l1.west)+(-2,0) node (l2)[lemma] {\footnotesize Lemma \ref{lem:EdgeSets}};
   \path (d7.west)+(-2,0) node (l3)[lemma] {\footnotesize Lemma \ref{lem:MainLem2}};
   \path (l3.west)+(-2,0) node (l4)[lemma] {\footnotesize Lemma \ref{lem:UnbalancedSetEdges2}};
   \path (l3.west)+(-2,-1) node (l5)[lemma] {\footnotesize Lemma \ref{lem:NoSmallSetEdges}};
   \path (l3.west)+(-2,1) node (l6)[lemma] {\footnotesize Lemma \ref{lem:UnbalancedSetEdges1}};
   \path (d8.west)+(-2,0) node (l7)[lemma] {\footnotesize Lemma \ref{lem:SmallIteration}};
   \path (l7.west)+(-2,-1) node (l8)[lemma] {\footnotesize Lemma \ref{lem:Pigeonhole}};
   \path (l7.west)+(-2,0) node (l9)[lemma] {\footnotesize Lemma \ref{lem:SmallPrime}};
   \path (l7.west)+(-2,1) node (l10)[lemma] {\footnotesize Lemma \ref{lem:HighDegree}};
   \path (l3.west)+(-2,-2) node (l11)[lemma] {\footnotesize Lemma \ref{lem:SmallSetEdges}};
   \path [draw, ->,>=stealth] (l1.east) -- node [above] {} (d6.west) ;
   \path [draw, ->,>=stealth] (l2.east) -- node [above] {} (l1.west) ;
   \path [draw, ->,>=stealth] (l3.east) -- node [above] {} (d7.west) ;
   \path [draw, ->,>=stealth] (l4.east) -- node [above] {} (l3.west) ;
   \path [draw, ->,>=stealth] (l5.east) -- node [above] {} (l3.south west) ;
   \path [draw, ->,>=stealth] ([yshift=-1]l6.east) -- node [above] {} (l3.north west) ;
   \path [draw, ->,>=stealth] ([yshift=3.5]l6.east) -- node [above] {} (l7.south west) ;                  
   \path [draw, ->,>=stealth] (l8.east) -- node [above] {} ([yshift=-3.5]l7.west) ;
   \path [draw, ->,>=stealth] (l9.east) -- node [above] {} ([yshift=3.5]l7.west) ;
   \path [draw, ->,>=stealth] (l7.east) -- node [above] {} (d8.west) ;
   \path [draw, ->,>=stealth] (l10.east) -- node [above] {} (l7.north west) ;
   \path [draw, ->,>=stealth] (l1.north) -- node [above] {} (l7.south) ;
   \path [draw, ->,>=stealth] (l11.north) -- node [above] {} (l5.south) ;
   \path [draw,->,>=stealth] (l1.south) -- node [above] {} (l3.north) ;
   \end{tikzpicture}
\end{center}}}
   \end{center}

%
%

\smallskip

(We have not included the essentially trivial statement of Lemma \ref{lem:InducedGraphs} or Lemma \ref{lem:Trivial} which are used frequently in the later sections.) All lemmas are proven in the section where they appear with the exception of Lemma \ref{lem:Cosmetic} and Lemma \ref{lem:HighDegreeSubgraph}, which are proven in Sections \ref{sec:CosmeticProof} and \ref{sec:HighDegree} respectively. All propositions are proven in sections later than they appear.

\section{Preliminaries}\label{sec:Prelim}

We first reduce the proof of Theorem \ref{thm:MainThm} to a second moment bound given by Proposition \ref{prop:GCD} below. This reduction is standard and appears in several previous works on the Duffin-Schaeffer conjecture. In particular, a vital component is the following ergodic 0-1 law due to Gallagher \cite{gallagher}.

\begin{lem}[Gallagher's 0-1 law]\label{lem:Gallagher}
Consider a function $\psi:\N\to\R_{\ge0}$ and let $\CA$ be as in \eqref{eq:A-dfn}. Then either $\lambda(\CA)=0$ or $\lambda(\CA)=1$. 
\end{lem}
\begin{proof}
This is Theorem 1 of \cite{gallagher}.
\end{proof}

\begin{lem}[The Duffin-Schaeffer Conjecture when $\psi$ only takes large values]\label{lem:Large}
Let $\psi:\N\rightarrow\R_{\ge 0}$ be a function, and let $\CA$ be as in \eqref{eq:A-dfn}. Assume, further, that:
\begin{enumerate}
\item For every $q\in\mathbb{Z}$, either $\psi(q)=0$ or $\psi(q)\ge 1/2$;
\item $\sum_{q=1}^\infty \psi(q)\phi(q)/q=\infty$.
\end{enumerate}
Then $\lambda(\CA)=1$.
\end{lem}

\begin{proof}
This follows from \cite[Theorem 2]{PV}.
\end{proof}
\begin{lem}[Bound for $\lambda(\CA_q\cap\CA_r)$]\label{lem:Pollington}
Consider a function $\psi:\N\to[0,1/2]$ and let $\CA_q$ be as in \eqref{eq:A_q-dfn}. In addition, given $q,r\in\N$, set
\[
M(q,r):=\max\{r\psi(q),q\psi(r)\}.
\]
If $q\neq r$, then we have
\[
\frac{\lambda(\CA_q\cap \CA_r)}
{\lambda(\CA_q)\lambda(\CA_r)}
	\ll \un_{M(q,r)\ge \gcd(q,r)} 
	\prod_{\substack{p|qr/\gcd(q,r)^2 \\ p>M(q,r)/\gcd(q,r)}} \left(1+\frac{1}{p}\right) .
\]
\end{lem}
\begin{proof}
This bound is given in \cite[p. 195-196]{PV}.
\end{proof}

Given the above lemma, we introduce the notation
\begin{equation}\label{eq:L_t-dfn}
L_t(a,b):= 
\sum_{\substack{p|ab/\gcd(a,b)^2 \\ p\ge t}} \frac{1}{p}
\end{equation}
for $a,b\in\N$ and $t\ge1$. The key result to proving Theorem \ref{thm:MainThm} is:

\begin{prop}[Second moment bound]\label{prop:GCD}
Let $\psi$ and $M(q,r)$ be as in as in Lemma \ref{lem:Pollington}, and consider $Y\ge X\ge1$ such that
\[
1\le \sum_{X\le q\le Y} \frac{\psi(q)\phi(q)}{q} \le2 .
\]
For each $t\ge1$, set
\eq{\label{eq:E_t}
\CE_t 
	= \big\{(v,w)\in (\Z\cap[X,Y])^2 :  \gcd(v,w)\ge  t^{-1}\cdot M(v,w) ,\ L_t(v,w)  \ge 10 \big\}.
}
Then 
\[
\sum_{(v,w)\in \CE_t}  \frac{\phi(v)\psi(v)}{v}\cdot \frac{\phi(w)\psi(w)}{w} \ll \frac{1}{t} .
\]
\end{prop}

\begin{proof}[Proof of Theorem \ref{thm:MainThm} assuming Proposition \ref{prop:GCD}]
We wish to prove that
\begin{equation}\label{eq:MainGoal}
\lambda(\CA)=1,
\end{equation}
where $\CA=\limsup_{q\to\infty}\CA_q$ with $\CA_q$ defined by \eqref{eq:A-dfn}. We first  write 
\[
\psi(q)=\psi_1(q)+\psi_2(q),
\qquad\text{where}\qquad
\psi_1(q)
	=\begin{cases}
	\psi(q)&\text{if}\ \psi(q)>1/2,\\
	0&\text{otherwise} .
	\end{cases}
\] 
In particular, $\psi_2(q)=\psi(q)$ if $\psi(q)\le 1/2$, and $\psi_2(q)=0$ otherwise. 

If it so happens that $\sum_{q=1}^\infty \psi_1(q)\phi(q)/q=\infty$, then we apply Lemma \ref{lem:Large} to $\psi_1$ to find that $\lambda(\limsup_{q\to\infty}\CB_q)=1$, where $\CB_q$ is defined as $\CA_q$ but with $\psi$ replaced by $\psi_1$. This proves \eqref{eq:MainGoal}, since $\psi_1(q)\le \psi(q)$, and so $\CB_q\subseteq\CA_q$.

Therefore we may assume without loss of generality that $\sum_{q=1}^\infty \psi_1(q)\phi(q)/q<\infty$, and so $\sum_{q=1}^\infty \psi_2(q)\phi(q)/q=\infty$. 
Thus, we have reduced Theorem \ref{thm:MainThm} to the case when 
\[
\psi(q)\le 1/2\quad\text{for all}\ q\ge1. 
\]

By Lemma \ref{lem:Gallagher}, the Duffin-Schaeffer conjecture will follow if we prove that $\lambda(\CA)>0$, 
since this means $\CA$ cannot have measure 0. Note that 
\begin{equation}
\label{eq:limsup-rewrite}
\CA=\limsup_{q\to\infty}\CA_q =  \bigcap_{j=1}^{\infty}
\bigcup_{q\ge j}\CA_q .
\end{equation}
Now, let $X$ be a large parameter and fix $Y=Y(X)$ to be minimal such that
\[
\sum_{X\le q\le Y} \frac{\phi(q)\psi(q)}{q} \in[1,2] .
\]
(Such a $Y$ exists since $\psi(q)\le 1/2$ for all $q$.) Hence, we see that it suffices to prove that
\begin{equation}
\lambda\bigg(\bigcup_{X\le q\le Y} \CA_q\bigg)\gg1 
\label{eq:Finite}
\end{equation}
uniformly for all large enough $X$, since this implies that $\lambda(\CA)>0$ by virtue of \eqref{eq:limsup-rewrite}, and hence Theorem \ref{thm:MainThm} follows.

For each $\alpha\in\R$, consider the counting function
\[  
Q(\alpha)=\#\{q\in\Z\cap [X,Y]: \alpha\in\CA_q\} = \sum_{X\le q\le Y} \un_{\CA_q}(\alpha).
\]
We then have
\begin{align*}
\supp(Q)&=\bigcup_{X\le q\le Y}\CA_q,\\
\int_0^1 Q(\alpha) \dee\alpha &= \sum_{X\le q\le Y}\lambda(\CA_q) \ge \sum_{X\le q\le Y} \frac{\phi(q)\psi(q)}{q} \ge 1 ,\\
\int_0^1 Q(\alpha)^2 \dee\alpha &= \sum_{X\le q,r \le Y}\lambda(\CA_q\cap \CA_r)  .
\end{align*}
Hence, the Cauchy-Schwarz inequality implies that
\[
\lambda\bigg(\bigcup_{X\le q\le Y}\CA_q\bigg)\int_0^1 Q(\alpha)^2\dee\alpha 
\ge\Bigl(\int_0^1Q(\alpha)\dee \alpha\Bigr)^2\ge1.
\]
Thus, to establish \eqref{eq:Finite}, it is enough to prove that
\begin{equation}\label{eq:2ndmoment-goal}
\sum_{X\le q,r\le Y} \lambda(\CA_q\cap \CA_r) \ll  1 .
\end{equation}
The terms with $q=r$ contribute a total
\[
\sum_{X\le q\le Y}\lambda(\CA_q)\le \sum_{X\le q\le Y} \frac{2\phi(q)\psi(q)}{q}\le 4,
\]
and so we only need to consider the contribution of those terms with $q\ne r$. Applying Lemma \ref{lem:Pollington}, we see that 
\[
\lambda(\CA_q\cap \CA_r)\ll 
	\un_{M(q,r)\ge \gcd(q,r)}\cdot \frac{\phi(q)\psi(q)}{q}\cdot \frac{\phi(r)\psi(r)}{r}\cdot 
	\prod_{\substack{p|qr/\gcd(q,r)^2 \\ p>M(q,r)/\gcd(q,r)}} \left(1+\frac{1}{p}\right) ,
\]
where we recall that
\[
M(q,r)=\max\{r\psi(q),q\psi(r)\}.
\]
Thus, \eqref{eq:2ndmoment-goal} is reduced to showing that
\begin{equation}\label{eq:goal}
	\sum_{\substack{X\le q,r\le Y \\ M(q,r)\ge\gcd(q,r) }} 
		\frac{\phi(q)\psi(q)}{q} \cdot\frac{\phi(r)\psi(r)}{r}
			 \prod_{\substack{p|qr/\gcd(q,r)^2 \\ p>M(q,r)/\gcd(q,r)}}\left(1+\frac{1}{p}\right) \ll 1.
\end{equation}
To prove this inequality, we divide the range of $q$ and $r$ into convenient subsets. 

The pairs $(q,r)\in (\Z\cap [X,Y])^2$ with
\[
 \prod_{\substack{p|qr/\gcd(q,r)^2 \\ p>M(q,r)/\gcd(q,r)}}\left(1+\frac{1}{p}\right) < e^{100} 
\]
contribute a total of at most
\[
e^{100}\bigg(\sum_{q\in [X,Y] } 
		\frac{\phi(q)\psi(q)}{q}\bigg)^2
		\le 4e^{100} 
\]
to the right hand side of \eqref{eq:goal}, and so can be ignored. 

For any other pair $(q,r)$, we see that 
\[
e^{100} \le 
	\prod_{\substack{p|qr/\gcd(q,r)^2 \\ p>M(q,r)/\gcd(q,r)}}\left(1+\frac{1}{p}\right)
	\le \exp\Bigl(\sum_{\substack{p|qr/\gcd(q,r)^2 \\ p>M(q,r)/\gcd(q,r)}}\frac{1}{p}\Bigr),
\]
so certainly we have
\[
\sum_{p|qr/\gcd(q,r)^2}\frac{1}{p} \ge 100.
\]
For any such pair, we let $j=j(q,r)$ be the largest integer such that
\[
\sum_{\substack{p| qr/\gcd(q,r)^2 \\ p\ge \exp\exp(j)}} \frac{1}{p}  \ge 10. 
\]
Since $j$ is chosen maximally, we have 
\[
\sum_{\substack{p| qr/\gcd(q,r)^2 \\ p\ge \exp\exp(j+1)}} \frac{1}{p}  <10 .
\]
Mertens' theorem then implies that 
\[
\sum_{\substack{p| qr/\gcd(q,r)^2 \\ p\ge \exp\exp(j)}} \frac{1}{p}=\sum_{\substack{p| qr/\gcd(q,r)^2 \\ \exp\exp(j)\le p <\exp\exp(j+1)}} \frac{1}{p}+\sum_{\substack{p| qr/\gcd(q,r)^2 \\ p\ge \exp\exp(j+1)}} \frac{1}{p}  \ll 1  .
\]
Therefore
\als{
	\prod_{\substack{p| qr/\gcd(q,r)^2 \\ p>M(q,r)/\gcd(q,r)}}\left(1+\frac{1}{p}\right)
	&\ll \prod_{M(q,r)/\gcd(q,r)<p\le \exp\exp(j)} \left(1+\frac{1}{p}\right)  \\
	&\ll\begin{cases}
		1&\text{if}\ M(q,r)/\gcd(q,r)\ge  \exp\exp(j),\\
		e^j &\text{otherwise},
	\end{cases}
}
where we used again Mertens' theorem. As above, those pairs with 
\[
\prod_{\substack{p| qr/\gcd(q,r)^2 \\ p>M(q,r)/\gcd(q,r)}}\left(1+\frac{1}{p}\right)\ll 1
\]
make an acceptable contribution to \eqref{eq:goal}. Therefore we only need to consider pairs $(q,r)$ with $M(q,r)/\gcd(q,r)<  \exp\exp(j)$.

We have thus reduced \eqref{eq:goal} to showing that
\begin{equation}\label{eq:goal2}
	\sum_{j\ge 0} e^j 
	\sum_{(q,r)\in \CE_{\exp\exp(j)}} \frac{\phi(q)\psi(q)}{q}\cdot \frac{\phi(r)\psi(r)}{r} \ll 1,
\end{equation}
where $\CE_t$ is defined by \eqref{eq:E_t}. To prove \eqref{eq:goal2}, we apply Proposition \ref{prop:GCD}, which shows that the inner sum is $O(1/\exp\exp(j))$. Since the sum of $e^j/\exp\exp(j)$ over $j\ge 0$ converges, this completes the proof of Theorem \ref{thm:MainThm}.
\end{proof}

Thus we are left to establish Proposition \ref{prop:GCD}.

\section{Bipartite GCD graphs}\label{sec:Graph}

In this section we introduce the key notation that will underlie the rest of the paper. In particular, we show that Proposition \ref{prop:GCD} follows from a statement given by Proposition \ref{prop:Graph} about a weighted graph with additional information about divisibility of the integers making up its vertices. The rest of the paper is then dedicated to establishing suitable properties of such graphs, which we call `GCD graphs'.

If we let
\[
\CV=\{q\in\Z\cap[X,Y]:\psi(q)\neq0\}
\]
and we weight the elements of $\CV$ with the measure
\[
\mu(q)=\frac{\phi(q)\psi(q)}{q},
\]
then Proposition \ref{prop:GCD} can be interpreted as an estimate for the weighted edge density of the graph with set of vertices $\CV$ and set of edges $\CE_t$ defined by \eqref{eq:E_t}.

Our strategy for proving Proposition \ref{prop:GCD} is to use a `compression' argument. More precisely, if $G_1$ denotes the graph described in the above paragraph, we will construct a finite sequence of graphs $G_1,\dots,G_J$ where we make a small local change to pass from $G_j$ to $G_{j+1}$ that increases the amount of structure in the graph. The final graph $G_J$ will then be highly structured and easy to analyze. To keep control over the procedure, we keep track of how certain statistics of the graph change at each step. This enables us to show that the relevant properties of $G_j$ are suitably controlled by $G_{j+1}$, and so $G_1$ is controlled by $G_J$, where everything is explicit.

To perform the above construction, we introduce some new notation to take into account the extra information about prime power divisibility which we need to carry at each stage.

\begin{dfn}[GCD graph]\label{gcd graph dfn}
Let $G$ be a septuple $(\mu,\CV,\CW,\CE,\CP,f,g)$ such that:	
\begin{enumerate}
		\item $\mu$ is a measure on $\N$ such that $\mu(n)<\infty$ for all $n\in\N$; we extend to $\N^2$ by letting
	\[
	\mu(\CN):=\sum_{(n_1,n_2)\in\CN}\mu(n_1)\mu(n_2) 
	\quad\text{for}\quad\CN\subseteq\N^2;	
	\]	
		\item $\CV$ and $\CW$ are finite sets of positive integers;
		\item $\CE\subseteq\CV\times \CW$, that is to say $(\CV,\CW,\CE)$ is a bipartite graph;
		\item $\CP$ is a set of primes;
		\item $f$ and $g$ are functions from $\CP$ to $\Z_{\ge0}$ such that for all $p\in\CP$ we have:
		\begin{enumerate}
		\item $p^{f(p)}|v$ for all $v\in\CV$, and $p^{g(p)}|w$ for all $w\in\CW$;
		\item if $(v,w)\in\CE$, then $p^{\min\{f(p),g(p)\}}\|\gcd(v,w)$;
		\item if $f(p)\neq g(p)$, then $p^{f(p)}\|v$ for all $v\in\CV$, and $p^{g(p)}\|w$ for all $w\in\CW$.
		\end{enumerate}
\end{enumerate}
We then call $G$ a (bipartite) \emph{GCD graph} with \emph{sets of vertices} $(\CV,\CW)$, \emph{set of edges} $\CE$ and \emph{multiplicative data} $(\CP,f,g)$. We will also refer to $\CP$ as the \emph{set of primes} of $G$. If $\CP=\emptyset$, we say that $G$ has \emph{trivial} set of primes and we view $f=f_\emptyset$ and $g=g_\emptyset$ as two copies of the empty function from $\emptyset$ to $\mathbb{Z}_{\ge 0}$.
\end{dfn}

\begin{dfn}[Non-trivial GCD graph]\label{non-trivial gcd graph dfn}
Let $G=(\mu,\CV,\CW,\CE,\CP,f,g)$. We say that $G$ is {\it non-trivial} if $\mu(\CE)>0$. 
\end{dfn}


We now recast Proposition \ref{prop:GCD} in the language of GCD graphs.

\begin{prop}[Edge set bound]\label{prop:Graph}
	Let $\psi:\N\rightarrow\R_{\ge 0}$, $t\ge1$ and $\mu$ be the measure $\mu(v)=\psi(v)\phi(v)/v$. Let $\CV\subseteq\N$ satisfy $0<\mu(\CV)\ll 1$. Let  $G=(\mu,\CV,\CV,\CE,\emptyset,f_\emptyset,g_\emptyset)$ be a bipartite GCD graph with measure $\mu$, vertex sets $\CV$, trivial set of primes, and edge set $\CE\subseteq\CE_t$, where $\CE_t$ is defined as in Proposition \ref{prop:GCD}. Then
	\[
	\mu(\CE)\ll 1/t.
	\]
\end{prop}

\begin{proof}[Proof of Proposition \ref{prop:GCD} assuming Proposition \ref{prop:Graph}]
	Recall the notation $\psi$, $M(q,r)$, $L_t(a,b)$, $X$, $Y$ and $\CE_t$ of Proposition \ref{prop:GCD}. We wish to show that
	\begin{equation}
	\sum_{(v_1,v_2)\in \CE_t} \frac{\phi(v_1)\psi(v_1)}{v_1}\cdot \frac{\phi(v_2)\psi(v_2)}{v_2}
	\ll  \frac{1}{t}  .
	\label{eq:Target}
	\end{equation}
	Let $\mu$ be the measure on $\N$ defined by $\mu(v):=\psi(v)\phi(v)/v$ and let $\CV=\Z\cap[X,Y]$, 
	so that
	\[
	\mu(\CV)=\sum_{X\le q\le Y}\frac{\phi(q)\psi(q)}{q}\in[1,2].
	\]
	Now define $\CE=\CE_t$ to be as in Proposition \ref{prop:GCD}. We see that $(\CV,\CV,\CE)$ forms a bipartite graph with vertex sets two copies of $\CV$ and edge set $\CE$. We now turn this bipartite graph into a GCD graph $G=(\mu,\CV,\CV,\CE,\emptyset,f_\emptyset,g_\emptyset)$ by attaching trivial multiplicative data to the bipartite graph (here $f_\emptyset$ and $g_\emptyset$ are viewed as two copies of the function of the empty set to $\mathbb{Z}_{\ge 0}$).
	
	Since $0<\mu(\CV)\ll 1$, Proposition \ref{prop:Graph} now applies, showing that
	\[
	\mu(\CE_t)\ll 1/t.
	\]
	This completes the proof.
\end{proof}

Thus we are left to establish Proposition \ref{prop:Graph}. As we briefly explained in Section \ref{sec:outline}, this will be done by passing iteratively to subgraphs of $G$ on which we control the divisibility by more and more primes. To formalize this procedure, we introduce the concept of a GCD subgraph.

\begin{dfn}[GCD subgraph]
Let $G=(\mu,\CV,\CW,\CE,\CP,f,g)$ and $G'=(\mu',\CV',\CW',\CE',\CP',f',g')$ be two GCD graphs. We say that $G'$ is a \emph{GCD subgraph} of $G$ if:
	\[
	\mu'=\mu,\quad \CV'\subseteq\CV,\quad \CW'\subseteq\CW,\quad \CE'\subseteq\CE,
	\quad \CP'\supseteq\CP, \quad f'\big|_{\CP}=f,\quad
	 g'\big|_{\CP}=g.
	\]
	We write $G'\preceq G$ if $G'$ is a GCD subgraph of $G$. Lastly, we say that $G'$ is a {\it non-trivial GCD subgraph} of $G$ if $\mu(\CE')>0$, that is to say $G'$ is non-trivial as a GCD graph.
\end{dfn}

We thus see from the above definition that we only accept $G'$ as a subgraph of $G$ if we have at least as much information about the divisibility of the vertices of $G'$ compared to those of $G$. In particular, we have that $p^{\min(f(p),g(p))} \| \gcd(v',w')$ for all $(v',w')\in\CE'$ and all $p\in\CP$.

We will devise an iterative argument that adds one prime at a time to $\CP$, so that we will eventually control the multiplicative structure of GCDs of connected vertices in the graph very well by the end of this process.

The main way we will produce a GCD subgraph of a GCD graph $G$ is by restricting to vertex sets with certain divisibility properties. Since we will use this several times, we introduce a specific notation for these GCD subgraphs.

\begin{dfn}[Special GCD subgraphs from prime power divisibility] 
\label{def:special graphs}	
Let $p$ be a prime number, and let $k,\ell\in\Z_{\ge0}$.
\begin{enumerate}
\item If $\CV$ is a set of integers and $k\in\Z_{\ge0}$, we set
	\[
	\CV_{p^k}=\{v\in\CV:p^k\|v\},
	\]
that is to say $\CV_{p^k}$ is the set of integers in $\CV$ whose $p$-adic valuation is exactly $k$. Here we have the understanding that $\CV_{p^0}$ denotes the set of $v\in\CV$ that are coprime to $p$. In particular, $\CV_{2^0}$ and $\CV_{3^0}$ denote different sets of integers. 
\item Let $G=(\CV,\CW,\CE)$ be a bipartite graph. If $\CV'\subseteq\CV$ and $\CW'\subseteq\CW$, we define
\[
\CE(\CV',\CW'):=\CE\cap(\CV'\times\CW').
\]
We also write for brevity
\[
\CE_{p^k,p^\ell}:=\CE(\CV_{p^k},\CW_{p^\ell}) .
\]
\item Let $G=(\mu,\CV,\CW,\CE,\CP,f,g)$ be a GCD graph such that $p\notin \CP$. We then define the septuple
\[
G_{p^k,p^\ell}=(\mu,\CV_{p^k},\CW_{p^\ell},\CE_{p^k,p^\ell},\CP\cup\{p\},f_{p^k},g_{p^\ell})
\]
where the functions $f_{p^k}$, $g_{p^\ell}$ are defined on $\CP\cup\{p\}$ by the relations $f_{p^k}\vert_\CP=f$, $g_{p^\ell}\vert_\CP=g$, 
\[
f_{p^k}(p)=k\quad\text{and}\quad g_{p^\ell}(p)=\ell.
\]
It is easy to check that $G_{p^k,p^\ell}$ is a GCD subgraph of $G$.
\end{enumerate}
\end{dfn}

The aim of our iterative procedure is to obtain a simple GCD subgraph $G'$ of our initial graph $G$ where the key quantitative aspects of $G$ are controlled by the corresponding quantities of $G'$. Here `simple' graphs have many primes occurring in $\gcd(v,w)$ for $(v,w)\in\mathcal{E}$ to a fixed exponent, whilst for subgraphs to maintain control over the original graph we need to maintain sufficiently many edges relative to the number of vertices. This leads us to our last four definitions:
\begin{dfn}[Quantities associated to GCD graphs]
Let $G=(\mu,\CV,\CW,\CE,\CP,f,g)$ be a GCD graph.
\begin{enumerate}
	\item If $\mu(\CV),\mu(\CW)>0$, then we define the \emph{edge density} of $G$ by
	\[
	\delta=\delta(G):=\frac{\mu(\CE)}{\mu(\CV)\mu(\CW)}.
	\] 
	If $\mu(\CV)=0$ or $\mu(\CW)=0$, we define the edge density of $G$ to be $0$. 
	\item The \emph{neighbourhood sets} are defined by
	\[
	\Gamma_G(v):=\{w\in\CW:\,(v,w)\in\CE\}\quad\text{for any}\ v\in\CV,
	\]
and similarly 
\[
\Gamma_G(w):=\{v\in\CV:\,(v,w)\in\CE\}
\quad\text{for any}\ w\in \CW.
\]
	\item We let $\CR(G)$ be given by
	\[
	\CR(G):= \{p\notin\CP: \exists (v,w)\in\CE\ \text{such that}\ p|\gcd(v,w)\} .
	\]
	That is to say $\CR(G)$ is the set of primes occurring in a GCD which we haven't yet accounted for. We split this into two further subsets:
	\[
	\CR^\sharp(G):=\left\{p\in\CR(G): 
	\exists k\in\Z_{\ge0}\ \text{such that}\ \frac{\mu(\CV_{p^k})}{\mu(\CV)},\frac{\mu(\CW_{p^k})}{\mu(\CW)}\ge 1-\frac{10^{40}}{p}\right\}
	\]
	and
	\[
	\CR^\flat(G):=\CR(G)\setminus \CR^\sharp(G).
	\]
	\item The \emph{quality} of $G$ is defined by
	\als{
	q(G)&:= \delta^{10}\mu(\CV)\mu(\CW)
	\prod_{p\in\CP} \frac{p^{|f(p)-g(p)|}}{(1-\un_{f(p)=g(p)\ge1}/p)^2(1-1/p^{31/30})^{10}}, 
}
where $\delta$ is the edge density of $G$. 
\end{enumerate}	
\end{dfn}
\begin{rem}
If $\mu(\CV),\mu(\CW)>0$ we see that
\als{
		q(G)&= \frac{\mu(\CE)^{10}}{\mu(\CV)^{9}\mu(\CW)^{9}}
	\prod_{p\in\CP} \frac{p^{|f(p)-g(p)|}}{(1-\un_{f(p)=g(p)\ge1}/p)^2(1-1/p^{31/30})^{10}} .	
	}
\end{rem}

As mentioned in Section \ref{sec:outline}, there are two natural candidates for a quantity to increment; either $\delta$ or 
$\mu(\CV)\mu(\CW)\prod_{p\in\CP}p^{|f(p)-g(p)|}$ (this is the natural generalization to non-squarefree integers). One should essentially think of the quality as a `hybrid' of the two quantities, but with some additional factors which are included for technical reasons.  The factor
\[
\prod_{p\in\CP}\frac{1}{(1-1/p^{31/30})^{10}}
\]
always lies in the interval $[1,\zeta(31/30)^{10}]$, and so is always of size bounded away from 0 and from $\infty$. This factor is included merely for convenience, and allows us to have a quality increment even if there is a tiny loss in our arguments in terms of $p$. The factor 
\[
\prod_{p\in\CP}\frac{1}{(1-\un_{f(p)=g(p)\ge 1}/p)^2}
\]
is crucial for the proof of a quality increment in Lemma \ref{lem:MainLem2} and Proposition \ref{prop:IterationStep2}. This is related to the technical point that it is vital that the weights of our vertices contain the factor $\phi(q)/q$. We will discuss this feature in more detail in Section \ref{sec:Concluding}.

\medskip

We will repeatedly make use of some trivial properties of GCD graphs, given by Lemma \ref{lem:Trivial} below, without further comment.
\begin{lem}[Basic properties of GCD graphs]\label{lem:Trivial} 
	Let $G_1,G_2,G_3$ be GCD graphs.
	\begin{enumerate} 	
				\item The property of being a GCD subgraph is transitive: If $G_1\preceq G_2$ and $G_2\preceq G_3$, then $G_1\preceq G_3$
				\item If $G_1\preceq G_2$, then $\CR(G_1)\subseteq \CR(G_2)$.
				\item If  $G_1=(\mu,\CV,\CW,\CE,\CP,f,g)$ is non-trivial, then $\mu(\CV),\mu(\CW)>0$.
				\item Let $G_1$ have edge density $\delta$. Then the following are equivalent:
				\begin{enumerate}
					\item $G_1$ is non-trivial. 
					\item $\delta>0$.
					\item $q(G_1)>0$.
				\end{enumerate}
	\end{enumerate}
\end{lem}
\begin{proof}
All statements are immediate from the definition of GCD subgraphs and of non-trivial GCD graphs.
\end{proof}

\begin{rem} In part (b) of Lemma \ref{lem:Trivial}, it is not necessarily the case that $\CR^\flat(G_1)\subseteq\CR^\flat(G_2)$ nor that $\CR^\sharp(G_1)\subseteq \CR^\sharp(G_2)$.
\end{rem}

Having introduced all necessary terminology, we turn to the task of establishing Proposition \ref{prop:Graph}.

\section{Reduction to a good GCD subgraph}\label{sec:Anatomy}

In this section, we reduce the proof of Proposition \ref{prop:Graph} (and hence of Theorem \ref{thm:MainThm}) to finding a `good' GCD subgraph as described in Proposition \ref{prop:BoundOnEdges} below. This reduction utilizes some results showing that few integers have lots of fairly small prime factors (based on `the anatomy of integers').

\begin{prop}[Existence of a good GCD subgraph]\label{prop:BoundOnEdges}
Let $G=(\mu,\CV,\CW,\CE,\emptyset,f_\emptyset,g_\emptyset)$ be a GCD graph with trivial set of primes and edge density $\delta>0$. Assume further that
\[
\CE\subseteq\{(v,w)\in\N^2 :  L_t(v,w) \ge10\}
\]
for some $t$ satisfying 
\[
t\ge10\delta^{-1/50}\quad\text{and}\quad t> 10^{2000}.
\]
Then there is a GCD subgraph $G'=(\mu,\CV',\CW',\CE',\CP',f',g')$ of $G$ with edge density $\delta'>0$ such that:
\begin{enumerate}[label=(\alph*)]
\item  $\CR(G')=\emptyset$;\label{prop:Parta}
\item For all $v\in\CV'$, we have $\mu(\Gamma_{G'}(v))\ge(9\delta'/10)\mu(\CW')$;\label{prop:Partb}
\item For all $w\in\CW'$, we have $\mu(\Gamma_{G'}(w))\ge(9\delta'/10)\mu(\CV')$;\label{prop:Partc}
\item One of the following holds: \label{prop:Partd}
\begin{enumerate}[label=(\roman*)]
\item \label{prop:Case1}$q(G')\gg \delta t^{50} q(G)$;
\item \label{prop:Case2}$q(G')\gg q(G)$, and if $(v,w)\in\CE'$ and we write them as $v=v'\prod_{p\in\CP'}p^{f'(p)}$ and $w=w'\prod_{p\in\CP'}p^{g'(p)}$, then $L_t(v',w')\ge 4$.
\end{enumerate}
\end{enumerate}
\end{prop}

Our task is to prove that Proposition \ref{prop:BoundOnEdges} implies Proposition \ref{prop:Graph}. To do so, we need a couple of preparatory lemmas that exploit the condition that $L_t(v',w')
\ge4$ in Case (d)-(ii) of Proposition \ref{prop:BoundOnEdges}.

\begin{lem}[Bounds on multiplicative functions]\label{lem:Tenenbaum}
	Let $k\in\N$ and write $\tau_k$ for the $k$-th divisor function. If $f$ is a multiplicative function such that $0\le f\le \tau_k$, then
	\[
	\sum_{n\le x} f(n) \ll_k x \cdot \exp\Big\{\sum_{p\le x}\frac{f(p)-1}{p}\Big\} .
	\]
\end{lem}
\begin{proof}
This is \cite[Theorem 14.2, p. 145]{dk-book}.
\end{proof}

\begin{lem}[Few numbers with many prime factors]\label{lem:Anatomy}
	For $x,t\ge1$ and $c\in[1,10]$, we have
	\[
	\#\bigg\{n\le x: \sum_{\substack{p|n\\  p\ge t}}\frac{1}{p} \ge  c \bigg\} \ll x\exp\{ - t^{e^{c-1}}\}  \, ;
	\]
	the implied constant is absolute.
\end{lem}

\begin{proof} We may assume that $t$ is large enough, since the result is trivial when $t$ is bounded. Set $T=t^{e^{c-1}}$, so that $\sum_{t\le p<T}1/p \le c-1/2$ by Mertens' theorem. Hence 
	\begin{equation}
	\#\bigg\{n\le x:\sum_{\substack{p|n\\ p\ge t}}\frac{1}{p} \ge c \bigg\}	
	\le \#\bigg\{n\le x: \sum_{\substack{p|n\\ p\ge T}}\frac{1}{p}  \ge 1/2 \bigg\} 
	\le e^{-T} \sum_{n\le x} \prod_{\substack{p|n\\ p\ge T}} e^{2T/p}  .
	\label{eq:Rankin}
	\end{equation}
	We wish to apply Lemma \ref{lem:Tenenbaum} when $f$ is the multiplicative function with $f(p^\nu)=e^{2T/p}$ for $p\ge T$ and all $\nu\ge 1$, and $f(p^\nu)=1$ for $p<T$. In particular, $f(p^\nu)\le e^2\le 8$ for all prime powers $p^\nu$, so that $f\le \tau_8$. Thus 
	\[
	\sum_{n\le x} \prod_{p|n,\ p\ge T} e^{2T/p}  
	= \sum_{n\le x} f(n) 
	\ll x \cdot \exp\Big\{\sum_{p\le x} \frac{f(p)-1}{p}\Big\} 
	=  x \cdot \exp\Big\{\sum_{T\le p\le x} \frac{e^{2T/p}-1}{p}\Big\} .
	\]
Since $e^{2T/p}=1+O(T/p)$ for $p\ge T$, and $\sum_{p\ge T} T/p^2\ll1$, the sum of $(e^{2T/p}-1)/p$ over $p\in[T,x]$ is $O(1)$. We thus conclude that
	\[
	\#\Big\{n\le x:\sum_{\substack{p|n\\ p\ge t}}\frac{1}{p} \ge c \Big\}\ll e^{-T}x.
	\]
Since $T=t^{e^{c-1}}$, the lemma has been proven.
\end{proof}

\begin{proof}[Proof of Proposition \ref{prop:Graph} assuming Proposition \ref{prop:BoundOnEdges}] Fix $t\ge1$ and 
let $G$ be the GCD graph of Proposition \ref{prop:Graph} with set of edges $\CE\subseteq\CE_t$ (where $\CE_t$ is defined in Proposition \ref{prop:GCD}), weight $\mu(v)=\psi(v)\phi(v)/v$ and edge density $\delta$. 

If $\delta\ll 1/t$, then $\mu(\CE)\ll 1/t$ and so we are done. Therefore we may assume that 
\[
\delta\ge 1/t\quad\text{and}\quad t>10^{2000}.
\] 
Note that this implies that 
\[
t\ge 10\delta^{-1/50}.
\]

We apply Proposition \ref{prop:BoundOnEdges} to $G$ to find a GCD subgraph $G'=(\mu,\CV',\CW',\CE',\CP',f',g')$ of $G$ with edge density $\delta'$ satisfying either case \ref{prop:Partd}-\ref{prop:Case1} or \ref{prop:Partd}-\ref{prop:Case2} of its statement. In addition, we have that:
\begin{enumerate}
\item $\CR(G')=\emptyset$;
\item $\mu(\Gamma_{G'}(v))\ge(9\delta'/10)\mu(\CW')$ for all $v\in\CV'$;
\item $\mu(\Gamma_{G'}(w))\ge(9\delta'/10)\mu(\CV')$ for all $w\in\CW'$.
\end{enumerate}
Set
\[
a:=\prod_{p\in\CP'}p^{f'(p)}
\quad\text{and}\quad
b:=\prod_{p\in\CP'}p^{g'(p)}.
\]
The definition of a GCD graph implies that 
\[
a|v\quad\text{for all}\ v\in\CV',
\qquad b|w\quad \text{for all}\ w\in\CW'.
\] 
Moreover, since $\CR(G')=\emptyset$, and $p^{\min\{f'(p),g'(p)\}}\|\gcd(v,w)$ for all $(v,w)\in\CE'$, we have that 
\[
\gcd(v,w)=\gcd(a,b)
\quad\text{for all}\quad(v,w)\in\CE'.
\] 

Now, note that
\[
\prod_{p\in\CP'}p^{|f'(p)-g'(p)|} = \prod_{p\in\CP'}p^{\max\{f'(p),g'(p)\}-\min\{f'(p),g'(p)\}}= \frac{\lcm(a,b)}{\gcd(a,b)}
	= \frac{ab}{\gcd(a,b)^2} ,
\]
as well as 
\[
\prod_{p\in\CP'}\frac{1}{(1-\un_{f'(p)=g'(p)\ge1}/p)^2(1-1/p^{31/30})^{10}} 
	\ll \prod_{p\in\CP'}\frac{1}{(1-\un_{f'(p)=g'(p)\ge1}/p)^2} \le \frac{ab}{\phi(a)\phi(b)} .
\]
Consequently, from the definition of $q(\cdot)$, we find
\begin{align}
q(G') &=  (\delta')^{10}\mu(\CV')\mu(\CW')
	\prod_{p\in\CP'} \frac{p^{|f'(p)-g'(p)|}}{(1-\un_{f'(p)=g'(p)\ge1}/p)^2(1-1/p^{31/30})^{10}}\nonumber\\
	& \ll (\delta')^{10} \mu(\CV')\mu(\CW') \frac{ab}{\gcd(a,b)^2} \cdot  \frac{ab}{\phi(a)\phi(b)}\nonumber \\
	&= (\delta')^{9}\mu(\CE')  \frac{ab}{\gcd(a,b)^2} \cdot  \frac{ab}{\phi(a)\phi(b)} .\label{eq:QG'Bound}
\end{align}
Proposition \ref{prop:BoundOnEdges} offers a lower bound on $q(G')/q(G)$. Since
\[
q(G) = \delta^{10} \mu(\CV)\mu(\CV)= \delta^{9}\mu(\CE),
\]
we can obtain an upper bound on the size of $\mu(\CE)$ by estimating $q(G')$ from above.

\medskip

Note that 
\[
\CE'\subseteq\CE
\subseteq\{(v,w)\in\CV\times\CV: M(v,w)\le t\cdot \gcd(v,w)\},
\]
where we recall that $M(v,w)=\max\{v\psi(w),w\psi(v)\}$. 
Since $\gcd(v,w)=\gcd(a,b)$ for all $(v,w)\in\CE'$, we infer that
\[
\psi(v)\le \frac{t\cdot \gcd(a,b)}{w}
\quad\text{and}\quad 
\psi(w)\le \frac{t\cdot \gcd(a,b)}{v}
\quad\text{for all}\quad 
(v,w)\in\CE'.
\]
The vertex sets $\CV',\CW'$ are finite sets of positive integers. For each $v\in\CV'$, let $w_{\max}(v)$ be the largest integer in $\CW'$ such that $(v,w_{\max}(v))\in\CE'$. (This quantity is well-defined in virtue of property (b) above. In addition, we emphasise to the reader that `largest' refers to the size of elements as positive integers, and does not depend on the measure $\mu$.) Similarly, for each $w\in\CW'$, let $v_{\max}(w)$ be the largest element of $\CV'$ such that $(v_{\max}(w),w)\in\CE'$. Consequently,
\begin{equation}\label{psi ub}
\psi(v)\le \frac{t\cdot \gcd(a,b)}{w_{\max}(v)}
\quad\text{and}\quad 
\psi(w)\le \frac{t\cdot \gcd(a,b)}{v_{\max}(w)}
\quad\text{for all}\quad v\in \CV',\,w\in\CW'.
\end{equation}

Now, let $w_0$ be the largest integer in $\CW'$ and $\CE''=\{(v,w)\in\CE': (v,w_0)\in\CE'\}$. We then have
\begin{equation}\label{eq:w0}
w_{\max}(v)=w_0 \quad\text{for all}\quad(v,w)\in\CE'' .
\end{equation}
In addition, since $G'$ satisfies conditions (b) and (c) in the statement of Proposition \ref{prop:BoundOnEdges}, we have
\[
\mu(\CE'') 
	=\sum_{v\in \Gamma_{G'}(w_0)}\mu(v)\mu(\Gamma_{G'}(v))
	\ge \mu(\Gamma_{G'}(w_0))\cdot \frac{9\delta'\mu(\CW')}{10}
	\ge \bigg(\frac{9\delta'}{10}\bigg)^2 \mu(\CV')\mu(\CW')\ge \frac{\delta'\mu(\CE')}{2} .
\]

Substituting this bound into \eqref{eq:QG'Bound}, we find
\begin{equation}\label{eq:q(G') bound}
q(G') \ll (\delta')^{8} \mu(\CE'')  \frac{a b}{\phi(a)\phi(b)}\cdot \frac{a b}{\gcd(a,b)^2} 
	\le \mu(\CE'')  \frac{ab}{\gcd(a,b)^2} \cdot \frac{ab}{\phi(a)\phi(b)} .
\end{equation}
Here we used the trivial bound $\delta'\le 1$ in the second inequality.
In addition,
\als{
\mu(\CE'')  = \sum_{(v,w)\in\CE''} \frac{\psi(v)\phi(v)}{v} \cdot \frac{\psi(w)\phi(w)}{w} .
}
Since $a|v$ and $b|w$, we have $\phi(v)/v\le \phi(a)/a$ and $\phi(w)/w\le \phi(b)/b$. Therefore 
\[
\mu(\CE'')\le \frac{\phi(a)\phi(b)}{a b}\sum_{(v,w)\in\CE''}\psi(v)\psi(w).
\]
Together with \eqref{psi ub}, \eqref{eq:w0} and \eqref{eq:q(G') bound}, this implies that
\begin{equation}
q(G') \ll  t^2ab \sum_{(v,w)\in\CE''} \frac{1}{v_{\max}(w) w_0}  
	\le  t^2ab \sum_{(v,w)\in\CE'} \frac{1}{v_{\max}(w) w_0}  .
\label{eq:G'Quality}
\end{equation}
We now split our argument depending on whether \ref{prop:Partd}-\ref{prop:Case1} or \ref{prop:Partd}-\ref{prop:Case2} of Proposition \ref{prop:BoundOnEdges} holds.

\bigskip

\noindent
\textit{Case 1: \ref{prop:Partd}-\ref{prop:Case1} of Proposition \ref{prop:BoundOnEdges} holds.}

\medskip

In this case we have $q(G')\gg \delta t^{50} q(G)$. 
Writing $v=v' a$ and $w=w' b$, we find that
\begin{align*}
\sum_{(v,w)\in\CE'} \frac{1}{v_{\max}(w) w_0}
&\le \sum_{w'\le w_0/b}\frac{1}{ w_0 v_{\max}(b w')}
	\sum_{v'\le v_{\max}(b w')/a}1\\
&\le \sum_{w'\le w_0/b}\frac{1}{ w_0 v_{\max}(b w')}\cdot \frac{v_{\max}(bw')}{a} \\
&\le \frac{1}{ab} .
\end{align*}
Together with \eqref{eq:G'Quality}, this implies that
\[
q(G')\ll t^2.
\]
Since $q(G')\gg \delta t^{50} q(G)$ in this case, and since $\delta \ge 1/t$, this gives
\[
\mu(\CE)  = \delta^{-9}q(G) \ll \delta^{-10}t^{-50}q(G')\ll \frac{1}{\delta^{10}t^{48}} \ll \frac{1}{t} .
\]
This establishes Proposition \ref{prop:Graph} in this case.

\bigskip

\noindent
\textit{Case 2: \ref{prop:Partd}-\ref{prop:Case2} of Proposition \ref{prop:BoundOnEdges} holds.}

\medskip

Write $v=v' a$ and $w=w' b$. In this case 
\begin{equation}
	\label{eq:L_t-lowerbound}
	L_t(v',w')\ge 4\quad\text{whenever}\quad(v,w)\in\CE'.
\end{equation}
We also have $q(G')\gg q(G)$. 

From \eqref{eq:L_t-lowerbound}, we see that either 
\[
\sum_{p|v',\,p\ge t}\frac{1}{p}\ge 2
\quad\text{or}
\quad
\sum_{p|w',\,p\ge t}\frac{1}{p}\ge 2
\]
whenever $(v,w)\in\CE'$.
Consequently,
\als{
\sum_{(v,w)\in\CE'} \frac{1}{v_{\max}(w) w_0}  
	&\le  \mathop{\sum\sum}_{\substack{w'\le w_0/b,\ v'\le v_{\max}(b w')/a\\ \sum_{p|v' w',\,p\ge t}1/p\ge 4}}  \frac{1}{v_{\max}(b w') w_0}  \\
	&\le S_1+S_2,
}
where
\begin{align*}
S_1&=\mathop{\sum\sum}_{\substack{w'\le w_0/b,\ v'\le v_{\max}(b w')/a\\ \sum_{p|v',\,p\ge t}1/p\ge 2}}  \frac{1}{v_{\max}(b w') w_0} ,\\
S_2&=\mathop{\sum\sum}_{\substack{w'\le w_0/b,\ v'\le v_{\max}(b w')/a\\ \sum_{p|w',\,p\ge t}1/p\ge 2}}  \frac{1}{v_{\max}(b w') w_0}  .
\end{align*}
For $S_1$, we note that 
\als{
S_1&\le \sum_{w'\le w_0/b} \frac{1}{w_0v_{\max}(b w')} \sum_{\substack{v'\le v_{\max}(b w')/a \\ \sum_{p|v',\,p\ge t}1/p\ge 2}} 1\\
	&\ll \sum_{w'\le w_0/b} \frac{1}{w_0v_{\max}(b w')} \cdot \frac{v_{\max}(b w')/a}{t^2e^t} \\
	&\le \frac{1}{ab t^2e^t} 
}
by Lemma \ref{lem:Anatomy}, since $\exp(t^e)\gg\exp(t)t^2$. 
Similarly for $S_2$, we find that
\als{
S_2&\le\sum_{\substack{w'\le w_0/b \\ \sum_{p|w',\,p\ge t}1/p\ge 2}} \frac{1}{w_0v_{\max}(b w')} \sum_{v'\le v_{\max}(b w')/a}1\\
	&\le \sum_{\substack{w'\le w_0/b \\ \sum_{p|w',\,p\ge t}1/p\ge 2}} \frac{1}{aw_0} \\
	&\ll \frac{1}{ab t^2e^t} ,
}
by applying Lemma \ref{lem:Anatomy} once again. Substituting these bounds into \eqref{eq:G'Quality}, we conclude that
\[
q(G') \ll e^{-t}.
\]
Since we have $q(G)\ll q(G')$ and $\delta\ge 1/t$, this gives
\[
\mu(\CE)  = \delta^{-9}q(G) \ll t^{9} q(G')\ll t^{9} e^{-t} \ll 1/t. 
\]
This establishes Proposition \ref{prop:Graph} in all cases.
\end{proof}

Thus we are left to prove Proposition \ref{prop:BoundOnEdges}.

\section{Reduction of Proposition \ref{prop:BoundOnEdges} to three iterative propositions}\label{sec:Iterative}

We will prove Proposition \ref{prop:BoundOnEdges} by an iterative argument, where we repeatedly find GCD subgraphs with progressively nicer properties. In this section we reduce the proof to five technical iterative statements, given by three key propositions (Propositions \ref{prop:IterationStep1}-\ref{prop:SmallPrimes}) and two auxiliary lemmas (Lemmas \ref{lem:Cosmetic}-\ref{lem:HighDegreeSubgraph}) given below.

\begin{prop}[Iteration when $\CR^\flat(G)\ne\emptyset$]\label{prop:IterationStep1}
Let $G=(\mu,\CV,\CW,\CE,\CP,f,g)$ be a GCD graph with edge density $\delta>0$ such that 
\[
\CR(G)\subseteq\{p>10^{2000}\}\quad\text{and}\quad \CR^\flat(G)\neq\emptyset.
\] 
Then there is a GCD subgraph $G'$ of $G$ with edge density $\delta'>0$ and multiplicative data $(\CP',f',g')$ such that
	\[
	\CP\subsetneq\CP'\subseteq\CP\cup\CR(G),
	\quad
	\CR(G')\subsetneq \CR(G),
	\quad 
	\min\bigg\{1,\frac{\delta'}{\delta}\bigg\} \cdot \frac{q(G')}{q(G)}\ge 2^N ,
	\]
	where $N=\#\{p\in\CP'\setminus \CP:f'(p)\neq g'(p)\}$.
\end{prop}

\begin{prop}[Iteration when $\CR^\flat(G)=\emptyset$]\label{prop:IterationStep2}
Let $G=(\mu,\CV,\CW,\CE,\CP,f,g)$ be a GCD graph with edge density $\delta>0$ such that 
\[
\CR(G)\subseteq\{p> 10^{2000}\},
\quad
\CR^\flat(G)=\emptyset ,
\quad
\CR^\sharp(G)\neq\emptyset.
\] 
Then there is a GCD subgraph $G'=(\mu,\CV',\CW',\CE',\CP',f',g')$ of $G$ such that
\[
	\CP\subsetneq\CP'\subseteq\CP\cup\CR(G),
\quad
\CR(G')\subsetneq \CR(G),
\quad 
q(G')\ge q(G).
\]
\end{prop}

Propositions \ref{prop:IterationStep1} and \ref{prop:IterationStep2} deal with large primes. We need a complementary result that handles the small primes.

\begin{prop}[Bounded quality loss for small primes]\label{prop:SmallPrimes}
Let $G=(\mu,\CV,\CW,\CE,\emptyset,f_\emptyset,g_\emptyset)$ be a GCD graph with edge density $\delta>0$ and trivial set of primes. Then there is a GCD subgraph $G'=(\mu,\CV',\CW',\CE',\CP',f',g')$ of $G$ with edge density $\delta'>0$ such that
\[
\CP'\subseteq\{p\le 10^{2000}\},
\quad
\CR(G')\subseteq\{p> 10^{2000}\},
\quad
\min\bigg\{1,\frac{\delta'}{\delta}\bigg\} \cdot \frac{q(G')}{q(G)}\ge \frac{1}{10^{10^{3000}}} .
\]
\end{prop}

Finally, we need two further technical estimates. The first one strengthens the quality of the inequality $L_t(v,w)\ge10$ when the set $\CR^\flat(G)$ is empty, whereas the second allows one to pass to a subgraph where all vertices have high degree.

\begin{lem}[Removing the effect of $\CR(G)$ from $L_t(v,w)$]\label{lem:Cosmetic}
	Let $t\ge300$ and $G=(\mu,\CV,\CW,\CE,\CP,f,g)$ be a GCD graph with edge density $\delta>0$ such that
	\[
	\CR^\flat(G)=\emptyset,
	\quad 
	\delta\ge (10/t)^{50},
	\quad
	\CE\subseteq\{(v,w)\in\CV\times\CW: L_t(v,w)\ge 10\}.
	\]
	Then there exists a GCD subgraph $G'=(\mu,\CV,\CW,\CE',\CP,f,g)$ of $G$ such that
	\[
	q(G')\ge \frac{q(G)}{2}>0\quad\text{and}\quad
	\CE'\subseteq \bigg\{(v,w)\in\CV\times\CW:
	\sum_{\substack{p|v w/\gcd(v,w)^2\\p\ge t,\ p\notin \CR(G)}}\frac{1}{p}\ge 5\bigg\}.
	\]
\end{lem}

\begin{lem}[Subgraph with high-degree vertices]\label{lem:HighDegreeSubgraph}
Let $G=(\mu,\CV,\CW,\CE,\CP,f,g)$ be a GCD graph with edge density $\delta>0$. Then there is a GCD subgraph $G'=(\mu,\CV',\CW',\CE',\CP,f,g)$ of $G$ with edge density $\delta'>0$ such that:
\begin{enumerate}
\item $q(G')\ge q(G)$;
\item $\delta'\ge \delta$;
\item For all $v\in\CV'$ and for all $w\in\CW'$, we have
\[
\mu(\Gamma_{G'}(v))\ge\frac{9\delta'}{10}\mu(\CW')
\quad\text{and}\quad \mu(\Gamma_{G'}(w))\ge\frac{9\delta'}{10}\mu(\CV').
\]
\end{enumerate}

\end{lem}

\begin{proof}[Proof of Proposition \ref{prop:BoundOnEdges} assuming Propositions \ref{prop:IterationStep1}-\ref{prop:SmallPrimes} and Lemmas \ref{lem:Cosmetic}-\ref{lem:HighDegreeSubgraph}] We will construct the required subgraph $G'$ in several stages. It suffices to produce a GCD subgraph $G'$ of $G$ satisfying only conclusions $(a)$ and $(d)$ of Proposition \ref{prop:BoundOnEdges}, since an application of Lemma \ref{lem:HighDegreeSubgraph} then produces a GCD subgraph satisfying all the conclusions.

\medskip

\noindent
\textit{Stage 1: Obtaining a GCD subgraph $G^{(1)}$ with $\CR(G^{(1)})\subseteq \{p> 10^{2000}\}$.}

\medskip

Since $G$ has set of primes equal to the empy set, we may apply Proposition \ref{prop:SmallPrimes} to $G$ to produce a GCD subgraph $G^{(1)}=(\mu,\CV^{(1)},\CW^{(1)},\CE^{(1)},\CP^{(1)},f^{(1)},g^{(1)})$ of $G$ with edge density $\delta^{(1)}$ and for which
\begin{equation}
\CR(G^{(1)})\subseteq\{p>10^{2000}\},
\quad q(G^{(1)})\ge\frac{q(G)}{10^{10^{3000}}}
\quad\text{and}\quad 
\delta^{(1)}q(G^{(1)})\ge \frac{\delta\cdot q(G)}{10^{10^{3000}}}.
\label{eq:G1Quality}
\end{equation}
In particular, we have
\begin{equation}\label{eq:R of subgraphs}
\CR(H)\subseteq \CR(G^{(1)})\subseteq \{p>10^{2000}\}
\end{equation}
for any $H\preceq G^{(1)}$ by Lemma \ref{lem:Trivial}(b).

\bigskip

\noindent
{\it Stage 2: Obtaining a GCD subgraph $G^{(2)}$ with $\CR^\flat(G^{(2)})=\emptyset$.}

\medskip

If $\CR^\flat(G^{(1)})\ne \emptyset$, then $G^{(1)}$ satisfies the conditions of Proposition \ref{prop:IterationStep1}. We then repeatedly apply Proposition \ref{prop:IterationStep1} to produce a sequence of GCD subgraphs of $G^{(1)}$ given by 
\[
G^{(1)}=:G^{(1)}_1\succeq G^{(1)}_2\succeq\cdots
\] 
until we obtain a GCD subgraph $G^{(2)}$ of $G^{(1)}$ which does not satisfy the conditions of Proposition \ref{prop:IterationStep1}. Since $\CR(G^{(1)}_{i+1})\subsetneq \CR(G^{(1)}_i)$ and $\CR(G^{(1)})$ is a finite set, this process must indeed terminate after a finite number of steps and produce a GCD graph $G^{(2)}:=(\mu,\CV^{(2)},\CW^{(2)},\CE^{(2)},\CP^{(2)},f^{(2)},g^{(2)})\preceq G^{(1)}$ that does not satisfy the conditions of Proposition \ref{prop:IterationStep1}. Since $\CR(G^{(2)})\subseteq \{p>10^{2000}\}$ by \eqref{eq:R of subgraphs}, it must be the case that
\[
\CR^\flat(G^{(2)})=\emptyset.
\]
In addition, Proposition \ref{prop:IterationStep1} implies that
\[
q(G^{(2)})\ge 2^N q(G^{(1)})
\quad\text{and}\quad  \delta^{(2)}q(G^{(2)})\ge 2^N\delta^{(1)}\, q(G^{(1)}),
\]
where 
\[
N=\#\{p\in\CP^{(2)}\setminus\CP^{(1)}:f^{(2)}(p)\neq g^{(2)}(p)\}.
\] 
Together with \eqref{eq:G1Quality}, this yields that
\begin{equation}
q(G^{(2)})\ge \frac{2^N}{10^{10^{3000}}} \cdot q(G)
\quad\text{and}\quad  
\delta^{(2)}q(G^{(2)})\ge 
	\frac{2^N}{10^{10^{3000}}} \cdot\delta\cdot q(G).
\label{eq:G2Quality} 
\end{equation}

On the other hand, if $\CR^\flat(G^{(1)})=\emptyset$, then we simply take $G^{(2)}=G^{(1)}$ and note that \eqref{eq:G2Quality} is trivially satisfied by \eqref{eq:G1Quality}.

\medskip

This completes Stage 2. The remaining part of the proof deviates according to whether the ratio $q(G^{(2)})/q(G)$ is larger or smaller than $(t/10)^{50}\delta/10^{10^{3000}}$.

\bigskip

\noindent 
{\it Case (a): $q(G^{(2)})/q(G)\ge (t/10)^{50}\delta/10^{10^{3000}}$.} 

\smallskip

In this case we do not need to keep track of the condition that $L_t(v,w)\ge10$ because we have a very large gain in the quality of the new graph. The next stage of the argument is then:

\bigskip

\noindent
{\it Stage 3a: Obtaining a GCD subgraph with} $\CR(G^{(3\text{a})})=\emptyset$.

\medskip

Notice that if $H\preceq G^{(2)}$, then $\CR(H)\subseteq\{p>10^{2000}\}$ by \eqref{eq:R of subgraphs}. Consequently, if $\CR(H)\neq\emptyset$, then either Proposition \ref{prop:IterationStep1} or Proposition \ref{prop:IterationStep2} is applicable to $H$, thus producing a GCD subgraph $H'$ of $H$ such that
\begin{equation}\label{eq:Prop10}
\CR(H')\subsetneq \CR(H)
\quad\text{and}\quad
q(H')\ge q(H).
\end{equation}
Since $\CR(G^{(2)})$ is finite, starting with $H_1=G^{(2)}$ and iterating the above fact, we can construct a finite sequence of GCD subgraphs 
\[
G^{(2)}=H_1\succeq H_2\succeq\cdots\succeq H_J=:G^{(3\text{a})}
\] 
such that
\[
\CR(G^{(3\text{a})})=\emptyset
\quad \text{and}\quad
q(G^{(3\text{a})})\ge q(G^{(2)}).
\]
Applying the assumption that $q(G^{(2)})/q(G)\ge (t/10)^{50}\delta/10^{10^{3000}}$, we infer that
\[
q(G^{(3\text{a})})\ge \Bigl(\frac{t}{10}\Bigr)^{50}\frac{\delta}{10^{10^{3000}}} \cdot q(G).
\]
Hence, the GCD graph $G'=G^{(3\text{a})}$ satisfies condition \ref{prop:Parta} and condition \ref{prop:Partd}-\ref{prop:Case1} of Proposition \ref{prop:BoundOnEdges}, giving the result in this case. (Recall that we may also guarantee conditions  \ref{prop:Partb} and \ref{prop:Partc} of Proposition \ref{prop:BoundOnEdges} by feeding our graph into Lemma \ref{lem:HighDegreeSubgraph}.)

\medskip

In order to complete the proof of Proposition \ref{prop:BoundOnEdges}, it remains to consider the situation when $q(G^{(2)})/q(G)$ is not large.

\bigskip

\noindent 
{\it Case (b): $q(G^{(2)})/q(G)<(t/10)^{50}\delta/10^{10^{3000}}$.} 

\medskip

In this case, the quality increment is small and we must make sure not to lose track of the condition $L_t(v,w)\ge10$. For this reason, we perform some cosmetic surgery to our graph before applying Proposition \ref{prop:IterationStep2}. This consists of {\it Stage 3b} that we present below.

\bigskip

\noindent 
{\it Stage 3b: Removing the effect of primes in $\CR(G^{(2)})$ from the anatomical condition $L_t(v,w)\ge10$.}

\medskip

Note that \eqref{eq:G2Quality} implies that
\[
\delta^{(2)}\ge\frac{\delta}{10^{10^{3000}}} \cdot \frac {q(G)}{q(G^{(2)})}\ge \Bigl(\frac{10}{t}\Bigr)^{50},
\]
and that 
\eq{\label{eq:N-bound}
2^N\le 10^{10^{3000}} \cdot \frac{q(G^{(2)})}{q(G)}\le \Bigl(\frac{t}{10}\Bigr)^{50}\delta\le t^{50},
}
where we recall that 
\[
N=\#\{p\in\CP^{(2)}\setminus\CP^{(1)}:f^{(2)}(p)\ne g^{(2)}(p)\}
\] 
(here we used the trivial bound $\delta\le 1$).

Since $\delta^{(2)}\ge (10/t)^{50}$, $\CR^\flat(G^{(2)})=\emptyset$, and $L_t(v,w)\ge10$ for all $(v,w)\in\CE^{(2)}$, it is the case that $G^{(2)}$ satisfies the conditions of Lemma \ref{lem:Cosmetic}. Consequently, there exists a GCD subgraph $G^{(3\text{b})}=(\mu,\CV^{(3\text{b})},\CW^{(3\text{b})},\CE^{(3\text{b})},\CP^{(3\text{b})},f^{(3\text{b})},g^{(3\text{b})})$  of $G^{(2)}$ with 
\begin{align}
\CP^{(3\text{b})}&=\CP^{(2)},\label{eq:P3b}\\
q(G^{(3\text{b})})&\ge \frac{q(G^{(2)})}{2},
\label{eq:G3Quality}
\end{align}
and such that 
\begin{equation}
\sum_{\substack{p|v w/\gcd(v,w)^2\\p\ge t,\ p\notin \CR(G^{(2)})}}\frac{1}{p}\ge 5
\quad\text{whenever}\quad
(v,w)\in\CE^{(3\text{b})},
\label{eq:Anat1}
\end{equation}

We claim that an inequality of the form \eqref{eq:Anat1} holds even if we remove from consideration the primes lying in the set
\[
\CP^{(2)}_{\text{diff}}:=\{p\in\CP^{(2)}:\,f^{(2)}(p)\ne g^{(2)}(p)\}.
\]
It turns out that we can do this rather crudely, starting from the estimate
\[
\sum_{\substack{p|v w/\gcd(v,w)^2\\p\ge t,\ p\in \CP^{(2)}_{\text{diff}}}}\frac{1}{p}
\le \frac{\#(\CP_{\text{diff}}^{(2)}\cap\{p\ge t\})}{t}.
\]
Recalling that $t>10^{2000}$ and $\CP^{(1)}\subseteq\{p\le10^{2000}\}$, we deduce that
\begin{equation}\label{eq:removing P^2_diff}
\sum_{\substack{p|v w/\gcd(v,w)^2\\p\ge t,\ p\in \CP^{(2)}_{\text{diff}}}}\frac{1}{p}
\le \frac{\#(\CP_{\text{diff}}^{(2)}\setminus\CP^{(1)} )}{t}=\frac{N}{t}.
\end{equation}
Since $t\ge 10^{2000}$, relation \eqref{eq:N-bound} implies that $N\le 2\log(t^{50})=100\log t<t$, that is to say the right hand side of \eqref{eq:removing P^2_diff} is $\le 1$. As a consequence,
\begin{equation}
\sum_{\substack{p|v w/\gcd(v,w)^2\\ p\ge t,\ p\notin \CR(G^{(2)})\cup \CP^{(2)}_{\text{diff}}}}\frac{1}{p}\ge 4
\quad\text{whenever}\quad (v,w)\in\CE^{(3\text{b})}.
\label{eq:Anat4}
\end{equation}
Having removed the effect to the condition $L_t(v,w)\ge10$ of primes from the sets $\CR(G^{(2)})\cup\CP_{\text{diff}}^{(2)}$, we are ready to complete the construction of $G'$ in Case (b).

\bigskip

\noindent 
{\it Stage 4b: Obtaining a GCD subgraph with} $\CR(G^{(4\text{b})})=\emptyset$.

\medskip

We argue as in Stage 3a: for each $H\preceq G^{(3\text{b})}$, we have  $\CR(H)\subseteq\{p>10^{2000}\}$ by \eqref{eq:R of subgraphs}. Hence, if $\CR(H)\neq\emptyset$, then either Proposition \ref{prop:IterationStep1} or Proposition \ref{prop:IterationStep2} is applicable to $H$, thus producing a GCD subgraph $H'$ of $H$ such that
\begin{equation}\label{eq:Prop11}
\CR(H')\subsetneq \CR(H),\quad
\CP_{H'}\subseteq \CR(H)\cup \CP_H ,
\quad\text{and}\quad
q(H')\ge q(H),
\end{equation}
where $\CP_H$ and $\CP_{H'}$ denote the set of primes of $H$ and of $H'$, respectively. 
Since $\CR(G^{(3\text{b})})$ is finite, starting with $H_1=G^{(3\text{b})}$ and iterating the above fact, we can construct a finite sequence of GCD subgraphs 
\[
G^{(3\text{b})}=H_1\succeq H_2\succeq\cdots\succeq H_J=:G^{(4\text{b})}
\] 
such that
\[
\CR(G^{(4\text{b})})=\emptyset\quad\text{and}\quad 
q(G^{(4\text{b})})\ge q(G^{(3\text{b})}).
\]
In addition, note that 
\[
\CP^{(4\text{b})}
	\subseteq \CR(G^{(3\text{b})})\cup\CP^{(3\text{b})}
	\subseteq \CR(G^{(2)})\cup\CP^{(2)},
\]
where the second relation follows by fact \eqref{eq:P3b} that $\CP^{(3\text{b})}=\CP^{(2)}$. 
We now verify that if we let 
\[
G'=G^{(4b)},
\]
then condition \ref{prop:Partd}-\ref{prop:Case2} of Proposition \ref{prop:BoundOnEdges} is satisfied. This suffices for the completion of the proof, since $G'$ clearly satisfies condition \ref{prop:Parta} of Proposition \ref{prop:BoundOnEdges}, and an application of Lemma \ref{lem:HighDegreeSubgraph} can also ensure conditions \ref{prop:Partb} and \ref{prop:Partc}.

First of all, note that by \eqref{eq:G3Quality} and \eqref{eq:G2Quality} and $q(G^{(4\text{b})})\ge q(G^{(3\text{b})})$, we have
\[
q(G') = q(G^{(4\text{b})})\ge q(G^{(3\text{b})})\ge \frac{q(G^{(2)})}{2} \ge \frac{q(G)}{2\cdot 10^{10^{3000}}}.
\]
Let $(v,w)\in\CE^{(4\text{b})}$. It remains to check that $L_t(v',w')\ge4$, where $v'$ and $w'$ are defined by the relations
\[
v=v'\prod_{p\in\CP^{(4\text{b})}}p^{f^{(4\text{b})}(p)}
\quad\text{and}\quad 
w=w'\prod_{p\in\CP^{(4\text{b})}}p^{g^{(4\text{b})}(p)}.
\]
By the definition of the set $\CR(G^{(4\text{b})})$ and since $\CR(G^{(4b)})=\emptyset$, all prime factors of $\gcd(v,w)$ belong to $\CP^{(4\text{b})}$. But for each prime $p\in\CP^{(4\text{b})}$ we have $p^{\min\{f^{(4\text{b})}(p),g^{(4\text{b})}(p)\}}\|\gcd(v,w)$. Thus
\[
\gcd(v,w) = \prod_{p\in\CP^{(4\text{b})}}p^{\min\{f^{(4\text{b})}(p),g^{(4\text{b})}(p)\}}.
\]
In particular, we must have that 
\[
\gcd(v',w')=1.
\]

Now, let $p$ be a prime such that
\[
p|\frac{vw}{\gcd(v,w)^2}\quad\text{and}\quad p\nmid v'w'.
\]
Since $p\nmid v'w'$ but $p|vw$, we must have $p\in\CP^{(4\text{b})}$, and so $p^{\min\{f^{(4\text{b})}(p),g^{(4\text{b})}(p)\}}\|\gcd(v,w)$. In addition, our assumptions that $p\nmid v'w'$ and $p|v w/\gcd(v,w)^2$ imply that $f^{(4\text{b})}(p)\ne g^{(4\text{b})}(p)$. 
If $p\in\CP^{(2)}$, we infer that $p\in \CP^{(2)}_{\text{diff}}$. On the other hand, if $p\notin\CP^{(2)}$, then the inclusion $\CP^{(4\text{b})}\subseteq \CP^{(2)}\cup\CR(G^{(2)})$ implies that $p\in \CR(G^{(2)})$. In either case, we have that $p\in\CP^{(2)}_{\text{diff}}\cup\CR(G^{(2)})$. Thus, since $\CE^{(4\text{b})}\subseteq\CE^{(3\text{b})}$, we may use the bound \eqref{eq:Anat4}, which gives
\[
L_t(v',w')=\sum_{\substack{p|v' w'/\gcd(v',w')^2\\ p\ge t}}\frac{1}{p}\ge\sum_{\substack{p|v w/\gcd(v,w)^2\\p\ge t,\ p\notin \CR(G^{(2)})\cup \CP^{(2)}_{\text{diff}}}}\frac{1}{p}\ge 4.
\]
In particular, $G'=G^{(4\text{b})}$ satisfies the conditions of case \ref{prop:Partd}-\ref{prop:Case2} of Proposition \ref{prop:BoundOnEdges}. This completes the proof of Proposition \ref{prop:BoundOnEdges} in Case (b) too.
\end{proof}

Thus we are left to establish Propositions \ref{prop:IterationStep1}-\ref{prop:SmallPrimes} and Lemmas \ref{lem:Cosmetic}-\ref{lem:HighDegreeSubgraph}. We begin with the last two results because they are easier to establish.

\section{Proof of Lemma \ref{lem:Cosmetic}}\label{sec:CosmeticProof}

In this section we establish Lemma \ref{lem:Cosmetic} directly.

For brevity, let
\[
S(v,w) = \sum_{\substack{p|v w/\gcd(v,w)^2\\p\in \CR(G),\ p\ge t^{50}}}\frac{1}{p}.
\]
We have
	\[
	\sum_{(v,w)\in\CE } \mu(v)\mu(w) S(v,w) 
	=   	\sum_{\substack{p\in \CR(G) \\ p\ge t^{50}}}\frac{1}{p}\cdot 
\mu\Bigl(\Bigl\{(v,w)\in\CE: p|\frac{v w}{\gcd(v,w)^2}\Bigr\}\Bigr) .
	\]
	Fix for the moment a prime $p\in\CR(G)$. Since we have $\CR^\flat(G)=\emptyset$, it must be the case that $p\in\CR^\sharp(G)$, that is to say there exists some $k\in\Z_{\ge0}$ such that 
	\[
	\mu(\CV_{p^k})\ge\Bigl(1-\frac{10^{40}}{p}\Bigr)\mu(\CV)
	\quad\text{and}\quad 
	\mu(\CW_{p^k})\ge\Bigl(1-\frac{10^{40}}{p}\Bigr)\mu(\CW).
	\]
	Now we note that if $p|v w/\gcd(v,w)^2$, then $p^j\|v$ and $p^\ell\|w$ for some $j\ne \ell$. In particular we cannot have $p^k \|v$ and $p^k\|w$.  Thus
	\als{
		\mu\Bigl(\Bigl\{(v,w)\in\CE: p|\frac{v w}{\gcd(v,w)^2}\Bigr\}\Bigr)&\le \mu((\CV\setminus\CV_{p^k})\times \CW)+\mu(\CV\times( \CW\setminus\CW_{p^k}))\\
		&	\le 2\cdot \frac{10^{40}}{p} \cdot \mu(\CV)\mu(\CW) .
	}
Thus we conclude that
	\als{
		\sum_{(v,w)\in\CE } \mu(v)\mu(w) S(v,w) 
		&\le  \sum_{p\ge t^{50}}\frac{2\cdot 10^{40} \mu(\CV)\mu(\CW) }{p^2}\\
		&\le \frac{2\cdot 10^{40}\mu(\CV)\mu(\CW) }{t^{50}}\\
		&<\frac{\mu(\CE) } {100}, 
}
where in the final line we used the fact that $\delta=\mu(\CE)/\mu(\CV)\mu(\CW)\ge (10/t)^{50}$. 

Now, let us define
\[
\CE':=\{(v,w)\in\CE:  S(v,w)\le 1\}.
\]
Evidently, we have that
\[
\mu(\CE\setminus\CE')=\sum_{\substack{(v,w)\in\CE \\ S(v,w)>1}}\mu(v)\mu(w)\le \sum_{(v,w)\in\CE } \mu(v)\mu(w) S(v,w) <\frac{\mu(\CE)}{100}.
\]
Thus $\mu(\CE')\ge 99\mu(\CE)/100$. We then take $G':=(\mu,\CV,\CW,\CE',\CP,f,g)$ and note that 
\[
\frac{q(G')}{q(G)}=\Bigl(\frac{\mu(\CE')}{\mu(\CE)}\Bigr)^{10}\ge \frac{1}{2} .
\]
Finally, we note that 
\[
\sum_{t\le p\le t^{50}}\frac{1}{p}= \lim_{y\to t^-}\sum_{y<p\le t^{50}}\frac{1}{p}\le \log(50)+\frac{1}{(\log t)^2}  \le 4
\]
for $t\ge 300$ by \cite[Theorem 5]{RS}, and so if $(v',w')\in\CE'$ then
\[
\sum_{\substack{p|v' w'/\gcd(v',w')^2\\p\in \CR(G)\\ p\ge t}}\frac{1}{p}\le  \sum_{\substack{p|v' w'/\gcd(v',w')^2\\p\in \CR(G)\\ p\ge t^{50}}}\frac{1}{p}+4\le 5.
\]
Hence, since $\CE'\subseteq\CE\subseteq \{(v,w)\in \CV\times\CW: L_t(v,w)\ge 10\}$, for any $(v',w')\in\CE'$ we have
\[
	\sum_{\substack{p|v' w'/\gcd(v',w')^2\\p\notin \CR(G)\\ p\ge t}}\frac{1}{p}\ge L_t(v',w')-5\ge 5.
\]
This completes the proof of Lemma \ref{lem:Cosmetic}.\qed

\medskip

We are left to establish Propositions \ref{prop:IterationStep1}-\ref{prop:SmallPrimes} and Lemma \ref{lem:HighDegreeSubgraph}.

\section{Proof of Lemma \ref{lem:HighDegreeSubgraph}}\label{sec:HighDegree}

In this section we establish Lemma \ref{lem:HighDegreeSubgraph}. We begin with an auxiliary lemma.

\begin{lem}[Quality increment or all vertices have high degree]\label{lem:HighDegree}
Let	$G=(\mu,\CV,\CW,\CE,\CP,f,g)$ be a GCD graph with edge density $\delta>0$. For each $v\in\CV$ and for each $w\in\CW$, we let
\[
\Gamma_G(v):=\{w\in\CW:\,(v,w)\in\CE\}
\quad\text{and}\quad
\Gamma_G(w):=\{v\in\CV:\,(v,w)\in\CE\}
\]
be the sets of their neighbours. Then one of the following holds:
\begin{enumerate}
\item For all $v\in\CV$ and for all $w\in\CW$, we have
\[
\mu(\Gamma_G(v))\ge\frac{9\delta}{10}\mu(\CW)
\quad\text{and}\quad \mu(\Gamma_G(w))\ge\frac{9\delta}{10}\mu(\CV).
\]
\item There is a GCD subgraph $G'=(\mu,\CV',\CW',\CE',\CP,f,g)$ of $G$ with edge density $\delta'\ge \delta$, quality $q(G')\ge q(G)$, and such that either $\CV'\subsetneq\CV$ or $\CW'\subsetneq \CW$.
\end{enumerate}
\end{lem}

\begin{proof} Assume that (a) fails. Then either its first or its second inequality fails. Assume that the first one fails for some $v\in\CV$; the other case is entirely analogous. Let $\CE'$ be the set of edges between the vertex sets $\CV\setminus\{v\}$ and $\CW$. Note that 
\[
\mu(\CE')= \mu(\CE)-\mu(v)\mu(\Gamma_G(v)) >0
\]
because $\mu(\Gamma_G(v))<9\delta\mu(\CW)/10$, $\mu(v)\le \mu(\CV)$, and $\mu(\CE)>0$ by the assumption $\delta>0$. In particular, we have $\mu(\CW),\mu(\CV\setminus\{v\})>0$. We then consider $G'=(\mu,\CV\setminus\{v\},\CW,\CE',\CP,f,g)$, which is a GCD subgraph of $G$. Let $G'$ have edge density $\delta'$. We claim that $\delta'\ge \delta$ and $q(G')\ge q(G)$.

Indeed, we have
\als{
\mu(\CE')
	= \mu(\CE)-\mu(v)\mu(\Gamma_G(v)) 
	&\ge \delta \mu(\CV)\mu(\CW)-\frac{9\delta}{10}\mu(v)\mu(\CW) \\
	&=\delta \big(\mu(\CV)-\mu(v)\big)\mu(\CW)\cdot\Big(1+\frac{\mu(v)/10}{\mu(\CV)-\mu(v)}\Big).
}
Thus the edge density $\delta'$ of $G'$ satisfies
\[
\delta'=\frac{\mu(\CE')}{\mu(\CV\setminus\{v\})\mu(\CW)}
		=\frac{\mu(\CE')}{\big(\mu(\CV)-\mu(v)\big)\mu(\CW)}
		\ge \delta\cdot \Bigl(1+\frac{\mu(v)/10}{\mu(\CV)-\mu(v)}\Bigr).
\]
Thus we see that $\delta'\ge \delta$, and that 
\begin{align*}
(\delta')^{10}\mu(\CV\setminus\{v\})\mu(\CW)
	&\ge \delta^{10}\big(\mu(\CV)-\mu(v)\big)\mu(\CW)
		\Bigl(1+\frac{\mu(v)}{\mu(\CV)-\mu(v)}\Bigr)
		= \delta^{10} \mu(\CV)\mu(\CW) .
\end{align*}
This proves our claim that $q(G')\ge q(G)$ too, thus completing the proof of the lemma.
\end{proof}
\begin{proof}[Proof of Lemma \ref{lem:HighDegreeSubgraph}]
We note that conclusion $(c)$ of Lemma \ref{lem:HighDegreeSubgraph} is the same as conclusion $(a)$ of Lemma \ref{lem:HighDegree}. Thus, if $G$ does not satisfy conclusion $(c)$ of Lemma \ref{lem:HighDegreeSubgraph}, then we may repeatedly apply Lemma \ref{lem:HighDegree} to produce a sequence of GCD subgraphs 
\[
G=:G_1\succeq G_2\succeq\cdots 
\]
until we arrive at a GCD subgraph of $G$ which satisfies conclusion $(a)$ of Lemma \ref{lem:HighDegree}. This process must terminate after a finite number of steps since at least one of the vertex sets of $G_{i+1}$ has one less element than the corresponding vertex set of $G_i$. Let the process terminate at $G_J$, which satisfies conclusion $(a)$ of Lemma \ref{lem:HighDegree}, and let $\delta_i$ be the edge density of $G_i$. Since $\delta_{i+1}\ge \delta_i$ and $q(G_{i+1})\ge q(G_i)$ by Lemma \ref{lem:HighDegree}, we have that 
\[
\delta_J\ge \delta_{J-1}\ge \dots \ge \delta_1=\delta,\quad \text{and}\quad q(G_J)\ge q(G_{J-1})\ge \dots\ge q(G_1)=q(G).
\]
Since the multiplicative data are also maintained at each iteration, we see that taking $G'=G_J$ gives the result.
\end{proof}

Thus we are left to establish Propositions \ref{prop:IterationStep1}-\ref{prop:SmallPrimes}.

\section{Preparatory Lemmas on GCD graphs}\label{sec:Prep}

Our remaining task is to prove Propositions \ref{prop:IterationStep1}-\ref{prop:SmallPrimes}. Before we attack these directly, we establish various preliminary results about GCD graphs in this section, which we will then use in the remaining sections to prove Propositions \ref{prop:IterationStep1}-\ref{prop:SmallPrimes}.

\begin{lem}[Quality variation for special GCD subgraphs]\label{lem:InducedGraphs}
Let $G=(\mu,\CV,\CW,\CE,\CP,f,g)$ be a GCD graph, $p\in\CR(G)$ and $k,\ell\in\Z_{\ge0}$. If $G_{p^k,p^\ell}$ is as in Definition \ref{def:special graphs}, then $G_{p^k,p^\ell}$ is a GCD subgraph of $G$. In addition, if $G$ is non-trivial and $\mu(\CV_{p^k}),\mu(\CW_{p^\ell})>0$, then we have
\[
\frac{q(G_{p^k,p^\ell})}{q(G)}
	=\bigg(\frac{\mu(\CE_{p^k,p^\ell})}{\mu(\CE)}\bigg)^{10}
	 \bigg(\frac{\mu(\CV)}{\mu(\CV_{p^k})}\bigg)^{9}
	 \bigg(\frac{\mu(\CW)}{\mu(\CW_{p^\ell})}\bigg)^{9}
	 \frac{p^{|k-\ell|}}{(1-\un_{k=\ell\ge1}/p)^2(1-1/p^{31/30})^{10}}.
\]
\end{lem}
\begin{proof}
This follows directly from the definitions.
\end{proof}

\begin{lem}[One subgraph must have limited quality loss]\label{lem:Pigeonhole}
Let $G=(\mu,\CV,\CW,\CE,\CP,f,g)$ be a GCD graph with edge density $\delta>0$, and let $\CV=\CV_1\sqcup\dots \sqcup\CV_I$ and $\CW=\CW_1\sqcup\dots \sqcup\CW_J$ be partitions of $\CV$ and $\CW$. Then there is a GCD subgraph $G'=(\mu,\CV',\CW',\CE',\CP,f,g)$ of $G$ with edge density $\delta'>0$ such that
\[
q(G')\ge \frac{q(G)}{(I J)^{10}},\qquad \delta'\ge\frac{\delta}{I J},
\]
and with $\CV'\in\{\CV_1,\dots,\CV_I\}$, $\CW'\in\{\CW_1,\dots,\CW_J\}$, and $\CE'=\CE\cap(\CV'\times\CW')$.
\end{lem}
\begin{proof}
For brevity let $\CE_{i,j}=\CE\cap(\CV_i\times\CW_j)$ be the edges between $\CV_i$ and $\CW_j$ for $i\in\{1,\dots,I\}$ and $j\in \{1,\dots,J\}$. Since the partitions of $\CV$ and $\CW$ induce a partition $\CE=\sqcup_{i=1}^I\sqcup_{j=1}^{J}\CE_{i,j}$ of $\CE$, we have
\[
\mu(\CE)=\sum_{i=1}^I\sum_{j=1}^J \mu(\CE_{i,j}).
\]
Thus, by the pigeonhole principle, there is a choice of $i_0$ and $j_0$ such that $\mu(\CE_{i_0,j_0})\ge \mu(\CE)/(IJ)>0$. We then let $G'=(\mu,\CV_{i_0},\CW_{j_0},\CE_{i_0,j_0},\CP,f,g)$, which is clearly a non-trivial GCD subgraph of $G$. We see that
\[
\frac{\delta'}{\delta}=\Bigl(\frac{\mu(\CE_{i_0,j_0})}{\mu(\CE)}\Bigr)\Bigl(\frac {\mu(\CV)}{\mu(\CV_{i_0})}\Bigr)\Bigl(\frac{\mu(\CW)}{\mu(\CW_{j_0})}\Bigr)\ge \frac{\mu(\CE_{i_0,j_0})}{\mu(\CE)}\ge \frac{1}{I J}
\]
and
\[
\frac{q(G')}{q(G)}
	=\bigg(\frac{\mu(\CE_{i_0,j_0})}{\mu(\CE)}\bigg)^{10}
	 \bigg(\frac{\mu(\CV)}{\mu(\CV_{i_0})}\bigg)^{9}
	 \bigg(\frac{\mu(\CW)}{\mu(\CW_{j_0})}\bigg)^{9}
\ge 	\bigg(\frac{\mu(\CE_{i_0,j_0})}{\mu(\CE)}\bigg)^{10} \ge \frac{1}{(I J)^{10}}.
\]
This gives the result.
\end{proof}

\begin{lem}[Few edges between unbalanced sets, I]\label{lem:UnbalancedSetEdges1}
Let	$G=(\mu,\CV,\CW,\CE,\CP,f,g)$ be a GCD graph with edge density $\delta>0$. Let $p\in\CR(G)$, $r\in\Z_{\ge1}$ and $k\in\mathbb{Z}_{\ge 0}$ be such that $p^r>10^{2000}$ and
\[
\frac{\mu(\CW_{p^k})}{\mu(\CW)}\ge 1-\frac{10^{40}}{p}.
\]
(In particular, if $p\le 10^{40}$, the last hypothesis is vacuous.) 

If we set $\CL_{k,r}=\{\ell\in\Z_{\ge0}:|\ell-k|\ge r+1\}$ and write $\delta_{p^k,p^\ell}$ for the edge density of the graph $G_{p^k,p^\ell}$, then one of the following holds:
\begin{enumerate}
\item There is $\ell\in\CL_{k,r}$ such that $q(G_{p^k,p^\ell})>2q(G)$ and $\delta_{p^k,p^\ell}q(G_{p^k,p^\ell})>2\delta q(G)>0$.
\item $\sum_{\ell\in\CL_{k,r}}  \mu(\CE_{p^k,p^\ell})\le\mu(\CE)/(4p^{31/30})$.
\end{enumerate}
\end{lem}

\begin{proof} Assume that conclusion $(b)$ does not hold, so $\sum_{\ell\in\CL_{k,r}}  \mu(\CE_{p^k,p^\ell})>\mu(\CE)/(4p^{31/30})$ and we wish to establish $(a)$. Then there must exist some $\ell\in\CL_{k,r}$ such that 
\[
\mu(\CE_{p^k,p^\ell}) > \frac{\mu(\CE)}{300\cdot2^{|k-\ell|/20}p^{31/30}} >0 ,
\]
where we used that $\sum_{|j|\ge 0}2^{-|j|/20}\le 2/(1-2^{-1/20})\le 60$. In particular, $G_{p^k,p^\ell}$ is a non-trivial GCD graph. Since $\mu(\CW_{p^k})\ge (1-10^{40}/p)\mu(\CW)$, we have that $\mu(\CW_{p^\ell})\le 10^{40}\mu(\CW)/p$. Consequently,
\begin{align*}
\frac{q(G_{p^k,p^\ell})}{q(G)}
	&= \bigg(\frac{\mu(\CE_{p^k,p^\ell})}{\mu(\CE)}\bigg)^{10}
	\bigg(\frac{\mu(\CV)}{\mu(\CV_{p^k})}\bigg)^{9}
	\bigg(\frac{\mu(\CW)}{\mu(\CW_{p^\ell})}\bigg)^{9} \frac{p^{|k-\ell|}}{(1-1/p^{31/30})^{10}} \\
	&\ge \Big(\frac{1}{300\cdot 2^{|k-\ell|/20}p^{31/30}}\Big)^{10}
	\Big(\frac{p}{10^{40}}\Big)^{9} p^{|k-\ell|} \\
	&\ge \frac{p^{-4/3}(p/2^{1/2})^{|k-\ell|}}{10^{25}10^{40\cdot 9}}.
\end{align*}
Since $|k-\ell|\ge r+1\ge 2r/3+4/3$, we have 
\[
p^{-4/3}(p/2^{1/2})^{|k-\ell|}\ge 2^{-4/3}(p/2^{1/2})^{2r/3}.
\]
In addition, note that $(p/2^{1/2})\ge p^{1/2}$ for all primes. Therefore
\begin{align*}
\frac{q(G_{p^k,p^\ell})}{q(G)}	&\ge \frac{p^{r/3}}{2^{4/3}\cdot 10^{385}}
	>2
\end{align*}
by our assumption that $p^r>10^{2000}$. 

Similarly, we have 
\begin{align*}
\frac{\delta_{p^k,p^\ell}}{\delta}\cdot \frac{q(G_{p^k,p^\ell})}{q(G)}
&= \bigg(\frac{\mu(\CE_{p^k,p^\ell})}{\mu(\CE)}\bigg)^{11}
\bigg(\frac{\mu(\CV)}{\mu(\CV_{p^k})}\bigg)^{10}
\bigg(\frac{\mu(\CW)}{\mu(\CW_{p^\ell})}\bigg)^{10} \frac{p^{|k-\ell|}}{(1-1/p^{31/30})^{10}} \\
&\ge \bigg(\frac{1}{300\cdot 2^{|k-\ell|/20}p^{31/30}}\bigg)^{11}
\bigg(\frac{p}{10^{40}}\bigg)^{10} p^{|k-\ell|} \\
&\ge \frac{p^{-41/30}(p/2^{11/20})^{|k-\ell|}}{10^{428}}\\
&\ge\frac{(p/2^{11/20})^{19r/30}}{2^{41/30\cdot 11/20}\cdot 10^{428}}.
\end{align*}
Since $p/2^{11/20}\ge p^{9/20}$ and $p^r>10^{2000}$, we conclude that 
\[
\frac{\delta_{p^k,p^\ell}}{\delta}\cdot \frac{q(G_{p^k,p^\ell})}{q(G)}>2.
\]
This completes the proof of the lemma.
\end{proof}

The symmetric version of Lemma \ref{lem:UnbalancedSetEdges1} to the above one also clearly holds:

\begin{lem}[Few edges between unbalanced sets, II]\label{lem:UnbalancedSetEdges2}
Let $G=(\mu,\CV,\CW,\CE,\CP,f,g)$ be a GCD graph with edge density $\delta>0$. Let $p\in\CR(G)$,  $r\in\mathbb{Z}_{\ge 1}$ and $\ell\in\Z_{\ge0}$ be such that $p^r>10^{2000}$ and
\[
\frac{\mu(\CV_{p^\ell})}{\mu(\CV)}\ge 1-\frac{10^{40}}{p},
\]
and set $\CK_{\ell,r}=\{k\in\Z_{\ge0}:|\ell-k|\ge r+1\}$. If $\delta_{p^k,p^\ell}$ denotes the edge density of the graph $G_{p^k,p^\ell}$, then one of the following holds:
\begin{enumerate}
\item There is $k\in\CK_{\ell,r}$ such that $q(G_{p^k,p^\ell})>2q(G)$ and $\delta_{p^k,p^\ell}q(G_{p^k,p^\ell})>2\delta q(G)>0$.
\item $\sum_{k\in\CK_{\ell,r}}  \mu(\CE_{p^k,p^\ell})\le\mu(\CE)/(4p^{31/30})$.
\end{enumerate}
\end{lem}

Next, we prove a lemma about the connectivity of small vertex sets of a GCD graph.

\begin{lem}[Few edges between small sets]\label{lem:SmallSetEdges}
Let	$G=(\mu,\CV,\CW,\CE,\CP,f,g)$ be a GCD graph with edge density $\delta>0$ and let $\eta\in(0,1)$. Then one of the following holds:
\begin{enumerate}
	\item For all sets $\CA\subseteq\CV$ and $\CB\subseteq\CW$ such that $\mu(\CA)\le \eta\cdot  \mu(\CV)$ and  $\mu(\CB)\le\eta \cdot \mu(\CW)$, we have $\mu(\CE\cap(\CA\times\CB))\le \eta^{9/5}\cdot \mu(\CE)$.
	\item There is a GCD subgraph $G'=(\mu,\CV',\CW',\CE',\CP,f,g)$ of $G$ such that $q(G')>q(G)$,  $\CV'\subsetneq \CV$ and $\CW'\subsetneq \CW$.
\end{enumerate}
\end{lem}

\begin{proof}
Assume that (a) fails. Hence, there exist sets $\CA\subseteq\CV$ and $\CB\subseteq\CW$ such that $\mu(\CA)\le\eta \cdot \mu(\CV)$, $\mu(\CB)\le\eta \cdot \mu(\CW)$ and $\mu(\CE\cap(\CA\times\CB))>\eta^{9/5}\cdot \mu(\CE)$. We then set $\CE'=\CE\cap(\CA\times \CB)$ and consider the GCD subgraph $G'=(\mu,\CA,\CB,\CE',\CP,f,g)$ of $G$. Since $\mu(\CE')>0$, this is a non-trivial GCD graph. In addition, since $\mu(\CV)>0$ (because $G$ is non-trivial) and $\eta<1$ (by assumption), we have $\mu(\CA)\le \eta \mu(\CV)<\mu(\CV)$, and thus $\CA\subsetneq \CV$. Similarly, we find that $\CB\subsetneq \CW$. Finally, for the quality of $G'$, we have
\[
\frac{q(G')}{q(G)}
	= \bigg(\frac{\mu(\CE')}{\mu(\CE)}\bigg)^{10}
	\bigg(\frac{\mu(\CV)}{\mu(\CA)}\bigg)^{9}
	\bigg(\frac{\mu(\CW)}{\mu(\CB)}\bigg)^{9} >\frac{(\eta^{9/5})^{10}}{\eta^9\cdot\eta^9}=1.
\]
This completes the proof of the lemma.
\end{proof}

By iterating this lemma, we arrive at the following result.

\begin{lem}[Subgraph with few edges between all small sets]\label{lem:NoSmallSetEdges}
		Let	$G=(\mu,\CV,\CW,\CE,\CP,f,g)$ be a GCD graph with edge density $\delta>0$, and let $\eta\in(0,1)$. Then there is a GCD subgraph $G'=(\mu,\CV',\CW',\CE',\CP,f,g)$ of $G$ with edge density $\delta'>0$ such that both of the following hold:
		\begin{enumerate}
			\item$q(G')\ge q(G)>0$.
			\item For all sets $\CA\subseteq\CV'$ and $\CB\subseteq\CW'$ such that $\mu(\CA)\le \eta\cdot \mu(\CV')$ and $\mu(\CB)\le \eta\cdot \mu(\CW')$, we have $\mu(\CE'\cap(\CA\times\CB))\le \eta^{9/5}\mu(\CE')$.
		\end{enumerate}
\end{lem}

\begin{proof} 
We note that conclusion $(b)$ of Lemma \ref{lem:NoSmallSetEdges} is the same as conclusion $(a)$ of Lemma \ref{lem:SmallSetEdges}. Thus, if $G$ does not satisfy conclusion $(b)$ of Lemma \ref{lem:NoSmallSetEdges}, then we may repeatedly apply Lemma \ref{lem:SmallSetEdges} to produce a sequence of GCD subgraphs 
	\[
	G=:G_1\succeq G_2\succeq\cdots 
	\]
	until we arrive at a GCD subgraph of $G$ which satisfies conclusion $(a)$ of Lemma \ref{lem:SmallSetEdges}. This process must terminate after a finite number of steps since $G_{i+1}$ has strictly smaller vertex sets than those of $G_i$. Let the process terminate at $G_J$, which satisfies conclusion $(a)$ of Lemma \ref{lem:SmallSetEdges}. Since $q(G_{i+1})>q(G_i)$ by Lemma \ref{lem:SmallSetEdges}, we have that 
	\[
	q(G_J)>q(G_{J-1})>\dots>q(G_1)=q(G).
	\]
	Lastly, since the multiplicative data are maintained at each iteration, we see that taking $G'=G_J$ gives the result.
\end{proof}

\section{Proof of Proposition \ref{prop:IterationStep1}} \label{sec:IterationStep1}

In this section we prove Proposition \ref{prop:IterationStep1}, which is the iteration procedure for `generic' primes. This section is essentially self-contained (relying only on the notation of Section \ref{sec:Graph} and the trivial Lemma \ref{lem:InducedGraphs}), and serves as a template for the proofs of the harder Propositions \ref{prop:IterationStep2} and \ref{prop:SmallPrimes}.

\begin{lem}[Bounds on edge sets]\label{lem:EdgeSets}
Consider a GCD graph $G=(\mu,\CV,\CW,\CE,\CP,f,g)$ and a prime $p\in\CR(G)$. 
For each $k,\ell\in\Z_{\ge0}$, let
\[
\alpha_k=\frac{\mu(\CV_{p^k})}{\mu(\CV)}
\quad\text{and}\quad
\beta_\ell=\frac{\mu(\CW_{p^\ell})}{\mu(\CW)}.
\]
Then there exist $k,\ell\in \Z_{\ge0}$ such that $\alpha_k,\beta_\ell>0$ and
\[
\frac{\mu(\CE_{p^k,p^\ell})}{\mu(\CE)} 	
	\ge 	
		\begin{cases}
		(\alpha_k\beta_k)^{9/10}
						&\text{if}\ k=\ell,\\
						\\
			\ds\frac{\alpha_k(1-\beta_k)
				+\beta_k(1-\alpha_k)
				+\alpha_{\ell}(1-\beta_\ell)
				+\beta_\ell(1-\alpha_\ell)}{2^{|k-\ell|/20}\times 1000}
						&\text{otherwise}.
		\end{cases}
\]
\end{lem}

\begin{proof} Let $\CX=\{(k,\ell)\in\Z_{\ge0}^2: \alpha_k,\beta_\ell>0\}$. Note that if $(k,\ell)\in\Z_{\ge0}^2\setminus \CX$, then $\mu(\CE_{p^k,p^\ell})\le \mu(\CV_{p^k})\mu(\CW_{p^\ell})=\alpha_k\beta_\ell \mu(\CV)\mu(\CW)=0$. Thus $\sum_{(k,\ell)\in \CX}\mu(\CE_{p^k,p^\ell}) = \mu(\CE)$. Hence, if we assume that the inequality in the statement of the lemma does not hold for any pair $(k,\ell)\in \CX$, we must have
\[
1=\sum_{(k,\ell)\in\CX}\frac{\mu(\CE_{p^k,p^\ell})}{\mu(\CE)}< S_1+S_2,
\]
	where
	\[
	S_1:=\sum_{k=0}^\infty(\alpha_k\beta_k)^{9/10}
	\]
	and
	\[
	S_2:=\sum_{\substack{k,\ell\ge0\\ k\neq\ell}}
		\frac{\alpha_k(1-\beta_k)
			+\beta_k(1-\alpha_k)
			+\alpha_\ell(1-\beta_\ell)
			+\beta_\ell(1-\alpha_\ell)}{2^{|k-\ell|/20}\times 1000}.
	\]
Thus, to arrive at a contradiction, it suffices to show that
	\[
	S_1+S_2\le 1.
	\]
First of all, note that $\sum_{|j|\ge1}2^{-|j|/20}=2/(2^{1/20}-1)\le 100$, whence	
\begin{align*}
S_2&\le 
	\frac{1}{10}
		\bigg(\sum_{k=0}^\infty\alpha_k(1-\beta_k)
		+\sum_{k=0}^\infty\beta_k(1-\alpha_k)
		+\sum_{\ell=0}^\infty\alpha_\ell(1-\beta_\ell)
		+\sum_{\ell=0}^\infty\beta_\ell(1-\alpha_\ell)\bigg) \\
	&=\frac{1}{5}\bigg(\sum_{k=0}^\infty\alpha_k(1-\beta_k)
		+ \sum_{\ell=0}^\infty\beta_\ell(1-\alpha_\ell)\bigg).
\end{align*}
Observing that 
\[
1-\beta_k=\sum_{\ell\ge0,\ \ell\neq k} \beta_\ell
\quad\text{and}\quad
1-\alpha_\ell=\sum_{k\ge0,\ k\neq \ell} \alpha_k,
\]
we conclude that
\[
S_2\le \frac{2}{5}\sum_{\substack{k,\ell\ge0\\ k\neq\ell}} \alpha_k\beta_\ell.
\]

Since $\alpha_k,\beta_\ell$ are non-negative reals which sum to 1, there exists some $k_0\ge0$ such that 
\[
\gamma:=\max_{k\ge0}\alpha_k\beta_k= \alpha_{k_0}\beta_{k_0}.
\]
We thus find that 
\[
S_1 = \sum_{k=0}^\infty(\alpha_k\beta_k)^{9/10}\le \gamma^{2/5}\sum_{k=0}^\infty(\alpha_k\beta_k)^{1/2}\le \gamma^{2/5}\Bigl(\sum_{k=0}^\infty \alpha_k\Bigr)^{1/2}\Bigl(\sum_{\ell=0}^{\infty}\beta_\ell\Bigr)^{1/2}= \gamma^{2/5}
\]
where we used the Cauchy-Schwarz inequality to bound $\sum_k (\alpha_k\beta_k)^{1/2}$ from above. 
We also find that
\[
\frac{5 S_2}{2}\le \sum_{\substack{k,\ell\ge0\\ k\neq\ell}} \alpha_k\beta_\ell
	= 1- \sum_{k=0}^\infty \alpha_k\beta_k
	\le 1-\gamma.
\]
As a consequence, 
\[
S_1+S_2\le \gamma^{2/5}+\frac{2}{5}(1-\gamma).
\]
The function $x\mapsto x^{2/5}+2(1-x)/5$ is increasing for $0\le x\le1$, and so maximized at $x=1$. Thus we infer that $S_1+S_2\le1$ as required, completing the proof of the lemma.
\end{proof}

\begin{lem}[Quality increment unless a prime power divides almost all]\label{lem:MainLem}
Consider a GCD graph $G=(\mu,\CV,\CW,\CE,\CP,f,g)$ with edge density $\delta>0$ and a prime $p\in\CR(G)$ with $p>10^{40}$. Then one of the following holds:
\begin{enumerate}
\item There is a GCD subgraph $G'$ of $G$ with multiplicative data $(\CP',f',g')$ and edge density $\delta'>0$ such that
\[
\CP'=\CP\cup\{p\},\quad
\CR(G')\subseteq \CR(G)\setminus\{p\},\quad 
\min\bigg\{1,\frac{\delta'}{\delta}\bigg\}
	\cdot \frac{q(G')}{q(G)} \ge 2^{\un_{f'(p)\neq g'(p)}} .
\]
\item There is some $k\in\Z_{\ge0}$ such that
\[
\frac{\mu(\CV_{p^k})}{\mu(\CV)}\ge1-\frac{10^{40}}{p}
\quad\text{and}\quad
\frac{\mu(\CW_{p^k})}{\mu(\CW)}\ge 1-\frac{10^{40}}{p}.
\]
\end{enumerate}
\end{lem}

\begin{proof} Let $\alpha_k$ and $\beta_\ell$ be defined as in the statement of Lemma \ref{lem:EdgeSets}. Consequently, there are $k,\ell\in\Z_{\ge0}$ such that $\alpha_k,\beta_\ell>0$ and
\begin{equation}\label{eq:FirstPigeonholeArgument}
\frac{\mu(\CE_{p^k,p^\ell})}{\mu(\CE)} 	
\ge 	
\begin{cases}
(\alpha_{k}\beta_{k})^{9/10}
&\text{if}\ k=\ell,\\
\\
\ds\frac{\alpha_k(1-\beta_k)
	+\beta_k(1-\alpha_k)
	+\alpha_\ell(1-\beta_\ell)
	+\beta_\ell(1-\alpha_\ell)}{2^{|k-\ell|/20}\times 1000}
&\text{otherwise}.
\end{cases}
\end{equation}
In particular, $\mu(\CE_{p^k,p^\ell})>0$, so that $G_{p^k,p^\ell}$ is a non-trivial GCD subgraph of $G$. We separate two cases, according to whether $k=\ell$ or not.

\bigskip

\noindent 
\textit{Case 1: $k=\ell$.}

\medskip

Let $G'=G_{p^k,p^k}$. Lemma \ref{lem:InducedGraphs} and our lower bound $\mu(\CE_{p^k,p^k})\ge (\alpha_k\beta_k)^{9/10}\mu(\CE)$ imply that
\als{
\frac{q(G')}{q(G)}	=\bigg(\frac{\mu(\CE_{p^k,p^k})}{\mu(\CE)}\bigg)^{10}
	(\alpha_k\beta_k)^{-9}
	\frac{1}{(1-\un_{k\ge1}/p)^2(1-1/p^{31/30})^{10}} \ge1.
}
In addition,
\[
\frac{\delta'}{\delta} 
	= \frac{\mu(\CE_{p^k,p^k})}{\mu(\CE)} \cdot
	 \frac{\mu(\CV)\mu(\CW)}{\mu(\CV_{p^k})\mu(\CW_{p^k})}  
	\ge (\alpha_{k}\beta_{k})^{9/10}\frac{1}{\alpha_k\beta_k} \ge 1.
\]
This establishes conclusion  (a) in this case, noting that $f'(p)=g'(p)=k$ so $\un_{f'(p)\ne g'(p)}=0$.

\bigskip

\noindent 
\textit{Case 2: $k\ne \ell$}

\medskip

As before, we let $G'=G_{p^k,p^\ell}$, and use Lemma \ref{lem:InducedGraphs} and our lower bound on $\CE_{p^k,p^\ell}$ to find that
\als{
	\frac{q(G')}{q(G)}
	&=\bigg(\frac{\mu(\CE_{p^k,p^\ell})}{\mu(\CE)}\bigg)^{10}
	(\alpha_k\beta_\ell)^{-9}
	\frac{p^{|k-\ell|}}{(1-1/p^{31/30})^{10}} \\
	&\ge\frac{S^{10}}{1000^{10}(\alpha_k\beta_\ell)^{9}}
		\cdot \Big(\frac{p}{2^{1/2}}\Big)^{|k-\ell|},
}
where
\[
S=\alpha_k(1-\beta_k)
+\beta_k(1-\alpha_k)
+\alpha_\ell(1-\beta_\ell)
+\beta_\ell(1-\alpha_\ell).
\]
In addition, we have
\als{
\frac{\delta'}{\delta}\cdot \frac{q(G')}{q(G)}
	&=\bigg(\frac{\mu(\CE_{p^k,p^\ell})}{\mu(\CE)}\bigg)^{11}
	(\alpha_k\beta_\ell)^{-10}
	\frac{p^{|k-\ell|}}{(1-1/p^{31/30})^{10}} \\
	&\ge\frac{S^{11}}{1000^{11}(\alpha_k\beta_\ell)^{10}}
		\cdot \Big(\frac{p}{2^{11/20}}\Big)^{|k-\ell|} .
}
Note that 
\eq{\label{S lb1}
S\ge \alpha_k(1-\beta_k)\ge \alpha_k\beta_\ell.
}
Indeed, this follows by our assumption that $k\neq\ell$, which implies that $\beta_k+\beta_\ell\le \sum_{j\ge 0}\beta_j=1$. Combining the above, we conclude that
\eq{\label{q'/q-lb}
	\min\bigg\{ \frac{q(G')}{q(G)},\frac{\delta'}{\delta}\cdot \frac{q(G')}{q(G)} \bigg\}
	\ge\frac{S^{2}}{1000^{11}\alpha_k\beta_\ell}
		\cdot \Big(\frac{p}{2^{11/20}}\Big)^{|k-\ell|}.
}

Now, assume that conclusion (a) of the lemma does not hold, so that the left hand side of \eqref{q'/q-lb} is $\le2$. Since $|k-\ell|\ge1$ and all primes are at least $2$, we must then have that
\[
S\le \frac{S^{2}}{\alpha_k\beta_\ell}
\le 2\cdot 10^{33}\bigg(\frac{2^{11/20}}{p}\bigg)^{|k-\ell|}\le\frac{10^{34}}{p}
\le\frac{1}{5},
\]
where we used our assumption that $p\ge 10^{40}$ for the last inequality. In particular, this gives
\begin{equation}\label{small quality}
S\le  \frac{10^{34}}{p}
\quad\text{and}\quad
\frac{S^2}{\alpha_k\beta_\ell}\le \frac{1}{5}.
\end{equation}

We note that
\begin{equation}
S\ge\alpha_k(1-\beta_k)+\beta_\ell(1-\alpha_\ell)\ge (\alpha_k+\beta_\ell)(1-\max\{\alpha_\ell,\beta_k\}).
\label{eq:SBound}
\end{equation}
Thus by the arithmetic-geometric mean inequality, and relations \eqref{eq:SBound} and \eqref{small quality}, we have
\[
(1-\max\{\alpha_\ell,\beta_k\})^2\le  \frac{(\alpha_k+\beta_\ell)^2}{4\alpha_k\beta_\ell}(1-\max\{\alpha_\ell,\beta_k\})^2\le \frac{S^2}{4\alpha_k\beta_\ell}\le \frac{1}{20}.
\]
In particular, $\max\{\alpha_\ell,\beta_k\}\ge 1/2$.

We consider the case when $\beta_k\ge1/2$; the case with $\alpha_\ell\ge 1/2$ is entirely analogous with the roles of $\beta$ and $\alpha$ swapped, and the roles of $k$ and $\ell$ swapped. Thus, to complete the proof of the lemma, it suffices to show that
\begin{equation}\label{small quality - large structure}
\alpha_k,\beta_k
	\ge1-\frac{10^{40}}{p}.
\end{equation}
The first inequality of \eqref{small quality} states that
\[
\alpha_k(1-\beta_k)
+\beta_k(1-\alpha_k)
+\alpha_\ell(1-\beta_\ell)
+\beta_\ell(1-\alpha_\ell)
	\le \frac{10^{34}}{p}.
\]
Since $\beta_k\ge1/2$, we infer that
\[
1-\alpha_k\le 2\beta_k(1-\alpha_k)\le 
	\frac{2\cdot 10^{34}}{p}\le \frac{10^{35}}{p} \le \frac{1}{2}. 
\]
In particular, $\alpha_k\ge 1-10^{40}/p$ and $\alpha_k\ge1/2$, whence
\[
1-\beta_k\le 2\alpha_k(1-\beta_k)\le 
\frac{2\cdot 10^{34}}{p}\le \frac{10^{40}}{p}.
\]
This completes the proof of \eqref{small quality - large structure} and hence of the lemma.
\end{proof}

\begin{proof}[Proof of Proposition \ref{prop:IterationStep1}]
This follows almost immediately from Lemma \ref{lem:MainLem}. Since $\CR(G)\subseteq\{p>10^{2000}\}$ by assumption, if $p\in\CR(G)$ then $p>10^{2000}$. We have also assumed that $\CR^{\flat}(G)\ne\emptyset$. Consequently, there is a prime $p\in\CR^\flat(G)$ with $p>10^{2000}>10^{40}$. We now apply Lemma \ref{lem:MainLem} with this choice of $p$. By definition of $\CR^\flat(G)$, conclusion $(b)$ cannot hold, and so conclusion $(a)$ must hold. This then gives the result.
\end{proof}

We are left to establish Proposition \ref{prop:SmallPrimes} and Proposition \ref{prop:IterationStep2}.

\section{Proof of Proposition \ref{prop:SmallPrimes}}\label{sec:SmallPrime}

In this section we prove Proposition \ref{prop:SmallPrimes}, which is the iteration procedure for small primes. This section relies on the notation of Section \ref{sec:Graph}, Lemma \ref{lem:HighDegree}, the Lemmas \ref{lem:InducedGraphs}-\ref{lem:UnbalancedSetEdges1} from Section \ref{sec:Prep} and Lemma \ref{lem:MainLem}. The basic idea of the proof is similar to that of Proposition \ref{prop:IterationStep1}, but we can no longer ensure a quality increment when the primes are small; instead we show that there is only a bounded loss.

\begin{lem}[Small quality loss or prime power divides positive proportion]\label{lem:SmallPrime}
Consider a GCD graph $G=(\mu,\CV,\CW,\CE,\CP,f,g)$ with edge density $\delta>0$, and let $p\in\CR(G)$ be a prime. Then one of the following holds:
\begin{enumerate}
\item There is a GCD subgraph $G'$ of $G$ with multiplicative data $(\CP',f',g')$ and edge density $\delta'>0$ such that
\[
\CP'=\CP\cup\{p\},\quad
\CR(G')\subseteq\CR(G)\setminus\{p\},\quad 
\min\bigg\{1,\frac{\delta'}{\delta}\bigg\}
	\cdot \frac{q(G')}{q(G)} \ge \frac{1}{10^{40}} .
\]
\item There is some $k\in\Z_{\ge0}$ such that
\[
\frac{\mu(\CV_{p^k})}{\mu(\CV)}\ge \frac{9}{10}
\quad\text{and}\quad
\frac{\mu(\CW_{p^k})}{\mu(\CW)}\ge \frac{9}{10}.
\]
\end{enumerate}
\end{lem}

\begin{proof}
Assume that conclusion $(a)$ does not hold, so we intend to establish $(b)$. For $k,\ell\in\Z_{\ge0}$, let $\mu(\CV_{p^k})=\alpha_k\mu(\CV)$ and $\mu(\CW_{p^\ell})=\beta_\ell\mu(\CW)$. We begin as in the proof of  Lemma \ref{lem:MainLem}, by considering $k,\ell\in\Z_{\ge0}$ satisfying \eqref{eq:FirstPigeonholeArgument} and the inequalities $\alpha_k,\beta_\ell>0$. In particular, $G_{p^k,p^\ell}$ is a non-trivial GCD subgraph of $G$.

We note that the proof of Lemma \ref{lem:MainLem} up to relation \eqref{q'/q-lb} requires no assumption on the size of $p$. Now, if $k=\ell$, then Case 1 of the proof of Lemma \ref{lem:MainLem} shows that conclusion $(a)$ must hold, contradicting our assumption. Therefore we may assume that $k\ne \ell$. Now, arguing as in Case 2 of the proof of Lemma \ref{lem:MainLem}, and setting $G'=G_{p^k,p^\ell}$ and
\[
S=\alpha_k(1-\beta_k)
+\beta_k(1-\alpha_k)
+\alpha_\ell(1-\beta_\ell)
+\beta_\ell(1-\alpha_\ell),
\]
we infer that
\[
	\frac{1}{10^{40}}\ge\min\bigg\{ \frac{q(G')}{q(G)},\frac{\delta'}{\delta}\cdot \frac{q(G')}{q(G)} \bigg\}
	\ge\frac{S^{2}}{1000^{11}\alpha_k\beta_\ell}
		\cdot \Big(\frac{p}{2^{11/20}}\Big)^{|k-\ell|}\ge\frac{S^{2}}{1000^{11}\alpha_k\beta_\ell}.
\]
Therefore we have that
\[
S\le \frac{S^2}{\alpha_k\beta_\ell}\le \frac{1}{10^{7}}.
\]
Since $S\ge (\alpha_k+\beta_\ell)(1-\max\{\alpha_\ell,\beta_k\})$, we have
\[
(1-\max\{\alpha_\ell,\beta_k\})^2\le \frac{ (\alpha_k+\beta_\ell)^2}{4\alpha_k\beta_\ell}(1-\max\{\alpha_\ell,\beta_k\})^2\le \frac{S^2}{4\alpha_k\beta_\ell}\le \frac{1}{100},
\]
so $\max\{\alpha_\ell,\beta_k\}\ge 9/10$. We deal with the case when $\beta_k\ge 9/10$; the case with $\alpha_\ell\ge 9/10$ is entirely analogous with the roles of $k$ and $\ell$ and the roles of $\alpha$ and $\beta$ swapped. 

Since $\beta_k\ge 9/10$, we have
\[
1-\alpha_k\le 2\beta_k(1-\alpha_k)\le 2S\le \frac{2}{10^{7}}
\]
In particular, $\alpha_k \ge 9/10$ and so conclusion $(b)$ holds, as required.
\end{proof}

\begin{lem}[Adding small primes to $\CP$]\label{lem:SmallIteration}
Let $G=(\mu,\CV,\CW,\CE,\CP,f,g)$ be a GCD graph with edge density $\delta>0$. Let $p\in\CR(G)$ be a prime with $p\le 10^{2000}$.

Then there is a GCD subgraph $G'$ of $G$ with set of primes $\CP'$ and edge density $\delta'>0$ such that 
\[
\CP'=\CP\cup\{p\},
\quad \CR(G')\subseteq\CR(G)\setminus\{p\},
\quad 
\min\bigg\{1,\frac{\delta'}{\delta}\bigg\}\cdot \frac{q(G')}{q(G)}\ge \frac{1}{10^{50}}.
\]
\end{lem}
\begin{proof}
We first repeatedly apply Lemma \ref{lem:HighDegree} until we arrive at a GCD subgraph 
\[
G^{(1)}=(\mu,\CV^{(1)},\CW^{(1)},\CE^{(1)},\CP,f,g)
\]
of $G$ with edge density $\delta^{(1)}$ such that 
\[
\delta^{(1)}\ge \delta
\quad\text{and}\quad
q(G^{(1)})\ge q(G),
\]
as well as 
\[
\mu(\Gamma_{G^{(1)}}(v))\ge \frac{9\delta^{(1)}}{10}\cdot \mu(\CW^{(1)})
\qquad\text{for all}\  v\in \CV^{(1)}.
\] 
(We must eventually arrive at such a subgraph since the vertex sets are strictly decreasing at each stage but can never become empty since the edge density remains bounded away from 0.)

We now apply Lemma \ref{lem:SmallPrime} to $G^{(1)}$. If conclusion $(a)$ of Lemma \ref{lem:SmallPrime} holds, then there is a GCD subgraph $G^{(2)}$ of $G^{(1)}$ satisfying the conclusion of Lemma \ref{lem:SmallIteration}, so we are done by taking $G'=G^{(2)}$. Therefore we may assume that instead conclusion $(b)$ of Lemma \ref{lem:SmallPrime} holds, so there is some $k\in\mathbb{Z}_{\ge 0}$ such that 
\begin{equation}\label{eq:FirstLowerBound}
\frac{\mu(\CV^{(1)}_{p^k})}{\mu(\CV^{(1)})}\ge \frac{9}{10}
\quad\text{and}\quad
\frac{\mu(\CW^{(1)}_{p^k})}{\mu(\CW^{(1)})}\ge \frac{9}{10}.
\end{equation}
In fact we claim that either the conclusion of Lemma \ref{lem:SmallIteration} holds, or we have the stronger condition
\begin{equation}\label{eq:ImprovedLowerBound}
\frac{\mu(\CV^{(1)}_{p^k})}{\mu(\CV^{(1)})}\ge \max\Bigl(\frac{9}{10},1-\frac{10^{40}}{p}\Bigr)
\quad\text{and}\quad
\frac{\mu(\CW^{(1)}_{p^k})}{\mu(\CW^{(1)})}\ge \max\Bigl(\frac{9}{10},1-\frac{10^{40}}{p}\Bigr).
\end{equation}
Relation \eqref{eq:ImprovedLowerBound} follows immediately from \eqref{eq:FirstLowerBound} if $p\le 10^{41}$, so let us assume that $p>10^{41}$. We then apply Lemma \ref{lem:MainLem} to $G^{(1)}$. If conclusion $(a)$ of Lemma \ref{lem:MainLem} holds, then there is a GCD subgraph $G^{(3)}$ of $G^{(1)}$ satisfying the required conditions of Lemma \ref{lem:SmallIteration}, so we are done by taking $G'=G^{(3)}$. Therefore we may assume that conclusion $(b)$ of Lemma \ref{lem:MainLem} holds, so that there is some $k'\ge0$ such that $\mu(\CV^{(1)}_{p^{k'}})/\mu(\CV^{(1)})\ge 1-10^{40}/p\ge 9/10$ and $\mu(\CW^{(1)}_{p^{k'}})/\mu(\CW^{(1)}) \ge 1-10^{40}/p\ge 9/10$. Since there cannot be two disjoint subsets of $\CV^{(1)}$ of density $\ge9/10$, we must then have $k'=k$, thus proving \eqref{eq:ImprovedLowerBound} in this case too. 

In conclusion, regardless of the size of $p$ we have established \eqref{eq:ImprovedLowerBound}. Next, we fix an integer $r\le 6644$ such that $p^r>10^{2000}$ (such an integer exists because $2^{6644}>10^{2000}$) and we apply Lemma \ref{lem:UnbalancedSetEdges1}.

If conclusion $(a)$ of Lemma \ref{lem:UnbalancedSetEdges1} holds, then we take $G'=G^{(1)}_{p^k,p^\ell}$, whose quality satisfies 
\[
q(G')\ge 2q(G^{(1)})\ge 2q(G)
\]
and whose edge density $\delta'$ satisfies 
\[
\delta'q(G')\ge2\delta^{(1)}q(G^{(1)})\ge2\delta q(G)>0.
\]
In particular, $\delta'>0$, so the proof is complete in this case.

Thus we may assume that conclusion $(b)$ of Lemma \ref{lem:UnbalancedSetEdges1} holds, so that
\[
\sum_{\ell\in\CL_{k,r}}  \mu(\CE^{(1)}_{p^k,p^\ell})\le\frac{\mu(\CE^{(1)})}{4p^{31/30}}<\frac{\mu(\CE^{(1)})}{4},
\]
where we recall the notation $\CL_{k,r}:=\{\ell\in\Z_{\ge0}:|\ell-k|\ge r+1\}$. Let 
\[
\tilde{\CW}^{(1)}=\bigcup_{\substack{\ell\ge0 \\ |\ell-k|\le r}}\CW^{(1)}_{p^\ell}
\]
and let
\[
\CE^{(2)}=\CE^{(1)}\cap(\CV^{(1)}_{p^k}\times\tilde{\CW}^{(1)})\subseteq\CE^{(1)}
\] 
be the set of edges between $\CV^{(1)}_{p^k}$ and $\tilde{\CW}^{(1)}$ in $G^{(1)}$. Since $\mu(\CV^{(1)}_{p^k})\ge 9\mu(\CV^{(1)})/10$ and $\mu(\Gamma_{G^{(1)}}(v))\ge 9\delta^{(1)}\mu(\CW^{(1)})/10$ for all $v\in \CV^{(1)}_{p^k}$, we have
\begin{align*}
\mu(\CE^{(2)})\ge \mu(\CE^{(1)}\cap(\CV_{p^k}^{(1)}\times\CW^{(1)}) ) -\sum_{\ell\in\CL_{k,r}}  \mu(\CE^{(1)}_{p^k,p^\ell})
&\ge \sum_{v\in \CV^{(1)}_{p^k}}\mu(v)\mu(\Gamma_{G^{(1)}}(v))-\frac{\mu(\CE^{(1)})}{4}\\
&\ge \frac{9 \delta^{(1)}}{10}\mu(\CV^{(1)}_{p^k})\mu(\CW^{(1)})-\frac{\mu(\CE^{(1)})}{4}\\
&\ge \frac{56}{100}\mu(\CE^{(1)})>0. 
\end{align*}
Let $G^{(2)}=(\mu,\CV^{(1)}_{p^k},\tilde{\CW}^{(1)},\CE^{(2)},\CP,f,g)$ be the GCD subgraph of $G^{(1)}$ formed by restricting to $\CV^{(1)}_{p^k}$ and $\tilde{\CW}^{(1)}$. Since $\mu(\CE^{(2)})>0$, $G^{(2)}$ is a non-trivial GCD subgraph. If $\delta^{(2)}$ denotes its edge density, then
\begin{align*}
\frac{\delta^{(2)}}{\delta^{(1)}}= 
\bigg(\frac{\mu(\CE^{(2)})}{\mu(\CE^{(1)})}\bigg)
	 \bigg(\frac{\mu(\CV^{(1)})}{\mu(\CV_{p^k}^{(1)})}\bigg)
	\bigg( \frac{\mu(\CW^{(1)})}{\mu(\tilde{\CW}^{(1)})}\bigg)
	\ge\frac{1}{2}\cdot 1\cdot 1=\frac{1}{2}.
\end{align*}
In addition, we have that
\begin{align*}
\frac{q(G^{(2)})}{q(G^{(1)})}
=\bigg(\frac{\mu(\CE^{(2)})}{\mu(\CE^{(1)})}\bigg)^{10}
\bigg(\frac{\mu(\CV^{(1)})}{\mu(\CV_{p^k}^{(1)})}\bigg)^9
\bigg( \frac{\mu(\CW^{(1)})}{\mu(\tilde{\CW}^{(1)})}\bigg)^9
\ge\bigg(\frac{1}{2}\bigg)^{10}\cdot 1^9\cdot 1^9=\frac{1}{2^{10}}.
\end{align*}
Finally, we apply Lemma \ref{lem:Pigeonhole} to the partition 
\[
\tilde{\CW}^{(1)}=\bigsqcup_{|\ell-k|\le r}\CW^{(1)}_{p^\ell}
\]
of $\tilde{\CW}^{(1)}$ into $\le 2\cdot 6644+1\le 15000$ subsets. This produces a GCD subgraph
\[
G^{(3)}=(\mu,\CV^{(1)}_{p^k},\CW^{(1)}_{p^\ell},\CE^{(1)}_{p^k,p^\ell},\CP,f,g)
\]
of $G^{(2)}$ for some $\ell\ge 0$ with $|\ell-k|\le r$ such that
\[
q(G^{(3)})\ge \frac{q(G^{(2)})}{15000^{10} }\ge \frac{q(G^{(1)})}{15000^{10}\cdot 2^{10}}\ge \frac{q(G)}{10^{50}}.
\]
In addition, Lemma \ref{lem:Pigeonhole} implies that the density of $G^{(3)}$, call it $\delta^{(3)}$, satisfies
\[
\delta^{(3)}q(G^{(3)})\ge \frac{\delta^{(2)}}{15000}\cdot \frac{q(G^{(2)})}{15000^{10}}
	\ge \frac{\delta^{(1)}}{15000\cdot 2}\cdot \frac{q(G^{(1)})}{15000^{10}\cdot 2^{10}}\ge \frac{\delta\,q(G)}{10^{50}}.
\]
Finally, we note that $G^{(1)}_{p^k,p^\ell}$ is a GCD subgraph of $G^{(3)}$ with set of primes $\CP\cup\{p\}$, edge density $\delta^{(1)}_{p^k,p^\ell}=\delta^{(3)}$, and quality $q(G^{(1)}_{p^k,p^\ell})\ge q(G^{(3)})$. Taking $G'=G^{(1)}_{p^k,p^\ell}$ then gives the result.
\end{proof}

\begin{proof}[Proof of Proposition \ref{prop:SmallPrimes}]
If $\CR(G)\cap\{p\le 10^{2000}\}=\emptyset$, then we can simply take $G'=G$.

If $\CR(G)\cap\{p\le 10^{2000}\}\ne\emptyset$, then we can choose a prime $p\in\CR(G)\cap\{p\le 10^{2000}\}$ and apply Lemma \ref{lem:SmallIteration}. We do this repeatedly to produce a sequence of GCD subgraphs 
\[
G=:G_1\succeq G_2\succeq\cdots
\]
such that
\begin{equation}\label{eq:SmallPrimesInduction}
\frac{q(G_{i+1})}{q(G_i)}\ge\frac{1}{10^{50}}
\quad\text{and}\quad \frac{\delta_{i+1}q(G_{i+1})}{\delta_iq(G_i)}\ge\frac{1}{10^{50}}
\end{equation}
for each $i$, where $\delta_i$ denotes the edge density of $G_i$. In addition, we let $\CP_i$ denote the set of primes associated to $G_i$, so that $\emptyset=\CP_1\subseteq\CP_2\subseteq\cdots \subseteq\{p\le 10^{2000}\}$. 

At each stage, the set $\CR(G_i)\cap\{p\le 10^{2000}\}$ is strictly smaller than before. So, after at most $10^{2000}$ steps we arrive at a GCD subgraph $G^{(1)}=(\mu,\CV^{(1)},\CW^{(1)},\CE^{(1)},\CP^{(1)},f^{(1)},g^{(1)})$ of $G$ with 
\[
\CP^{(1)}\subseteq\{p\le 10^{2000}\}
\quad\text{and}\quad \CR(G^{(1)})\cap\{p\le 10^{2000}\}=\emptyset.
\]
Let $\delta^{(1)}$ denote the edge density of the end graph $G^{(1)}$. Iterating the two inequalities of \eqref{eq:SmallPrimesInduction} at most $10^{2000}$ times, we find that
\[
\frac{q(G^{(1)})}{q(G)}\ge \frac{1}{(10^{50})^{10^{2000}}}\ge \frac{1}{10^{10^{3000}}}
\quad\text{and}\quad 
\frac{\delta^{(1)}q(G^{(1)})}{\delta\,q(G)}
\ge \frac{1}{(10^{50})^{10^{2000}}}\ge \frac{1}{10^{10^{3000}}}.
\] 
Thus, taking $G'=G^{(1)}$ gives the result.
\end{proof}

Thus we are just left to establish Proposition \ref{prop:IterationStep2}.

\section{Proof of Proposition \ref{prop:IterationStep2}}\label{sec:IterationStep2}

Finally, in this section we prove Proposition \ref{prop:IterationStep2}, and hence complete the proof of Theorem \ref{thm:MainThm}. The proof is similar to that of Proposition \ref{prop:IterationStep1}, but more care is required when dealing with the primes coming from $\CR^{\sharp}(G)$.

\begin{lem}[Quality increment even when a prime power divides almost all]\label{lem:MainLem2} 
Consider a GCD graph $G=(\mu,\CV,\CW,\CE,\CP,f,g)$ with edge density $\delta>0$ and let $p\in\CR(G)$ be a prime with $p\ge 10^{2000}$. Then there is a GCD subgraph $G'$ of $G$ with set of primes $\CP'=\CP\cup\{p\}$ such that
\[
\CR(G')\subseteq\CR(G)\setminus\{p\}
\quad\text{and}\quad  q(G')\ge q(G)>0.
\]
\end{lem}

\begin{proof} First of all, we may assume without loss of generality that for all sets $\CA\subseteq \CV$ and $\CB\subseteq \CW$, we have that 
\begin{equation}\label{eq:NoEdgesBetweenSmallSets}
\mu(\CE\cap(\CA\times \CB)) \le \frac{\mu(\CE)}{2p^{3/2}} 
\quad\text{whenever}\quad \max\bigg\{ \frac{\mu(\CA)}{\mu(\CV)}, \frac{\mu(\CB)}{\mu(\CW)} \bigg\} \le \frac{10^{40}}{p} .
\end{equation}
Indeed, if $G$ does not satisfy \eqref{eq:NoEdgesBetweenSmallSets}, then we apply Lemma \ref{lem:NoSmallSetEdges} with $\eta=10^{40}/p$ to replace $G$ by a non-trivial subgraph $G^{(1)}$ that does have this property (noticing that $(10^{40}/p)^{9/5}\le 1/(2p^{3/2})$ for $p\ge 10^{2000}$). In addition, $G^{(1)}$ has the same multiplicative data as $G$ and its quality is strictly larger. Hence, we may work with $G^{(1)}$ instead. So, from now on, we assume that \eqref{eq:NoEdgesBetweenSmallSets} holds. 

We now apply Lemma \ref{lem:MainLem}. If conclusion $(a)$ of Lemma \ref{lem:MainLem} holds, then we are done. Thus we may assume that conclusion $(b)$ holds, that is to say there is some $k\in\Z_{\ge0}$ such that 
\[
\frac{\mu(\CV_{p^k})}{\mu(\CV)}\ge1-\frac{10^{40}}{p}
\quad\text{and}\quad
\frac{\mu(\CW_{p^k})}{\mu(\CW)}\ge 1-\frac{10^{40}}{p}.
\]
In particular, by \eqref{eq:NoEdgesBetweenSmallSets} we see that
\begin{equation}\label{eq:NoBadEdges1}
\mu\big(\CE(\CV\setminus\CV_{p^k},\CW\setminus \CW_{p^k})\big) 
\le \frac{\mu(\CE)}{2p^{3/2}}.
\end{equation}

Now, set
\[
\tilde{\CV}_{p^k}=\CV_{p^{k-1}}\cup\CV_{p^k}\cup\CV_{p^{k+1}}
\quad\text{and}\quad
\tilde{\CW}_{p^k}=\CW_{p^{k-1}}\cup\CW_{p^k}\cup\CW_{p^{k+1}} ,
\]
with the convention that $\CV_{p^{-1}}=\emptyset=\CW_{p^{-1}}$. 
In view of Lemmas \ref{lem:UnbalancedSetEdges1} and \ref{lem:UnbalancedSetEdges2} applied with $r=1$, we may assume that 
\begin{equation}\label{eq:NoBadEdges2}
\mu\big(\CE(\CV\setminus\tilde{\CV}_{p^k},\CW_{p^k})\big)
	= \sum_{\substack{i\ge0 \\ |i-k|\ge2}} \mu(\CE(\CV_{p^i},\CW_{p^k})) \le \frac{\mu(\CE)}{4p^{31/30}},
\end{equation}
and
\begin{equation}\label{eq:NoBadEdges3}
\mu\big(\CE(\CV_{p^k},\CW\setminus\tilde{\CW}_{p^k})\big)
	= \sum_{\substack{j\ge0 \\ |j-k|\ge2}} \mu(\CE(\CV_{p^k},\CW_{p^j})) \le \frac{\mu(\CE)}{4p^{31/30}} .
\end{equation}
Hence, if we let
\[
\CE^*=\CE(\CV_{p^k},\tilde{\CW}_{p^k})\cup \CE(\tilde{\CV}_{p^k},\CW_{p^k}) ,
\]
then \eqref{eq:NoBadEdges1}-\eqref{eq:NoBadEdges3} imply that
\[
\frac{\mu(\CE^*)}{\mu(\CE)} \ge 1- \frac{1}{2p^{31/30}} - \frac{1}{2p^{3/2}} 
			\ge \bigg(1-\frac{1}{p^{31/30}}\bigg)^{2/3} >0 ,
\]
where we used our assumption that $p>10^{2000}$ and the inequality $(1-x)^{2/3}\le 1-2x/3$ for $x\in[0,1]$ that follows from Taylor's theorem. We then consider the non-trivial GCD subgraph $G^*=(\mu,\CV,\CW,\CE^*,\CP,f,g)$ of $G$ formed by restricting the edge set to $\CE^*$. Note that
\eq{\label{G^*/G}
\frac{q(G^*)}{q(G)}=\Bigl(\frac{\mu(\CE^*)}{\mu(\CE)}\Bigr)^{10} \ge \bigg(1-\frac{1}{p^{31/30}}\bigg)^{20/3}.
}

Now, let $(v,w)\in\CE^*$. We have the following five possibilities: 
\begin{enumerate}
\item $v\in\CV_{p^k}$ and $w\in\CW_{p^k}$, in which case $p^k\|v,w$ and $p^k\|\gcd(v,w)$;
\item $v\in\CV_{p^k}$ and $w\in\CW_{p^{k+1}}$, in which case $p^k|v,w$ and $p^k\|\gcd(v,w)$;
\item $v\in\CV_{p^{k+1}}$ and $w\in\CW_{p^k}$, in which case $p^k|v,w$ and $p^k\|\gcd(v,w)$;
\item $v\in\CV_{p^k}$ and $w\in\CW_{p^{k-1}}$, in which case $p^k\|v$, $p^{k-1}\|w$ and $p^{k-1}\|\gcd(v,w)$;
\item $v\in\CV_{p^{k-1}}$ and $w\in\CW_{p^k}$, in which case $p^{k-1}\|v$, $p^k\|w$ and $p^{k-1}\|\gcd(v,w)$.
\end{enumerate}
We then set $G^+=(\mu,\CV^+,\CW^+,\CE^+,\CP\cup\{p\},f^+,g^+)$, where:
\[
\CV^+=\CV_{p^k}\cup\CV_{p^{k+1}},
\quad
\CW^+=\CW_{p^k}\cup\CW_{p^{k+1}},
\quad
\CE^+=\CE^*\cap(\CV^+\times \CW^+),
\]
as well as
\[
f^+\big|_\CP=f,
\quad
f^+(p)=k,
\quad
g^+\big|_{\CP}=g,
\quad
g^+(p)=k .
\]
By looking at possibilities (a), (b) and (c), it is easy to check that $G^+$ is a GCD subgraph of $G^*$ (and hence of $G$). Note that $\mu(\CV^+)\ge \mu(\CV_{p^k})\ge 1-10^{40}/p>0$. Similarly, we have $\mu(\CW^+)>0$. Consequently, its quality satisfies the relation
\[
\frac{q(G^+)}{q(G^*)} =
	\bigg(\frac{\mu(\CE^+)}{\mu(\CE^*)}\bigg)^{10} 
	\bigg( \frac{\mu(\CV)}{\mu(\CV^+)}\bigg)^{9}
	 \bigg( \frac{\mu(\CW)}{\mu(\CW^+)}\bigg)^{9}
		\frac{1}{(1-\un_{k\ge1}/p)^2(1-1/p^{31/30})^{10}} .	
\]
(This relation is valid even if $\mu(\CE^+)=0$.) We separate two cases.

\bigskip

\noindent 
{\it Case 1:} $k=0$. 

\medskip

In this case $\CV_{p^{k-1}}=\CW_{p^{k-1}}=\emptyset$, so all parameters of $G^+$ are the same as those of $G^*$ except that the set of primes of $G^+$ is $\CP\cup\{p\}$ instead of $\CP$ and $f$,$g$ have been extended to take the value 0 at $p$. As a consequence, 
\[
\frac{q(G^+)}{q(G^*)} =
		\frac{1}{(1-1/p^{31/30})^{10}} .	
\]
In particular, by \eqref{G^*/G} we have
\[
q(G^+)=\frac{q(G^*)}{(1-1/p^{31/30})^{10}} \ge \frac{(1-1/p^{31/30})^{20/3}}{(1-1/p^{31/30})^{10}} q(G)\ge q(G)>0 .
\]
Thus the lemma follows by taking $G'=G^+$.

\bigskip

\noindent 
{\it Case 2:} $k\ge1$.

\medskip

 In this case we have
\al{
\frac{q(G^+)}{q(G^*)} =
	\bigg(\frac{\mu(\CE^+)}{\mu(\CE^*)}\bigg)^{10} 
	\bigg( \frac{\mu(\CV)}{\mu(\CV^+)}\bigg)^{9}
	 \bigg( \frac{\mu(\CW)}{\mu(\CW^+)}\bigg)^{9}
		\frac{1}{(1-1/p)^2(1-1/p^{31/30})^{10}} .
\label{eq:G+}
}
We also consider the GCD subgraphs  $G_{p^k,p^{k-1}}$ and $G_{p^{k-1},p^k}$ of $G$. Notice that $\mu(\CV_{p^k})\ge1-10^{40}/p>0$ for $p\ge 10^{2000}$. Hence, if $\mu(\CW_{p^{k-1}})>0$, then Lemma \ref{lem:InducedGraphs} implies that
\al{
\frac{q(G_{p^k,p^{k-1}})}{q(G^*)} =
	\bigg(\frac{\mu(\CE_{p^k,p^{k-1}})}{\mu(\CE^*)}\bigg)^{10} 
	\bigg( \frac{\mu(\CV)}{\mu(\CV_{p^k})}\bigg)^{9}
	 \bigg( \frac{\mu(\CW)}{\mu(\CW_{p^{k-1}})}\bigg)^{9}
		\frac{p}{(1-1/p^{31/30})^{10}} .
\label{eq:G2}
}
Similarly, if $\mu(\CV_{p^{k-1}})>0$, then we have 
\al{
\frac{q(G_{p^{k-1},p^k})}{q(G^*)} =
	\bigg(\frac{\mu(\CE_{p^{k-1},p^k})}{\mu(\CE^*)}\bigg)^{10} 
	\bigg( \frac{\mu(\CV)}{\mu(\CV_{p^{k-1}})}\bigg)^{9}
	 \bigg( \frac{\mu(\CW)}{\mu(\CW_{p^k})}\bigg)^{9}
		\frac{p}{(1-1/p^{31/30})^{10}} .
\label{eq:G3}
}

Since $\mu(\CV_{p^k})\ge(1-10^{40}/p)\mu(\CV)$, we have that $\mu(\CV_{p^{k-1}})\le 10^{40}\mu(\CV)/p$. Similarly, we have that $\mu(\CW_{p^{k-1}})\le 10^{40}\mu(\CW)/p$. To this end, let $0\le A,B\le 10^{40}$ be such that
\al{
\frac{\mu(\CV_{p^{k-1}})}{\mu(\CV)} = \frac{A}{p} 
\quad\text{and}\quad
\frac{\mu(\CW_{p^{k-1}})}{\mu(\CW)} = \frac{B}{p}.
\label{eq:AB}
}
We note that this implies that
\[
\frac{\mu(\CV^+)}{\mu(\CV)}\le 1-\frac{A}{p}\quad\text{and}\quad\frac{\mu(\CW^+)}{\mu(\CW)} \le  1-\frac{B}{p}.
\]
We also note that $\mu(\CE_{p^k,p^{k-1}})\le \mu(\CV_{p^k})\mu(\CW_{p^{k-1}})\le B\mu(\CV)/p$, so if $\mu(\CE_{p^k,p^{k-1}})>0$ then $B>0$. Similarly if $\mu(\CE_{p^{k-1},p^{k}})>0$ then $A>0$.

Combining \eqref{eq:G+} and \eqref{eq:AB} with \eqref{G^*/G}, we find
\al{
\frac{q(G^+)}{q(G)} &\ge
	\bigg(\frac{\mu(\CE^+)}{\mu(\CE^*)}\bigg)^{10} 
	\frac{1}{(1-A/p)^{9}(1-B/p)^{9}(1-1/p)^2(1-1/p^{31/30})^{10/3}} . \label{eq:ineqA}
	}
Similarly, provided $B>0$, \eqref{eq:G2}, \eqref{eq:AB} and \eqref{G^*/G} give
\al{
\frac{q(G_{p^k,p^{k-1}})}{q(G)}& \ge
	\bigg(\frac{\mu(\CE_{p^k,p^{k-1}})}{\mu(\CE^*)}\bigg)^{10} 
		\frac{p^{10}}{B^{9}(1-1/p^{31/30})^{10/3}}  \, ,\label{eq:ineqB}
		}
		and, provided $A>0$, \eqref{eq:G3}, \eqref{eq:AB} and \eqref{G^*/G} give
		\al{
\frac{q(G_{p^{k-1},p^k})}{q(G)} &\ge
	\bigg(\frac{\mu(\CE_{p^{k-1},p^k})}{\mu(\CE^*)}\bigg)^{10} 
		\frac{p^{10}}{A^{9}(1-1/p^{31/30})^{10/3}} .\label{eq:ineqC}
}
We now claim that at least one of the following inequalities holds:
\al{
\frac{\mu(\CE^+)}{\mu(\CE^*)}	&> 
	(1-A/p)^{9/10}(1-B/p)^{9/10}(1-1/p)^{2/10}(1-1/p^{31/30})^{1/3} \,; \label{eq:ineq1}\\
\frac{\mu(\CE_{p^k,p^{k-1}})}{\mu(\CE^*)}& > 
	\frac{B^{9/10}(1-1/p^{31/30})^{1/3}}{p} \,;\label{eq:ineq2}\\
\frac{\mu(\CE_{p^{k-1},p^k})}{\mu(\CE^*)}& > 
	\frac{A^{9/10}(1-1/p^{31/30})^{1/3}}{p} \,.\label{eq:ineq3}
}
If \eqref{eq:ineq1} holds then $q(G^+)\ge q(G)$ by \eqref{eq:ineqA}. If \eqref{eq:ineq2} holds, then $\mu(\CE_{p^k,p^{k-1}})>0$, so $B>0$, and so $q(G_{p^k,p^{k-1}})\ge q(G)$ by \eqref{eq:ineqB} and \eqref{eq:ineq2}. Finally, if \eqref{eq:ineq3} holds, then $\mu(\CE_{p^{k-1},p^k})>0$, so $A>0$, and so $q(G_{p^{k-1},p^k})\ge q(G)$ by \eqref{eq:ineqC} and \eqref{eq:ineq3}. Therefore this claim would complete the proof by choosing $G'\in\{G^+,G_{p^k,p^{k+1}},G_{p^{k+1},p^k}\}$ according to which of the inequalities \eqref{eq:ineq1}-\eqref{eq:ineq3} hold.

Since $\mu(\CE^+)+\mu(\CE_{p^k,p^{k-1}})+\mu(\CE_{p^{k-1},p^k})=\mu(\CE^*)$, at least one of \eqref{eq:ineq1}-\eqref{eq:ineq3} holds if we can prove that
\[
S:= \bigg( (1-A/p)^{\frac{9}{10}}(1-B/p)^{\frac{9}{10}}(1-1/p)^{\frac{1}{5}} 
	+ \frac{B^{\frac{9}{10}}}{p} 
	+ \frac{A^{\frac{9}{10}}}{p} \bigg) \bigg(1-\frac{1}{p^{31/30}}\bigg)^{1/3} < 1. 
\]
Using the inequality $1-x\le e^{-x}$ three times, we find that
\[
S\le  \bigg( \exp\Big(-\frac{9A+9B+2}{10 p} \Big) 
	+ \frac{B^{\frac{9}{10}}}{p} 
	+ \frac{A^{\frac{9}{10}}}{p} \bigg) \bigg(1-\frac{1}{p^{31/30}}\bigg)^{1/3} .	
\]
Since we also have that $e^{-x}\le 1-x+x^2/2$ for $x\ge0$, as well as $0\le A,B\le 10^{40}$, we conclude that
\[
S\le  \bigg( 1-\frac{9A+9B+2}{10 p}+ \frac{10^{81}}{p^2} 
	+ \frac{B^{\frac{9}{10}}}{p} 
	+ \frac{A^{\frac{9}{10}}}{p} \bigg) \bigg(1-\frac{1}{p^{31/30}}\bigg)^{1/3} .	
\]
By the arithmetic-geometric mean inequality, we have that $(9A+1)/10\ge A^{9/10}$ and $(9B+1)/10\ge B^{9/10}$, whence
\[
S\le  \bigg( 1 + \frac{10^{81}}{p^2} \bigg)\bigg(1-\frac{1}{p^{31/30}}\bigg)^{1/3}.
\]
Since $(1-x)^{1/3}\le 1-x/3$ for $x\in[0,1]$, we must have that $S< 1$ for $p\ge 10^{2000}$, thus completing the proof of the lemma.
\end{proof}

\begin{proof}[Proof of Proposition \ref{prop:IterationStep2}]
This follows almost immediately from Lemma \ref{lem:MainLem2}. Our assumptions that $\CR(G)\subseteq\{p>10^{2000}\}$ and $\CR^\sharp(G)\neq\emptyset$ imply that there is a prime $p>10^{2000}$ lying in $\CR(G)$. 
Thus we can apply Lemma \ref{lem:MainLem2} with this choice of $p$ and complete the proof.
\end{proof}

This completes the proof of Proposition \ref{prop:IterationStep2}, and hence Theorem \ref{thm:MainThm}.

\section{Concluding remarks and counterexamples to the Model Problem}\label{sec:Concluding}

It is a vital feature of our proof that the weight of all vertices $v$ has a factor $\phi(v)/v$, as naturally arises from the setup of the Duffin-Schaeffer conjecture. This allows our proof to (just) work, but without weights of this type our argument would fail. At first sight this point may appear to be a mere technicality, but without these weights there are genuine counterexamples to the entire approach.

\medskip

First, let us see where the proof breaks down without the $\phi(v)/v$ factors. Although most of the argument holds for a general measure $\mu$, in Proposition \ref{prop:Graph} we specialize to the measure $\mu(v)=\psi(v)\phi(v)/v$. In the proof of Proposition \ref{prop:Graph} (in particular, in relation \eqref{eq:G'Quality}), the $\phi(v)\phi(w)/v w$ factor cancels out the factor $a b /\phi(a)\phi(b)$ coming from
\begin{equation}\label{eq:Quality-PhiEffect}
\prod_{p\in\CP}(1-\un_{f(p)=g(p)\ge 1}/p)^{-2}
\end{equation}
in the definition of quality. Otherwise, the proof of Proposition \ref{prop:Graph} would fail. On the other hand, if we were to modify the definition of the quality and remove from it the product in \eqref{eq:Quality-PhiEffect}, then instead the proof of Lemma \ref{lem:MainLem2} would break down and we would not obtain a quality increment when there are many primes dividing a proportion of $1-1/p$ of each vertex set. Thus the argument we present fails without the $\phi(q)/q$ weights.

\medskip

Now, let use explain why the presence of the weight $\phi(v)/v$ is essential for the kind of argument we have given to work. Without using the $\phi(v)/v$ weights, we essentially are attempting to prove that the Model Problem of Section \ref{sec:outline} has an affirmative answer. However, one can construct examples to show that this is not the case. Such examples are based on the observation that all pairwise GCDs of elements of $\{n!/j:\, n/2\le j\le n\}$ are at least $(n-2)!$, but there is no fixed integer of size $\gg (n-2)!$ dividing a positive proportion of elements of this set. (We thank Sam Chow for showing us this construction.) 

Specifically, we select an integer $n\sim \log\log{x}$, a prime $p\in[x^{1-c}n^2/n!,(9/8)x^{1-c}n^2/n!]$, and then take
\[
\CS:=\Bigl\{\frac{n!}{j}p m:\,\frac{3 n}{4}\le j\le n,\,\frac{x^c}{n}\le m\le \frac{4 x^c}{3 n},\,\gcd(m,j)=1\Bigr\}.
\]
It is straightforward to verify that $\CS\subseteq[x,2x]$, and $\#\CS\asymp x^c$. Moreover, we see that if $v_1=n! p m_1/j_1$ and $v_2=n!p m_2/j_2$ are two elements of $\CS$, then
\[
\gcd(v_1,v_2)=\frac{p n!\gcd(m_1,m_2)}{\lcm(j_1,j_2)}\ge \frac{p n!}{n^2}\ge x^{1-c},
\]
so \textit{all} pairs $v_1,v_2$ in $\CS$ have a large gcd. However, we can easily check that there is no integer $d\gg x^{1-c}$ dividing a positive proportion of elements of $\CS$, and so this shows that the Model Problem of Section \ref{sec:outline} has a negative answer.

On the other hand, if we count integers $v$ with weight $\mu(v)=\phi(v)/v$, then the set $\CS$ we defined above has total weight $\mu(\CS)\asymp x^c/\log{n}$, and so it fails to be of a sufficiently large size unless we take $n$ bounded (in which case the prime $p$ is of size $\asymp x^{1-c}$ and it divides a positive proportion of the elements of $\CS$).  Thus the above counterexample no longer works if we count integers with weight $\mu$.


\begin{thebibliography}{99}
	\bibitem{aistleiner}
	C. Aistleitner,\,
	{\it A note on the Duffin-Schaeffer conjecture with slow divergence.}
	Bull. Lond. Math. Soc. 46 (2014), no. 1, 164--168. 
	
	\bibitem{aistleitner2}
	\bysame,\,
	{\it Decoupling theorems for the Duffin-Schaeffer problem.} Progress report (2019), 24 pages,
	 arXiv:1907.04590.
	
	\bibitem{aistleitner}
	C. Aistleitner, T. Lachmann, M. Munsch, N. Technau, and A. Zafeiropoulos,\,
	{\it The Duffin-Schaeffer conjecture with extra divergence.}
	Preprint, \url{https://arxiv.org/abs/1803.05703}.
	
	\bibitem{DS-survey-article} 
	V. Beresnevich, V. Bernik, M. Dodson and S. Velani,\,
	{\it Classical metric Diophantine approximation revisited.} 
	Analytic number theory, 38--61, Cambridge Univ. Press, Cambridge, 2009.
	
	\bibitem{extra-div2} 
	V. Beresnevich, G. Harman, A. K. Haynes and S. Velani,\,
	{\it The Duffin-Schaeffer conjecture with extra divergence II.} 
	Math. Z. 275 (2013), no. 1-2, 127--133.
	
	\bibitem{hausdorff DS}
	V. Beresnevich and S. Velani,\,
	{\it A mass transference principle and the Duffin-Schaeffer conjecture for Hausdorff measures.}
	Ann. of Math. (2) 164 (2006), no. 3, 971--992. 
	
	\bibitem{catlin}
	P. A. Catlin,\,
	{\it Two problems in metric Diophantine approximation. I.}
	J. Number Theory 8 (1976), no. 3, 282--288.
	
	\bibitem{DS} 
	R. J. Duffin and A. C. Schaeffer,\,  {\it Khinchin's problem in metric Diophantine approximation.} 
	Duke Math. J. 8 (1941), 243--255.
	
	\bibitem{dyson}
	F. J. Dyson,\, {\it A theorem on the densities of sets of integers.} 
	J. London Math. Soc. 20 (1945), 8--14.	
	
	\bibitem{erdos} 
	P. Erd\H os,\,
	{\it On the distribution of the convergents of almost all real numbers.} 
	J. Number Theory 2 (1970), 425--441. 
	
	\bibitem{erdoskorado}
	P. Erd\H{o}s, C. Ko, and R. Rado,\,{\it Intersection theorems for systems of finite sets.}
	Quart. J. Math. Oxford Ser. (2) 12 (1961), 313--320.	
	
	\bibitem{gallagher}
	P. Gallagher,\, 
	{\it Approximation by reduced fractions.} 
	J. Math. Soc. Japan 13 (1961), 342--345. 

	\bibitem{harman}
	G. Harman,\,
	{\it Metric number theory.}
	London Mathematical Society Monographs. New Series, 18. The Clarendon Press, Oxford University Press, New York, 1998. 
	
	\bibitem{extra-div1}
	A. K. Haynes, A. D.  Pollington, S. L. Velani,\,
	{\it The Duffin-Schaeffer conjecture with extra divergence.} 
	Math. Ann. 353 (2012), no. 2, 259--273.
	
	\bibitem{Khinchin-paper}
	A. Khintchine,\,
	{\it Einige S\"atze \"uber Kettenbrüche, mit Anwendungen auf die Theorie der Diophantischen Approximationen.} (German) 
	Math. Ann. 92 (1924), no. 1-2, 115--125. 
	
	
	\bibitem{Khinchin-book}
	A. Khinchin,\,
	{\it Continued fractions.} With a preface by B. V. Gnedenko. Translated from the third (1961) Russian edition. Reprint of the 1964 translation. Dover Publications, Inc., Mineola, NY, 1997.
	
	\bibitem{dk-book}
	D. Koukoulopoulos,\,
	{\it The distribution of prime numbers.} Graduate Studies in Mathematics, 203. American Mathematical Society, Providence, RI, 2019.
	
	\bibitem{montgomery}
	H. L. Montgomery,\,
	{\it Ten lectures on the interface between analytic number theory and harmonic analysis.}
	CBMS Regional Conference Series in Mathematics, 84. 
	Published for the Conference Board of the Mathematical Sciences, Washington, DC; 
	by the American Mathematical Society, Providence, RI, 1994. 
	
	\bibitem{PV}
	A. D. Pollington and R. C. Vaughan,\, 
	{\it  The $k$-dimensional Duffin and Schaeffer conjecture.} Mathematika 37 (1990), no. 2, 190--200. 
	
	
	\bibitem{RS} J. B. Rosser and L. Schoenfeld,\,
	{\it Approximate formulas for some functions of prime numbers.}
	Illinois J. Math. 6 (1962), 64--94.
	
	\bibitem{roth1} K. F. Roth,\,
	{\it Sur quelques ensembles d'entiers.} 
	C. R. Acad. Sci. Paris 234 (1952), 388--390.
	
	\bibitem{roth2}
	\bysame,\,
	{\it On certain sets of integers.}
	J. London Math. Soc. 28 (1953), 104--109.
	
	\bibitem{vaaler} 
	J. D. Vaaler,\,
	{\it On the metric theory of Diophantine approximation.}
	Pacific J. Math. 76 (1978), no. 2, 527--539. 
	 
	\bibitem{walfisz} 
	A. Walfisz,\,
	{\it Ein metrischer Satz über Diophantische Approximationen.}
	Fundamenta Mathematicae vol. 16 (1930), 361--385.
\end{thebibliography}
\end{document}